%% file: highly_versal.tex
\documentclass[a4paper]{amsart}

\input{amsart_header_18_05}

\usepackage[foot]{amsaddr}

\usepackage[foot]{amsaddr}

\usepackage{enumitem}

\newcommand{\clpnt}[1]{{#1}_{\mathrm{cl}}}
\newcommand{\Gr}{{\mathbf{Gr}}} 

\newcommand{\sIso}{{\mathcal{I}so}}

\newcommand{\sEnd}{{\mathcal{E}nd}}
\newcommand{\sHom}{{\mathcal{H}om}}

\newcommand{\Sch}{{\mathsf{Sch}}}
\newcommand{\Aff}{{\mathsf{Aff}}}
\newcommand{\ftAff}{{\mathsf{ftAff}}}
\newcommand{\nAff}{{\mathsf{nAff}}}
\newcommand{\qProj}{{\mathsf{qProj}}}

\newcommand{\Tors}{{\mathsf{Tors}}}
\newcommand{\VB}{{\mathsf{VB}}}

\renewcommand{\Step}[1]{\noindent {\it Step #1.}} 

\title{Highly Versal Torsors}

\author{Uriya A.\ First$^*$}
\address{$^*$University of Haifa}
\email{uriya.first@gmail.com}


\begin{document}

\maketitle

\begin{abstract}
Let $G$ be a linear algebraic group over an infinite field
$k$.
Loosely speaking, a $G$-torsor over $k$-variety is said to be versal
if it specializes to every $G$-torsor over any $k$-field. The existence
of versal torsors is well-known.
We show that there exist $G$-torsors  that admit even stronger
versality properties. For example, for every $d\in\N$,
there exists a $G$-torsor  over a smooth quasi-projective $k$-scheme
that   specializes to every
torsor over a quasi-projective   $k$-scheme after removing some codimension-$d$
closed subset from the latter. Moreover,
such specializations are abundant in a well-defined
sense. Similar results hold if we replace $k$ with an arbitrary base-scheme.
In the course of the proof we show that every globally
generated rank-$n$ vector bundle over a $d$-dimensional $k$-scheme of finite
type
can be generated by $n+d$ global sections.

When $G$ can be embedded in a group scheme of unipotent upper-triangular
matrices, we further show that there exist $G$-torsors specializing
to every $G$-torsor over any affine $k$-scheme.
We show that the converse holds when $\Char k=0$.

%%%%%%%%%%%% Change begins 
We apply our highly versal torsors  to show that, for fixed $m,n\in\N$,
the symbol length of any degree-$m$ period-$n$ Azumaya algebra  over any local $\Z[\frac{1}{n},e^{2\pi i/n}]$-ring 
is uniformly bounded. A similar statement holds in the semilocal case, but under mild restrictions on the base ring.
%
%We apply our highly versal torsors  to show that, for fixed $m,n\in\N$
%and a fixed semilocal $\Z[\frac{1}{n},e^{2\pi i/n}]$-ring $R$ with infinite residue fields, 
%the symbol length of  Azumaya algebras over $R$ having degree $m$ and period $n$
%is uniformly bounded. 
%Under mild assumptions on $R$, e.g., 
%if $R$ is local, the bound depends only on $m$ and $n$, and not on $R$.
%%%%%%%%%%%% Change ends 
\end{abstract}

\setcounter{tocdepth}{1}
\tableofcontents

\section{Introduction}

Let $k$ be an infinite field and let $G$ be a linear algebraic group over $k$,
i.e., an affine group scheme of finite type over $k$.
Informally, a $G$-torsor $E\to X$, where $X$ is a $k$-variety, is said to be \emph{versal}
if $E\to X$ specializes to every $G$-torsor over a    $k$-field.
The question of how abundant are the  specializations   leads to several variations
of this notion, e.g., see \cite{Duncan_2015_versality_of_alg_grp_actions}.
In particular, one says that:
\begin{enumerate}[label=(\roman*)]
	\item $E\to X$ is \emph{weakly versal}
	if for every   $k$-field $K$
	and $G$-torsor $E'\to \Spec K$, there is a morphism
	$f:\Spec K\to X$ such that $E'\cong \Spec K\times_{f,X} E$
	as $k$-schemes with a $G$-action;
	\item $E\to X$ is \emph{versal}
	if for every   $k$-field $K$
	and $G$-torsor $E'\to \Spec K$, there is a morphism
	$f$ as in (i), and the joint image  of all such $f$ is dense in $X$; and
	\item $E\to X$ is \emph{strongly versal}\footnote{
		This is called \emph{very versal} in \cite{Duncan_2015_versality_of_alg_grp_actions}.
	}
	if for every  $k$-field $K$
	and $G$-torsor $E'\to \Spec K$, there is a dominant $G$-equivariant
	rational map $\bbA^n_k\times_k E'\dashrightarrow E$.
\end{enumerate}
Restricting $K$ to be a finitely generated $k$-field does not affect
the validity of (i)--(iii).
The torsor $E\to X$ is also called \emph{generic} if it is versal
and $X$ is a rational $k$-variety.

The study of versal torsors has a rich history and a large range of applications.
It goes back to  the   universal bundles introduced
by topologists in the first half of the 20th century,
through   Amitsur's work on the construction of
central-simple algebras that are not crossed products   \cite{Amitsur_1972_central_division_algebras},
to the theory of essential dimension initiated by Z.~Reichstein and collaborators,
and the theory of cohomological invariants 
of algebraic groups \cite{Serre_2003_cohomological_invariants}.
We refer to reader to the introduction and appendix of \cite{Duncan_2015_versality_of_alg_grp_actions}
for a detailed historical survey.

We further note that torsors over   particular algebraic groups $G$ are sometimes equivalent to algebraic objects
of a certain type, e.g., for a     constant finite
algebraic group $\Gamma$, the $\Gamma$-torsors are equivalent to $\Gamma$-Galois extensions,
$\uPGL_n(k)$-torsors are equivalent to Azumaya algebras of degree $n$, 
etcetera. (This can be formally
stated as an equivalence between two categories fibered over $k$-schemes --- the category
of $G$-torsors and the [opposite] category of objects of the given type with specializations
as its morphisms.)
One can therefore discuss versal $\Gamma$-Galois extensions, versal Azumaya algebras --- e.g., see
\cite[p.~73]{Saltman_1999_lectures_on_div_alg} ---,
and so forth; in fact, much of the theory of versal torsors originated from 
such special cases.

\medskip

In recent years, several examples of $G$-torsors, or equivalent 
algebraic objects, which admit versality properties going beyond $k$-fields
have been exhibited. By this we mean that these $G$-torsors  specialize to every
$G$-torsor $E'\to X'$ with $X'$ ranging over a class of $k$-schemes
that goes beyond the class of  $k$-fields.
We shall informally refer to such objects as ``highly versal''.
The earliest   such example that we are aware of goes back to Saltman
\cite[Theorem~2.3]{Saltam_1978_noncrossed_products}, who constructed for every $p$-group
$G$ a $G$-Galois extension of rings specializing to any other $G$-Galois extension of 
$p$-torsion rings. 
Another such construction was given  by Fleischmann and Woodcock \cite{Fleischmann_2011_non_linear_actions_of_p_groups}. 
Saltman also constructed highly versal cyclic Galois 
extension in arbitrary and mixed characteristic (but with more restricted specialization properties)
in the recent paper \cite{Saltman_2022_mixed_char_cyclic_matters_preprint}.
Highly versal Azumaya algebras with involution 
were constructed by  Auel, the author and B.~Williams in \cite[\S5.1]{Auel_2019_Azumaya_algebras_without_inv}
(and applied in \cite{First_2022_Brauer_class_of_Az_alg_preprint}).
Taking a step further, in \cite{First_2022_generators}, the author, Z.~Reichstein and B.~Williams
showed that   for every  linear algebraic group $G$ over $k$
and every $d\in\N\cup\{0\}$,
there exists a $G$-torsor $E\to X$, where $X$ is a $k$-algebraic space
(or a $k$-scheme if $G$ is reductive),
which specializes to every $G$-torsor over a 
finite type affine $k$-scheme of dimension $\leq d$;
such torsors were called weakly $d$-versal in [op.\ cit.].
In fact, predating all of these works, a
theorem of  O.\ F\"orster \cite{Forster_1964_number_of_generators}
on the number of generators of a projective module over a noetherian (commutative) ring
can be restated as saying that
the  generic rank-$n$ projective
module generated by $(n+d)$-elements studied by Raynaud \cite{Raynaud_1965_universal_proj_module}
(equiv., its corresponding $\uGL_n(\Z)$-torsor)
is highly versal in the sense that it specializes
to any rank-$n$ projective module over a  noetherian ring of Krull dimension $\leq d$;
see Section~\ref{sec:generators}.

\medskip

The present work is a   systematic  study 
of the higher versality properties that were glimpsed on in 
the works just mentioned. 
In particular, we introduce notions of weak, ordinary and strong
higher-versality, analogous to  (i)--(iii) above,
construct examples of torsors having these properties,
and apply them to prove finiteness results on the  \emph{symbol length} of local
and semilocal rings.

While our definition  of higher versality applies
to    group schemes over  arbitrary base schemes,
for the sake of simplicity and  easier 
comparability with the classical
definitions, \emph{we give here a slightly weaker
form of the definitions,
and also restrict them to linear algebraic groups
$G$ over a field $k$}; the general definition is given in  Section~\ref{sec:definitions}.

Let $d\in\N\cup\{0\}$ and let $\catC$ be a class
of   $k$-schemes, e.g.,  $\nAff/k$ ---
the class of affine noetherian $k$-schemes ---,
or $\qProj/k$ --- the class of quasi-projective $k$-schemes. 
Given a $G$-torsor $E\to X$, with $X$ a $k$-algebraic space, we say that:
\begin{enumerate}[label=(\roman*$'$)]
	\item $E\to X$ is \emph{weakly $d$-versal} for $\catC$
	if for every $G$-torsor $E'\to X'$ with $ X'\in\catC$,
	there is an open subscheme $U\subseteq X'$ and a $k$-morphism
	$f:U\to X$
	such that $X'-U$ is of codimension $>d$ in $X'$
	and there is a $G$-torsor isomorphism $E'|_U\cong f^*E$;\footnote{
		This definition is more restrictive than 
		the weak $d$-versality discussed in \cite{First_2022_generators}.	
	}
	\item $E\to X$ is \emph{$d$-versal} for $\catC$
	if for every $G$-torsor $E'\to X'$ with $\emptyset\neq X'\in\catC$,
	there is an open subscheme $U\subseteq X'$
	and $k$-morphisms $\{f_i :U \to X\}_{i\in I}$
	such that  $X'-U$ is of codimension $>d$ in $X'$,
	the induced map $\bigsqcup_{i\in I} U\to X$
	is schematically dominant, and for every $i\in I$,
	there is a $G$-torsor isomorphism
	$E'|_U\cong f^*_iE$; and
	\item $E\to X$ is \emph{strongly $d$-versal} for $\catC$
	if for every $G$-torsor $E'\to X'$ with $\emptyset\neq X'\in\catC$,
	there is an open subscheme $U\subseteq X'$,  $n\in\N$
	and a schematically dominant 
	morphism
	$f:\bbA^n_k \times_k U\to  X$
	such that   $X'-U$ is of codimension $>d$ in $X'$
	and there is a $G$-torsor isomorphism  $\bbA^n_k\times_k  E'|_U\cong f^*E$.
\end{enumerate}
Here, we have written $E'|_U$ for the $G$-torsor $U\times_{X'} E'$
and $f^*E$ for the $G$-torsor which is the pullback of $E\to X$ along $f:U\to X$.
A $k$-morphism  $f:Y\to X$ is schematically dominant if the induced
map $\calO_X\to f_*\calO_Y$ is a monomorphism (see \cite[\S11.10]{EGA_iv_1967}
or Section~\ref{sec:schematically-dom}); when $X$ and $Y$
are $k$-varieties, this is the same as being dominant.
By convension, the dimension of the empty  scheme  is $-\infty$, so
we must have $U=X'$ if $\dim X'\leq d$. 
In particular, we must take $U=X'$ if $\catC$ consists of schemes
of dimension $\leq d$.

\begin{example}
Let $\catC$ be the class of affine irreducible
$k$-varieties.
Then a $G$-torsor is  weakly  $0$-versal for $\catC$ if and only
if it is  weakly  versal in the sense   (i) above. 
To see this, take the $k$-field $K$ from   (i) 
to be the fraction field of $X'\in \catC$ and spread out.

The same reasoning shows that a $G$-torsor which is $0$-versal  for $\catC$  is versal in the sense of (ii),
but {\it a priori}, $0$-versality for $\catC$ is more restrictive than versality.
The reason is that if $K=\mathrm{Frac}(X')$ and $\{f_i:\Spec K\to X\}_{i\in I}$ are morphisms such
that their joint image is dense in $X$, then, while we can extend each $f_i$ to some dense open $U_i\subseteq X'$,
it is not clear why we can take all the $U_i$ to be the same dense open subscheme.

Being strongly $0$-versal for $\catC$ is also {\it a priori} stronger
than being strongly versal in the sense of (iii) above,
because in (iii$'$) we insist on having a morphism
$\bbA^n_k\times_k U\to X$ for some dense open $U\subseteq X'$, rather than a rational map $\bbA^n_k\times_k X'\dashrightarrow X$.
\end{example}

Let $S$ be a scheme,  let  $G$ be a linear group scheme over $S$
that is flat and locally of   finite presentation, and let $\nAff /S$ denote
the class of affine noetherian schemes equipped with a morphism to $S$.\footnote{
	Warning: In general, this is \emph{not} the class of noetherian $S$-schemes $Y$
	whose structure morphism $Y\to S$ is affine.
}
In 
Section~\ref{sec:existence},
we give explicit constructions
of $G$-torsors
$E\to X$ which are strongly $d$-versal for $\nAff/S$, for any prescribed $d\in\N $.
For example, if $G$ is a   subgroup of $\uGL_n(S)$, then the quotient morphism
\[
\label{EQ:example-of-d-versal-torsor}
E=\uGL_{n+d}(S)/\begin{bmatrix}
1_{n\times n} & \bbA^{n\times d}_S \\
0_{d\times n} & \uGL_d(S)
\end{bmatrix}
\to
\uGL_{n+d}(S)/\begin{bmatrix}
G & \bbA^{n\times d}_S \\
0_{d\times n} & \uGL_d(S)
\end{bmatrix} =X
\]
is such a $G$-torsor (Theorem~\ref{TH:highly-versal-II}),
which also enjoys the following \emph{extension property}
(Theorem~\ref{TH:extension}):
If $E'\to X'$ is a $G$-torsor with $X'$ affine of dimension $\leq d$,
$X''$ is a closed subscheme of $X'$
and $E''\to X''$ is the pullback of $E'\to X'$
along $X''\to X'$, then every specialization of $E $
to $E'' $ extends to a specialization of $E $
to $E' $.
 
In general, the base $X$ in our constructions is an $S$-algebraic space. However, if $G$
is reductive  or finite and $S$ is noetherian, then $X$ is an $S$-scheme. 
In fact, in the reductive and finite cases,
we can    even make $X$ affine at the cost of replacing $ \nAff/S$ with its subclass of schemes of dimension $\leq d$
(Theorem~\ref{TH:highly-versal-I}).
When $S$ is the spectrum of an infinite field $k$, things become even better in the sense
that  there exist $G$-torsors over
quasi-projective $k$-schemes which are strongly $d$-versal for the class of
quasi-projective $k$-schemes (Theorem~\ref{TH:highly-versal-III}).

\begin{remark}
	Our Theorem~\ref{TH:highly-versal-II} implies  in particular  that for every linear algebraic group
	over a  field $k$ (possibly finite), and every $d\in\N$, there exists a finite-type smooth $k$-scheme $X$
	and a $G$-torsor $E\to X$ which specializes to every $G$-torsor
	$E'\to X'$ such that $X'$ is an affine $k$-variety of dimension $\leq d$.
	However, when $d>0$, the promised specialization $X'\to X$ is generally
	far from being an immersion. 
	Indeed, we can choose $X'$ to be a curve admitting
a closed point $x'$ such that $\dim T_{x'} X'>\dim T_x X$ for every $x\in X$,
which means that there is no immersion $X'\to X$. Also, if 
$k$ is finite, then  we can arrange that $|X'(k)|>|X(k)|$,
again making it impossible to define an immersion of $X'$ in $X$. 
With this in mind, the   existence of weakly $d$-versal $G$-torsors
whose base scheme is of finite type of $k$ may seem  surprising.
Informally, it also suggests that there should be limitations on the 
``complexity'' of $G$-torsors
over $d$-dimensional   affine $k$-schemes.
\end{remark}

We prove the existence of strongly   $d$-versal  
$G$-torsors  by reducing it to the existence
of  \emph{weakly} $d$-versal rank-$n$ vector bundles 
which admit some symmetry properties (Corollary~\ref{CR:symmetric-weak-versal-is-strong}). 
We then show that some constructions
of ``generic'' rank-$n$
vector bundles generated by $d+n$ global sections,
e.g., those 
considered in \cite{Raynaud_1965_universal_proj_module}
and (implicitly) in \cite{First_2022_generators},
satisfy the required conditions. The heart of the argument
is a theorem of O.\ F\"orster   \cite{Forster_1964_number_of_generators}
stating that a rank-$n$ vector bundle over a $d$-dimensional
noetherian affine scheme is generated by $n+d$ global sections, and its
improvements in \cite{Swan_1967_num_of_generators_of_module}, 
\cite{First_2017_number_of_generators}, \cite{First_2022_generators},
which we need to strengthen even further
for our purposes (Theorems~\ref{TH:fr-vec-bundles}, \ref{TH:fr-variant}, \ref{TH:generators-bound-any-scheme}). For example, letting $k$ denote an infinite field,
we show that every 
globally generated rank-$n$ vector bundle over a finite type $k$-scheme 
is generated by $d+n$ global sections away from some closed subscheme of codimension $>d$.

\medskip

Returning to group schemes over a general base $S$,
we also consider the question
of whether there exist
$G$-torsors   over finite type $S$-schemes which are $\infty$-versal
for $ \nAff/S$.
We show that the answer is ``yes''   if $G$ can be embedded in a group 
of upper-triangular unipotent matrices over $S$ (Theorem~\ref{TH:infinity-versal}),
and that the converse holds when $S$ is the spectrum of a field
of characteristic $0$ (Theorem~\ref{TH:converse}).
This gives a conceptual explanation of why, for a field $k$ of characteristic $p>0$,
every finite $p$-group $G$ admits a $G$-Galois extension of $k$-rings which specializes
to any  other $G$-Galois extension of $k$-rings \cite{Saltam_1978_noncrossed_products},
while such $G$-Galois extensions do not exist   when $\Char k=0$ and $G$ is nontrivial.

\medskip

Finally, we apply one of our constructions of highly versal torsors
together with Hoobler's extension of the Merkurjev--Suslin Theorem
to semilocal rings \cite{Hoobler_2006_Merkurjev_Suslin_over_semilocal}
to prove finiteness results about the symbol length of semilocal (commutative) rings.
In more detail, let $n,m\in\N$,
and let $R$ be a ring that is a $\Z[\frac{1}{n},e^{2\pi i/n}]$-algebra.
Recall that the $(n,m)$-symbol length of $R$ is the minimal $L\in\N\cup \{0,\infty\}$ 
such that every degree-$m$ period-$n$ Azumaya algebra over $R$
is Brauer equivalent to the tensor product of at most $L$ degree-$n$ \emph{symbol Azumaya algebras} over $R$.
%%%%%%%%%%%% Change begins 
We show in Theorem~\ref{TH:symbol-length-II} that for every $n,m,s\in\N$,
there is   $L\in \N$ such that
for every  semilocal $\Z[\frac{1}{n},e^{2\pi i/n}]$-ring $R$ having at most $s$ maximal ideals,
the  $(n,m)$-symbol length of $R$ is at most $L$.
(In particular, the $(n,m)$-symbol length of 
any semilocal $\Z[\frac{1}{n},e^{2\pi i/n}]$-ring  is finite.)
A similar uniform upper bound on the symbol length also exists
if, instead of bounding the number of maximal ideals of the semilocal ring $R$,
we assume that $R$ has infinite residue fields
and is an algebra over some fixed semilocal ring $R_0$; see Theorem~\ref{TH:symbol-length-I}.
Our results imply in particular that
the $(n,m)$-symbol length of all fields (regardless of
the characteristic) is uniformly bounded ---
a result known to experts.
%
%We show in Theorem~\ref{TH:symbol-length-I} that for every $n,m\in\N$
%and every semilocal $\Z[\frac{1}{n},e^{2\pi i/n}]$-ring $R$, there is $L\in \N$
%such that the $(n,m)$-symbol length of every semilocal $R$-ring $R_1$
%having infinite residue fields is at most $L$. (In particular, the $(n,m)$-symbol length of 
%any semilocal $\Z[\frac{1}{n},e^{2\pi i/n}]$-ring with infinite residue fields is finite.)
%This is well-known when $R$ and $R_1$ are fields, and is classically
%proven   using versal division algebras; 
%our proof is similar in flavor except it makes use of    torsors which are \emph{strongly}
%versal for (spectra  of) semilocal rings.
%We further show in Theorem~\ref{TH:symbol-length-II}
%that there is a uniform bound on the $(n,m)$-symbol length of all local rings
%with an infinite residue field. In particular, it follows that
%the $(n,m)$-symbol length of all fields (regardless of
%the characteristic) is uniformly bounded ---
%a result  known to experts.
%%%%%%%%%%%% Change ends 

\medskip

The structure of the paper is as follows:
Sections~\ref{sec:torsors} and~\ref{sec:vec-bundles}
are preliminary and collect facts about torsors and vector bundles.
In Section~\ref{sec:schematically-dom} we establish
some facts about schematically dominant morphisms.
In Section~\ref{sec:definitions} we introduce the definition
of weakly, ordinary and strongly $d$-versal torsors and study some of
their properties. In particular, we reduce
the existence of \emph{strongly} $d$-versal torsors for linear group schemes
to the existence of \emph{weakly} $d$-versal vector bundles 
admitting some symmetry properties.
The subject matter of Sections~\ref{sec:generic}
and~\ref{sec:generators}
is to show existence of such $d$-versal vector bundles;
their construction is given in Section~\ref{sec:generic}
and their desired properties are established in Section~\ref{sec:generators}.
This is used in Section~\ref{sec:existence}
to prove the existence of torsors that are strongly $d$-versal for some classes of schemes.
We apply this to particular group schemes
in  Section~\ref{sec:examples} to deduce the existence of strongly
$d$-versal Galois extensions, Azumaya algebras and so on.
Section~\ref{sec:infty-versal} concerns with $\infty$-versal torsors.
Finally, in Section~\ref{sec:symbol},
we apply some of our earlier results to bounding the symbol length of Azumaya
algebras over semilocal rings.

\subsection*{Acknowledgments} 

We are grateful
to A.\ Chapman,  Z.\ Reichstein, 
D.\ Saltman and B.\ Williams for comments and suggestions regarding this work.
We also thank A.\ Duncan for an inspiring conversation.

\subsection*{Notation}

Throughout, a ring means a commutative (unital,
associative) ring. Algebras are not necessarily commutative.

\medskip

As is common in algebraic geometry, the dimension of a topological space $X$
is the supremum of the set of all $d\in \N\cup\{0 \}$ for which
there is 
chain of closed irreducible   subsets
$X_0\subsetneq \dots\subsetneq X_d$. By convention, the supremum of the empty
set is $-\infty$, and so the dimension of the empty space is $-\infty$.
The space $X$ is noetherian if every ascending chain of open subsets eventually stabilizes.
Recall that if $Y$ is a subspace of $X$, then $\dim Y\leq \dim X$
and $Y$ is noetherian if $X$ is. 
The set of closed points of $X$ is denoted $\clpnt{X}$ and given the subspace topology.
When $X=\Spec R$ for a ring $R$, the subspace $\clpnt{X}$ is just the subspace
of maximal ideals of $R$, denoted $\Max R$.

If $X$ is irreducible and $Y$ is a closed
irreducible subset of $X$, then  the codimension of $Y$
in $X$ --- denoted $\codim(Y,X)$ --- is defined to be the supremum
of the set of $d\in \N\cup\{0\}$  for which
there is 
chain of closed irreducible   subsets
$Y=Y_0\subsetneq \dots\subsetneq Y_d\subseteq X$.
We extend this definition to general $X$ and $Y$
by letting $\codim(Y,X)$ be the infimum of   $\codim(Y',X')$
when $X'$ ranges over the irreducible components of $X$
and $Y'$ ranges over the irreducible components of $Y\cap X'$.\footnote{
	Caution: This definition   differs from the definition in
	\cite[Tag \href{https://stacks.math.columbia.edu/tag/02I3}{02I3}]{stacks_project}
	even when $Y$ is irreducible.  
}
When $X$ is catenary (see \cite[Tag \href{https://stacks.math.columbia.edu/tag/02I1}{02I1}]{stacks_project}; e.g., a finite type scheme over a field), 
$\codim(Y,X)\geq d$ if and only if
$\dim X'\geq \dim (X'\cap Y)+d$ for every irreducible component $X'$ of $Y$.\footnote{
	The catenarity assumption is needed to guarantee that for every
	pair of irreducible closed subsets $Y'\subseteq X'$ of $X$,
	we have  $\codim(Y',X')= \dim X'-\dim Y'$  and   the three quantities used in the equality
	are finite.
}
By convention, the infimum of the empty set is $\infty$, and so 
$\codim(\emptyset, X)=\infty$ (even if $X=\emptyset$).

\medskip

Let $X$ be a scheme.
The residue field of $X$ at $x\in X$ is denoted $k(x)$. If $\calM$ is an $\calO_X$-module,
we write $\calM_x$ for the stalk of $\calM$ at $x$ and $\calM(x)=\calM_x\otimes_{\calO_{X,x}}k(x)$.
Given an open neighborhood $U$ of $x$ and $m\in \Gamma(U,\calM)$, we also write
$m(x)$ for the image of $m$ in $\calM(x)$.
As usual,  the rank of $\calM$
is the function $\rank \calM :X\to \textsf{cardinals}$ 
sending $x\in X$
to $\dim_{k(x)}\calM(x)$.
A vector bundle over $X$ is a locally free $\calO_X$-module of (point-wise) finite rank.
A line bundle is a vector bundle of rank $1$.

If not otherwise said, irreducible components of $X$ are given the reduced induced
subscheme structure.

\medskip

Let $S$ be a base scheme. Given
$S$-schemes $X$ and $S_1$, 
we write $X_{S_1}$ for $X\times_SS_1$  viewed as an $S_1$-scheme.
If $f:X\to Y$ is a  morphism of $S$-schemes, then $f_{S_1}:X_{S_1}\to Y_{S_1}$
denotes the pullback of $f$ along the structure morphism $S_1\to S$.
Given a class   $\catC$   of $S$-schemes, we write $\catC/S_1$
for the class of $S_1$-schemes with underlying $S$-scheme in $\catC$.

Recall that a collection of $S$-morphisms $\{f_i:Y_i\to X\}_{i\in I}$
is  an fppf covering if each $f_i$ is flat and locally of finite presentation (abbreviated: fppf)
and $X=\bigcup_{i\in I}f_i(Y_i)$ as sets.
The category $\Sch/S$ of   $S$-schemes together with the Grothendieck topology generated
by fppf coverings is
the \emph{large fppf site} of $S$, denoted $(\Sch/S)_{\fppf}$.  
We will freely identify  schemes and algebraic spaces over $S$
with the sheaves that they
represent on the site $(\Sch/S)_{\fppf}$. 
In particular,
given an $S$-algebraic space $X$ and an $S$-scheme $U$, we shall write $X(U)$ for the set of $S$-morphisms
from $U$ to $X$ and call this set the $U$-sections of $X$ (over $S$). 
Morphisms between $S$-algebraic spaces
will usually be defined by specifying their action on $U$-sections
for every $S$-scheme $U$.

\medskip

An $S$-group means a group scheme over $S$. An algebraic group over a field
$k$ is a $k$-group of finite type.
Let $G$  be an $S$-group.
A subgroup of $G$ is a subgroup $S$-scheme of $G$. 
If $G$ acts from the right on an $S$-algebraic space $X$, we let $X/G$
denote the quotient of $X$ by $G$ in category of sheaves over $(\Sch/S)_\fppf$,
i.e., $X/G$ is the fppf-sheafification of the functor $U\mapsto X(U)/G(U)$
from $S$-schemes to sets. The quotient $X/G$ is not always representable by an 
$S$-scheme, but we will treat it as an $S$-scheme when it is. Likewise for $S$-algebraic spaces.

An $S$-group $G$ with structure morphism $p:G\to S$
is said to be \emph{finite and locally free} if $p$ is finite
and $p_*\calO_G$ is a locally free $\calO_S$-module.
 
We write $\uGL_n(S)$ for the $S$-group representing
the sheaf of invertible $n\times n$
matrices over $S$, which is sometimes denoted $\uGL_n(\calO_S)$;
for an $S$-scheme $U$, we have $\uGL_n(S)(U)=\nGL{\Gamma(U,\calO_U)}{n}$.
By default,  sections of $\calO_S^n$ are viewed as column vectors,
so that   $\uGL_n(S)$ acts on $\calO_S^n$ from the left by  matrix multiplication.
The  subgroup of $\uGL_n(S)$
consisting of  matrices of determinant $1$ is denoted $\uSL_n(S)$.
When $S=\Spec R$ for a ring $R$, we  write $\uGL_n(R)$ for $\uGL_n(S)$ 
and $\uSL_n(R)$ for $\uSL_n(S)$.

An $S$-group $G$
is said to be \emph{linear} if there is an $S$-group monomorphism $G\to \uGL_{n }(S)$
for some $n\in\N$.\footnote{
	Caution: Some texts call an $S$-group linear if there is 
	an $S$-vector bundle $V$  and a monomorphism $G\to \sEnd_{\calO_S}(V)$.
	This is  equivalent to our definition when $S$ is affine.
} It is \emph{strictly linear} if 
there is an $S$-group morphism $G\to \uGL_n(S)$ that is also a closed immersion.

\medskip

We write $\mathbb{A}^{n\times m}_S$ for the affine space
of $n\times m$ matrices over $S$, which is just $\mathbb{A}^{nm}_S$
but with the coordinates reorganized into an $n\times m$ matrix.
Given $n,d\in\N\cup\{0\}$
and $S$-schemes $A,B,C,D$    equipped with monomorphisms $A\to \mathbb{A}^{n\times n}_S$,
$B\to \mathbb{A}^{n\times d}_S$, $C\to \mathbb{A}^{d\times n}_S$, $D\to \mathbb{A}^{d\times d}_S$,
we write $[\begin{smallmatrix} A & B \\ C & D\end{smallmatrix}]$
for the scheme $A\times B\times C\times D$
endowed with the monomorphism $(a,b,c,d)\mapsto
[\begin{smallmatrix} a & b \\ c & d\end{smallmatrix}]$ into
$\bbA^{(n+d)\times (n+d)}_S$.
When $A=\Spec S$ and $A\to \mathbb{A}^{n\times n}_S$ is the zero matrix (resp.\ unit matrix) section,
we shall write $[\begin{smallmatrix} 0_{n\times n} & B \\ C & D\end{smallmatrix}]$
(resp.\ $[\begin{smallmatrix} 1_{n\times n} & B \\ C & D\end{smallmatrix}]$) instead of 
$[\begin{smallmatrix} A & B \\ C & D\end{smallmatrix}]$, and likewise
with $B$, $C$, $D$. The subscript $n\times n$ will be dropped from $0_{n\times n}$
when it is clear from the context. (For example, such notation was used on page
\pageref{EQ:example-of-d-versal-torsor}.)

\section{Preliminaries on Torsors}
\label{sec:torsors}

Let $S$ be a scheme. Throughout this section, all schemes and algebraic spaces
are over $S$, and morphisms are $S$-morphisms.  

Let $G$ be an $S$-group and let $X$ be an $S$-algebraic space.
For our purposes, a $G$-torsor
over $X$ consists of a \emph{sheaf}
$E$ on $(\Sch/S)_{\fppf}$
together with a 
right $G$-action $E\times_S G\to E$ and a morphism $p:E\to X$
such that 
\begin{enumerate}[label=(\arabic*)]
	\item $p(e\cdot g)=p(e)$ for every $S$-scheme $U$, $g\in G(U)$ and $e\in E(U)$,
	and 
	\item there is an fppf covering $X'\to X$  such that $X'\times_X E\cong X'\times_S G$ 
	as sheaves with a right $G$-action. 
\end{enumerate} 
(Otherwise said, a $G$-torsor over $X$ is a right sheaf $G$-torsor for the fppf
topology.) When these conditions hold we will simply say that $p:E\to X$, or just $E$,
is a $G$-torsor over the base $X$. In this case, the induced
sheaf morphism $[e]\mapsto p(e):E/G\to X$ is an isomorphism. 
We sidestep the question of which $G$-torsors are representable by $S$-schemes 
or $S$-algebraic spaces and satisfy
with noting that  if $G\to S$ is affine and $X$ is a scheme, then all $G$-torsors are schemes
\cite[III, Theorem~4.3]{Milne_1980_etale_cohomology}.

A morphism from a
$G$-torsor $p_1:E_1\to X_1$
to another $G$-torsor
$p_2:E_2\to X_2$
is a pair $(\hat{f},f)$
consisting of a $G$-equivariant $S$-morphism $\hat{f}:E_1\to E_2$
and an $S$-morphism $f:X_1\to X_2$ such that $p_2\circ \hat{f}=f\circ p_1$.
Since $X_i\cong E_i/G$, every $G$-equivariant
morphism $\hat{f}:E_1\to E_2$
determines and is determined by a morphism from $E_1\to X_1$ to $E_2\to X_2$.
The collection of $G$-torsors
and their morphisms form a category denoted  $\Tors(G)$. 
The subcategory of $\Tors(G)$ consisting of $G$-torsors over a fixed
$S$-algebraic space $X$ and morphisms of the form $(\hat{f},\id_X)$
is denoted $\Tors(G,X)$ and called the category of $G$-torsors over $X$. 
We write $E_1\cong E_2$ when  $E_1\to X_1$ 
and $E_2\to X_2$ are isomorphic in $\Tors(G)$. Isomorphism
in the subcategory $\Tors(G,X)$ is denoted by $\cong_X$.

Let $p:E\to X$ be a $G$-torsor and let $f:Y\to X$ be an $S$-morphism.
Then the pullback $p_f: Y\times_X E\to Y$ of $p$ along $f$ is a $G$-torsor
($G$ acts on $Y\times_X E$ via   $E$).
We often abbreviate $Y \times_X E$ to $f^*E$. If $f:U\to X$ is an open immersion,
we also write $E|_U$ instead of $f^*E$. 

\begin{remark}\label{RM:morphism-pullback-correspondence}
Let $p:E \to X$ and $p':E'\to X'$ be $G$-torsors.
There is a one-to-one correspondence between
morphisms $(\hat{f},f)$ from $E'\to X'$ to $E\to X$
and pairs $(\hat{u},f)$ consisting of an $S$-morphism $f:X'\to X$
and an isomorphism $ \hat{u}=(\hat{u},\id_X) : f^*E\to E'$ in $\Tors(G,X)$;
the pairs $(\hat{u},f)$ are   called \emph{specializations}
of $E$ to $E'$.
Explicitly, given $(\hat{f},f)$, define $\hat{u}$ to
be the inverse of $e'\mapsto (p'(e'),\hat{f}(e')):E'\to f^*X$,
and given $(\hat{u},f)$,
let $\hat{f}$ be the composition of $\hat{u}^{-1}$ and the second projection
$f^*X=X'\times_X E\to E$.
\end{remark}

Let $\vphi:H\to G$ be a 
morphism of $S$-groups,
and let $p:E\to X$ be an $H$-torsor. 
Then $p$ and $\vphi$
give  rise to a $G$-torsor $p^G:E\times^{\vphi,H} G\to X$. Here, $E\times^{\vphi,H} G$ denotes the quotient of the sheaf $E\times G$
by the equivalence relation $E\times H\times G\to (E\times G)\times (E\times G)$
given by $(e,h,g)\mapsto((e h,g),(e,\vphi(h) g))$ on sections,
and $p^G$ is given by $p^G([e,g])=p(e)$ on sections.
There is a natural $H$-equivariant morphism $E\to E\times^{\vphi,H} G$
given by $e\mapsto [e,1_G]$ on sections. When $\vphi$
is clear from the context, we will write $E\times^H G$
instead of $E\times^{\vphi,H}G$.

A left action of an $S$-group $H$ on a
$G$-torsor $p:E\to X$ consists of left actions of $H$ on $E$
and $X$ such that $p$ is $H$-equivariant and the actions of $H$ and $G$ on $E$
commute.

\medskip

Let $G$ be an   $S$-group acting from the right on  an $S$-algebraic space $E$.
We say that $G$ acts freely on $E$ if $G(U)$ acts freely on $E(U)$ for every $S$-scheme $U$.
In general, the sheaf $E/G$ on $(\Sch/S)_{\fppf}$
may not be an $S$-scheme even when $E$ is, but under mild assumptions
on $G$, it is an $S$-algebraic space and $E\to E/G$ is a $G$-torsor. 
(In fact, this is the   reason why
we need to consider algebraic spaces in this work.)

\begin{lem}\label{LM:free-action-quotient}
Let $G$ be a faithfully flat locally of finite presentation $S$-group 
acting freely from the right on an $S$-algebraic space $E$.
Then the quotient sheaf $E/G$ is an $S$-algebraic space and $E\to E/G$ is a $G$-torsor.
\end{lem}
 
\begin{proof}
Consider the groupoid object $(E,R,s,t,c)$ in $S$-algebraic spaces 
given by 
$R(U)=\{(x,xg)\where x\in E(U), g\in G(U)\}$ (note that $R\cong E\times G$
because $G$ acts freely on $E$), 
$s_U,t_U:R(U)\to E(U)$ being
the first and second projections and $c_U=[x\mapsto (x,x)]:E(U)\to R(U)$.
Since $G$ acts freely on $E$, both $s$ and $t$ are isomorphic to the first projection $E\times_S G
\to E$,
so $s$ and $t$   are flat and locally of finite
presentation (because $G\to S$ is).
Thus, the stack $[E/R]$ is algebraic 
\cite[Tag \href{https://stacks.math.columbia.edu/tag/06FI}{06FI}]{stacks_project}.
The freeness of the action further implies
that   $[E/R]$ is an $S$-algebraic space
\cite[Tag  
\href{https://stacks.math.columbia.edu/tag/04SZ}{04SZ}]{stacks_project},
which   must coincide with the quotient sheaf $E/G$. 
By \cite[Tag  
\href{https://stacks.math.columbia.edu/tag/06FH}{06FH}]{stacks_project},
$E\to E/G$ is an fppf covering. Since the morphism
$(e,g)\mapsto (e,eg):E\times_S G \to E\times_{E/G} E $ 
is a $G$-equivariant isomorphism, it follows that $E\to E/G$ is a $G$-torsor.
\end{proof}

The question of whether $E/G$ is a scheme when $G$ acts freely on $E$
is difficult in general, even if we impose
the  necessary condition that $(g,e)\mapsto (ge,e):G\times_S E\to E\times_S E$ is an immersion.
We refer the reader to   Seshadri \cite{Seshardri_1977_geometric_reductivity}, 
Thomason \cite[\S3]{Thomason_1987_equivariant_resolution_linearlization} and related works
for an extensive discussion of the case where $G$ is reductive\footnote{
	An $S$-group $G$ is called reductive if $G\to S$ is smooth
	and the geometric fibers of $G\to S$ are \emph{connected} reductive algebraic groups.
} 
(see also  Anantharaman \cite[Appendix, Theorem 6]{Anantharaman_1973_group_schemes_homogeneous_spaces}
for the relation between the geometric quotients of \cite{Seshardri_1977_geometric_reductivity} and 
the sheaf quotients we consider here).
In this work, we will only be interested in the case where $E$ is a strictly
linear $S$-group and $G$
is a   subgroup of $E$, where we have the following positive result.

\begin{thm}\label{TH:quo-of-subgroups}
	Let $S$ be a locally noetherian scheme,
	let $G$ be a closed subgroup of $\uGL_n(S)$
	and let $H$ be a subgroup of $G$.
	Suppose that there is an exact sequence  
	$1\to H'\to H\to H''\to 1$ of group 
	sheaves on $(\Sch/S)_{\fppf}$ with $H'$ a reductive $S$-group  
	and $H''$ a finite and locally free  $S$-group.
	Then $G/H$ is (represented by) an  $S$-scheme of finite type
	and $G/H\to S$ is affine.
	If $H'=1$, then the assumption that $S$ is locally noetherian can be dropped.
\end{thm}

\begin{proof}
	It is enough to show this after passing to an affine covering of $S$,
	so we may assume that $S$ is affine and noetherian.

	We first show that $G/H'$ is an affine $S$-scheme
	and $G/H'\to S$ is      of finite type.
	Since $H'$ is reductive, it is closed in $G$ \cite[XVI, Corollaire~1.5]{SGAiii}.
	By embedding $\uGL_n(S)$ in $\uSL_{n+1}(S)$ via $x\mapsto [\begin{smallmatrix}
	x & 0 \\ 0 & (\det x)^{-1}\end{smallmatrix}]$,
	we may assume that $G$ is closed inside the affine space $V=\mathbb{A}^{n\times n}_S$. 
	Observe that $H'$ acts linearly on $V$ because $G$ does.
	Write $G=\Spec B$ and define the fixed ring $B^{H'}$ as in \cite[p.~229]{Seshardri_1977_geometric_reductivity}.
	(Briefly, writing $S=\Spec A$, the ring $B^{H'}$ consists of those $b\in B$ whose image in $B\otimes_A A'$
	is fixed under  $H'(A')$ for any $A$-ring $A'$.)
	Since $H'$ is closed in $G$, 
	for every algebraically closed field $k$ over $S$, the orbits of the right action of $H'_k$
	on $G_k$ are closed.
	Now,  \cite[Theorem~3]{Seshardri_1977_geometric_reductivity} tells
	us that $\Spec B^{H'}$ is a \emph{geometric quotient} of $G$ by $H'$.
	Consequently,
	$\Spec B^{H'}$ is the quotient of $\Spec B$ by $G$ in the category of ringed spaces.
	The fact that $H'$ is closed in $G$
	also tells us that the map
	$(h,g)\mapsto (gh,g):H'\times G\to G\times G$ is a closed immersion.
	Since $H'$ is flat of finite presentation over $S$ (because it is reductive), we may
	apply   \cite[Appendix I, Theorem 6]{Anantharaman_1973_group_schemes_homogeneous_spaces},
	which says that $\Spec B^{H'}$ represents the sheaf $G/H'$ on $(\Sch/S)_{\fppf}$.
	This means that $G/H'$ is an affine scheme. By Lemma~\ref{LM:free-action-quotient}, 
	$G\to G/H'$ is also an fppf covering.
	Since  $G\to S$ is of finite type, this means that $G/H'\to S$ is also of finite type 
	\cite[Tag \href{https://stacks.math.columbia.edu/tag/0367}{0367}]{stacks_project}.
	
	Next, observe that $H''$ acts freely on the affine scheme $X:=G/H'$, and $X/H''=G/H$.
	Thus, by 
	\cite[Tag \href{https://stacks.math.columbia.edu/tag/07S7}{07S7}]{stacks_project}
	and	
	\cite[Theorem~B.18]{Milne_2017_algebraic_groups}	
	(for instance; see also p.~604 in [op.\ cit.]), $X/H''$ is an affine scheme and $X\to X/H''$ is an $H''$-torsor.
	As in the previous paragraph, $G/H=X/H''\to S$ is also of finite type.
\end{proof}

As for the question of which affine $S$-groups $G$ can be realized
as closed subgroups of $\uGL_n(S)$, we have:

\begin{thm}[{Thomason \cite[Corollary~3.2]{Thomason_1987_equivariant_resolution_linearlization}}]
	\label{TH:Thomason}
	Let $S$ be an affine noetherian scheme, and let $G$ be an affine $S$-group
	of finite type. Then there is $n\in\N$ such that $G$ is isomorphic to a closed
	subgroup of $\uGL_n(S)$ if one of the following hold:
	\begin{enumerate}[label=(\arabic*)]
		\item $S$ is regular and $\dim S\leq 1$ (e.g., if $S=\Spec \Z$),
		\item $S$ is regular, $\dim S\leq 2$ and $G\to S$ is smooth with connected fibers,
		\item $G$ is semisimple, 
		\item $S$ is normal and $G$ is reductive. 
	\end{enumerate}
\end{thm}

See   \cite{Gille_2022_when_is_reductive_grp_sch_linear} for more information about the reductive case.

When $S$ is the spectrum of a field, things simplify significantly, and we have the following well-known
stronger versions of Theorems~\ref{TH:quo-of-subgroups} and~\ref{TH:Thomason}.

\begin{thm}\label{TH:quotient-of-groups-over-a-field}
	Let $k$ be a field, let $G$ be an algebraic group over $k$
	and let $H$ be an algebraic   subgroup of $G$. Then:
	\begin{enumerate}[label=(\roman*)]
		\item $G/H$ is a finite type $k$-scheme. When $G$ is affine,
		$G/H$ is quasi-projective over $k$,
		and if
		the neutral connected component of $H$ is moreover reductive, then $G/H$ is   affine.
		\item If $G$ is affine, then  it is  strictly linear.
	\end{enumerate}
\end{thm}

\begin{proof}
	See \cite[Theorems~5.28]{Milne_2017_algebraic_groups}
	(for instance) for the first assertion of (i),
	[op.\ cit., \S7e, proof of Theorem~7.18]  for the second assertion of (i),
	and [op.\ cit., Corollaries~4.10, 3.35] 
	for (ii). The last assertion of (i) is part of Matsushima's Theorem,
	see \cite{Richardson_1977_affine_coset_spaces}, for instance.
\end{proof}

\section{Preliminaries on Vector Bundles}
\label{sec:vec-bundles}

Let $X$ be a scheme. Recall that a vector bundle $V$ on $X$
is a locally free $\calO_X$-module of finite rank.
A \emph{basis} for $V$ is a set of global sections $\{v_1,\dots,v_n\}\subseteq \Gamma(X,V)$
such that the map $(\alpha_1,\dots,\alpha_n)\mapsto \sum_i \alpha_i v_i:\calO_X^n\to V$
is an isomorphism.

As usual, if $f:X'\to X$ is a morphism, then we let
$f^*V$ denote $f^{-1}V\otimes_{f^{-1}\calO_{X}}\calO_{X'}$,
which is a vector bundle over $X'$.
We write $V_{X'}$ instead of $f^*V$ when $f$ is clear from the context.
When $f$ is an open immersion, we also write $V|_{X'}$ for $f^*V$.
The sheaf $V$ extends to a representable sheaf on $(\Sch/X)_{\fppf}$
by setting $V(X')=\Gamma(X,V_{X'})$ for every $X$-scheme $X'$,
and we shall occasionally use $V$ denote the scheme representing this sheaf.
From this point of view, $V_{X'}$ is just $V\times_X X'$ as schemes over $X'$.

The dual of a vector bundle $V$ over $X$ is the $\calO_X$-sheaf $V^\vee:=\sHom_{\calO_X}(V,\calO_X)$
on  $(\Sch/X)_{\fppf}$ sending an $X$-scheme $X'$ to $\Hom_{\calO_{X'}}(V_{X'},\calO_{X'})$.

Fix $n\in\N$. We let $\VB_n(X)$ denote the category of rank-$n$
vector bundles over $X$ with $\calO_X$-module isomorphisms as morphisms.
It is well-known that $\VB_n(X)$ is equivalent to $\Tors(\uGL_n(X),X)$, the category of $\uGL_n(X)$-torsors
over $X$. The equivalence is given
by sending a rank-$n$ vector bundle $V$ to $\sIso(\calO_X^n,V)$, the sheaf
mapping an $X$-scheme $X'$ to the set of $\calO_{X'}$-isomorphisms from $\calO_{X'}^n$ to $V_{X'}$.
Since $\uGL_n(X)=\sIso(\calO_X^n,\calO_X^n)$, composition
defines a right action of $\uGL_n(X)$ on $\sIso(\calO_X^n,V)$, making it into a torsor over this $X$-group.
The equivalence $\VB_n(X)\to \Tors(\uGL_n(X),X)$ is compatible with pullbacks. That is,
if $f:X'\to X$ is a morphism, $V$ is a vector bundle over $X$ and $E$ is its corresponding $\GL_n(X)$-torsor,
then there is a natural isomorphism between $f^*E$ and the $\uGL_n(X')$-torsor corresponding to $f^*V$.

Suppose now that $S$ is a scheme and $X$ is an $S$-scheme.
Then $\Tors(\uGL_n(X),X)$ is equivalent to $\Tors(\uGL_n(S),X)$.
The following lemma describes the functor of points of the $\uGL_n(S)$-torsor (viewed as an $S$-scheme)
corresponding to  a rank-$n$ vector bundle over $X$.

\begin{lem}\label{LM:GLn-torsor-of-bundle}
	Let $S$ be a scheme, let $X$ be an $S$-scheme,
	let $V$ be a rank-$n$ vector bundle over $X$,
	and let $p:E\to X$ be the corresponding $\uGL_n(S)$-torsor.
	Then the functor of points of $E$
	is isomorphic to the functor sending an $S$-scheme $Y$
	to the set of 
	tuples $(f,v_1,\dots,v_n)$ consisting of 
	a $S$-morphism $f:Y\to X$
	and an ordered basis $v_1,\dots,v_n$ for $f^*V$.	
	Under this isomorphism,
	the (right) action of $\uGL_n(S)$ on $E$  
	is given by
	$(f,v_1,\dots,v_n)\cdot (g_{ij})_{i,j}=(f,\sum_{i=1}^n g_{i1} v_i,\dots,\sum_{i=1}^n g_{in} v_i)$
	on sections. 
\end{lem}

\begin{proof}
	We just noted that the sheaf that $E$
	represents on $(\Sch/X)_{\fppf}$
	is  
	$\sIso(\calO_X^n,E)$ 
	and the  right action of  $\uGL_n(X)=\sIso(\calO_X^n,\calO_X^n)$ is   given
	section-wise by composition.
	Let $e_1,\dots,e_n$
	denote the standard basis of $\calO_X^n$.
	Then, for every $X$-scheme $Y$,
	the map  $\psi\mapsto (\psi(e_1|_Y),\dots,\psi(e_n|_Y))$
	defines a natural bijection between $E(Y)=\sIso(\calO_Y^n,V_Y)$ ($E$ is viewed as an $X$-scheme)
	and the set of ordered bases of $V_Y$.
	
	Now let $Y$ be an $S$-scheme. The data of an $S$-morphism $\tilde{f}:Y\to E$
	is equivalent to the data of a morphism $f:Y\to X$ and, viewing
	$Y$ as an $X$-scheme via $f$, an $X$-morphism
	$f':Y\to E$; the correspondence is given by $\tilde{f}\mapsto (f,f'):= (p\circ \tilde{f},\tilde{f})$.
	Combining this with the previous paragraph gives the lemma.
\end{proof}

\begin{lem}\label{LM:torsor-of-dual}
	Let $S$ be a scheme, let $X$ be an $S$-scheme,
	let $V$ be a rank-$n$ vector bundle over $X$, 
	and  let $E$ be its corresponding $\uGL_n(S)$-torsor.
	Then, up to isomorphism in $\Tors(\uGL_n(S),X)$, the $\uGL_n(S)$-torsor corresponding to 
	the dual bundle $V^\vee$
	is $E$ endowed with the right $\uGL_n(S)$-action given by
	$(x,g)\mapsto x(g^{\trans})^{-1}$ on sections.
\end{lem}

\begin{proof}
	Let $Y$ be an $S$-scheme.
	By Lemma~\ref{LM:GLn-torsor-of-bundle}, 
	$E(Y)$ is naturally isomorphic to  the set of tuples $(f,v_1,\dots,v_n)$ where  $f:Y\to X$
	is an $S$-morphism 
	and $v_1,\dots,v_n$ is a basis of $f^*V$.
	Likewise, if $E'$ is the $\uGL_n(S)$-torsor corresponding to $V^\vee$,
	then $E'(Y)$ is the set of tuples $(f,u_1,\dots,u_n)$ consisting
	of an $S$-morphism $f:Y\to X$
	and a basis  $u_1,\dots,u_n$  of $f^*(V^{\vee})\cong (f^*V)^\vee$.
	There is a natural bijection between $E(Y)$ and $E'(Y)$
	mapping $(f,v_1,\dots,v_n)$ to $(f,u_1,\dots,u_n)$,
	where $u_1,\dots,u_n$ is the dual basis of $v_1,\dots,v_n$.
	Thus, $E\cong E'$ as $S$-schemes. One readily checks that
	pulling the $\uGL_n(S)$-action on $E'$ via this isomorphism
	gives the $\uGL_n(S)$-action on $E$ in the lemma.
\end{proof}

\begin{prp}\label{PR:quo-of-GLn-torsor-is-a-scheme}
	Let $S$ be scheme, and let $G$ be a subgroup of
	$\uGL_n(S)$ 
	that is flat and locally of finite presentation over $S$,
	and 
	such   that   $\uGL_n(S)/G$  
	is an $S$-scheme.
	Let $E $ be a $\uGL_n(S)$-torsor over an $S$-scheme $X$. Then the $S$-algebraic space $E/G$ 
	(see Lemma~\ref{LM:free-action-quotient}) is a scheme.
	Moreover, if $X$ and $\uGL_n(S)/G$
	are   of finite type (resp.\ locally of finite type) over $S$, then so is $E/G$.
\end{prp}

\begin{proof}
	Let $V$ be the vector bundle corresponding to $E$.
	Then there is an open covering $\{U_i\}_{i\in I}$
	of $X$ such that $V|_{U_i}\cong \calO_{U_i}^n$ for all $i\in I$.
	Thus, if $U$ is one of the $U_i$, or an intersection $U_i\cap U_j$,
	we have a $\uGL_n(S)$-equivariant isomorphism 
	$E|_U\cong U\times_S \uGL_n(S)$. This means that $E|_U/G\cong U\times_S (\uGL_n(S)/G)$,
	which is a scheme by our assumption.
	Observe that the pullback
	of $E|_U/G\to E/G$ along $E\to E/G$ is the open immersion $E|_U\to E$.
	Since $E\to E/G$ is an fppf covering
	(Lemma~\ref{LM:free-action-quotient}),
	this means that $E|_U/G\to E/G$ is an open immersion of algebraic spaces
	\cite[Tag \href{https://stacks.math.columbia.edu/tag/041X}{041X}]{stacks_project}.
	Similarly, $E|_{U_i\cap U_j}/G$ is open in $E|_{U_i}/G$ for all $i,j\in I$.
	This shows that $E/G$ can be obtained by gluing the schemes $E|_{U_i}/G$
	($i\in I$)
	along their open subschemes   $E|_{U_i\cap U_j}/G$, so $E/G$ is a scheme.
	
	If $X$ and $\uGL_n(S)/G$ are 
	locally of finite type over $S$, then so are
	the schemes $E|_U/G\cong U\times (\uGL_n(S)/E)$ and it follows
	that $E/G$ is also locally of finite type.
	
	Suppose now that $X$ and $\uGL_n(S)/G$ are 
	of finite type over $S$, i.e., they are locally
	of finite type and quasi-compact over $S$.
	We need to show that $E/G\to S$ is quasi-compact. To that end,
	it is enough to consider the case where $S$ is affine.
	Then $X$ is quasi-compact, and thus 
	we may choose the covering $\{U_i\}_{i\in I}$ to be a 
	\emph{finite} covering. Since $\uGL_n(S)/G$
	is quasi-compact, each of the $E|_{U_i}/G$
	is quasi-compact, and since $I$ is finite, $E/G$ would be quasi-compact as well.
\end{proof}

\section{Schematically Dominant   Morphisms}
\label{sec:schematically-dom}

In order to adapt the theory of versal torsors over fields to 
a general base scheme, we need to replace   dominant morphisms 
with \emph{schematically dominant} morphisms
and \emph{universally schematically dominant} morphisms. 
This section is recalls these notions
and establishes  some  related facts  
that will be needed in the sequel.

Throughout,   $S$ is a fixed base scheme.

\medskip

Schematically dominant morphisms of schemes 
are discussed in detail in   EGA${}_{\mathrm{IV}}$ \cite[\S11.10]{EGA_iv_1967}.
The definition extends naturally to $S$-algebraic spaces as follows.

A  morphism of $S$-algebraic spaces $f:Y\to X$ is said to be 
\emph{schematically dominant} if the induced  morphism $\calO_X\to f_*\calO_Y$ is a monomorphism
of sheaves on the (small) \emph{\'etale} site of $X$.
Equivalently, for every \'etale morphism $U\to X$ of $S$-algebraic spaces,
the induced map $f^*:\Gamma(U,\calO_X)\to \Gamma(Y\times_X U,\calO_Y)$ is injective.
More generally, a family of   morphisms $\{f_i:Y_i\to X\}_{i\in I}$
is said to be (jointly)  schematically dominant if the   morphism $\bigsqcup_{i\in I}Y_i\to X$
restricting to $f_i$ on $Y_i$
is schematically dominant.

The morphism $f:Y\to X$ is said to be \emph{$S$-universally schematically
dominant} if for every $S$-algebraic space $S'$, the pullback
$f_{S'}:Y_{S'}\to X_{S'}$ is schematically dominant.
It is said to be \emph{universally schematically dominant} if 
for every morphism of $S$-algebraic spaces $h:X'\to X$, the pullback $f_h:Y\times_XX'\to X'$
is schematically dominant.
Similarly, a family of morphisms $\{f_i:Y_i\to X\}_{i\in I}$
is said to be  ($S$-)uni\-versally  schematically dominant  if the induced 
morphism $\bigsqcup_{i\in I}Y_i\to X$
is ($S$-)uni\-versally schematically dominant.

One readily sees that the composition of schematically dominant morphisms
is schematically dominant. If $g:Z\to Y$ and $f:Y\to X$ are two morphisms of $S$-algebraic spaces
and $f\circ g$ is schematically dominant, then so is $f$. Similar statements hold
for ($S$-)universally schematically dominant morphisms.

\begin{example}\label{EX:sch-dom-basic-examples}
	(i) Suppose $S=\Spec R$ for a ring $R$
	and  $\vphi:A\to B$ is a morphism of $R$-rings.
	Then the induced morphism $f:\Spec B\to \Spec A$ is schematically
	dominant if and only if $\vphi$ is injective. (Indeed, if $U:=\Spec A'\to \Spec A$
	is \'etale, then $\Gamma(U,\calO_{\Spec A})\to \Gamma({\Spec B}\times_{\Spec A} U,\calO_{\Spec B})$
	is just the   map $a'\mapsto 1_B\otimes a':A'\to B\otimes_A A'$, and it is injective 
	whenever $\vphi$ is, because
	$A'$ is flat over $A$.)
	
	(ii) Any fpqc covering of $S$-algebraic spaces $f:Y\to X$ is universally
	schematically dominant. This follows  from fpqc descent for the structure
	sheaf of $X$ on the large fpqc site of $X$ 
	(\cite[Tag \href{https://stacks.math.columbia.edu/tag/04P2}{04P2}]{stacks_project};
	the structure sheaf is represented by $\bbA^1_X$).  
\end{example}

\begin{remark}
	In EGA${}_{\mathrm{IV}}$ \cite[\S11.10]{EGA_iv_1967},
	a morphism of $S$-schemes $f:Y\to X$ is called schematically
	dominant when $\calO_X\to f_*\calO_Y$ is a monomorphism
	of sheaves on the (small) \emph{Zariski} site of $X$,
	and $S$-universally schematically dominant if 
	$f_{S'}:Y_{S'}\to X_{S'}$ is schematically dominant for every
	\emph{$S$-scheme} $S'$. This agrees with our definition
	by the next proposition. Consequently, any $S$-morphism of $S$-schemes $f:Y\to X$
	that is schematically dominant is dominant, and the converse holds if $X$ is reduced. 
\end{remark}

\begin{prp}\label{PR:sch-dom-equiv-with-SGA}
	Let $S$ be a scheme and let $f:Y\to X$ be a morphism of $S$-schemes. 
	Write $\calO_X^{\Zar}$ (resp.\ $\calO_X^{\et}$)
	for the structure sheaf of $X$ on the small Zariski (resp.\ \'etale)
	site of $X$.
	Then:
	\begin{enumerate}[label=(\arabic*)]
		\item The   map   $\calO^{\Zar}_X\to f_*\calO^{\Zar}_Y$ induced by $f$
		is a monomorphism
		if and only if the  map 
		$\calO^{\et}_X\to f_*\calO^{\et}_Y$ induced by $f$ is a monomorphism.
		\item If $f_{S'}:Y_{S'}\to X_{S'}$ is schematically dominant for every
		$S$-scheme $S'$, then this is also true for every $S$-algebraic space $S'$.
	\end{enumerate}
\end{prp}

\begin{proof}
	(i) The ``if'' part is clear.
	As for the ``only if'' part, by \cite[Theorem 11.10.5(ii)]{EGA_iv_1967}, for every morphism $g:X'\to X$
	that is flat and locally of finite presentation, the pullback $f_g:Y\times_XX'\to X'$ 
	also satisfies the condition that $\calO_{X'}^{\Zar}\to (f_g)_*\calO^{\Zar}_{Y\times X'}$ is a mononmorphism. 
	Thus, $\calO_X(X')\to \calO_{Y}(Y\times_XX')$ is injective.
	Letting $g:X'\to X$ range over the \'etale morphisms to $X$ shows that
	$\calO^{\et}_X\to f_*\calO^{\et}_Y$ is a monomorphism.
	
	(ii) Let $g:S'\to S$ be a morphism of $S$-algebraic spaces.
	Then there is an $S$-scheme $S''$ and an \'etale covering $h:S''\to S'$.
	By assumption, 
	$f_{S''}:Y_{S''}\to X_{S''}$	
	is schematically
	dominant. Furthermore, since $h_X:X\times_SS''\to X\times_S S'$ is an \'etale covering
	of $S$-algebraic spaces, it is schematically dominant (Example~\ref{EX:sch-dom-basic-examples}(ii)).
	Thus, the composition $h_X\circ f_{S''}:Y\times_S S''\to X\times_S S'$ is schematically
	dominant. The morphism $h_X\circ f_{S''}$ also factors as $f_{S'}\circ h_Y$, so  
	$f_{S'}:Y_{S'}\to X_{S'}$ is also   schematically dominant.
\end{proof}

\begin{thm}[{\cite[Theorem 11.10.5]{EGA_iv_1967}}] \label{TH:sch-dom-at-the-fibers}
	Let $X$ be a locally noetherian $S$-scheme and 
	let $\{f_i:Z_i\to X\}_{i\in I}$ be morphisms of $S$-schemes  
	such that each $Z_i$ is flat over $S$. For $s\in S$, put $X_s=X\times_S \Spec k(s)$
	and define $(Z_i)_s$, $(f_i)_s$ similarly. Then the family
	$\{f_i:Z_i\to X\}_{i\in I}$ is $S$-universally schematically
	dominant if and only if, for every $s\in S$, the family 
	$\{(f_i)_s:(Z_i)_s\to X_s\}$ is schematically   dominant.
\end{thm}

This theorem has a number of corollaries which we record.

\begin{cor}\label{CR:mor-of-var-is-sch-dom}
	Let $k$ be a field, let $X$ be a reduced locally noetherian $k$-scheme,
	and let $f:Y\to X$ be a morphism
	of $k$-schemes. Then $f$ is dominant if and only if it is $k$-universally
	schematically dominant.
\end{cor} 

\begin{cor}\label{CR:sections-of-An}
	Suppose that $S$ is an affine   noetherian   scheme with infinite residue fields
	and let $n\in\N\cup\{0\}$.
	Let $\{f_i:S\to \bbA^n_S\}_{i\in I}$ be the $S$-sections of $\bbA^n$.
	Then the family $\{f_i:S\to \bbA^n_S\}_{i\in I}$  is $S$-universally schematically dominant.
\end{cor}

\begin{proof}
	Write $S=\Spec R$.
	The  $S$-sections of $\bbA^n_S$ can be naturally identified with $R^n$.
	Let $s\in S$, $k=k(s)$ and put $\quo{R}=\im (R\to k)$.
	Then, once viewing $k^n$
	as the set of $k$-sections
	of $\bbA^n_k$,
	the family  
	$\{(f_i)_s:\Spec k\to \bbA^n_k\}_{i\in I}$
	is the subset $\quo{R}^n$.
	Since $k$ is the fraction field of $\quo{R}$, the ring $\quo{R}$ is infinite.
	It is now routine to check that $\{(f_i)_s:\Spec k\to \bbA^n_k\}_{i\in I}$ is
	dominant, and thus schematically dominant because $\bbA^n_k$ is reduced.
	We finish by applying Theorem~\ref{TH:sch-dom-at-the-fibers}.
\end{proof}

\begin{cor}\label{CR:sections-of-GLn}
	Let $k$ be an infinite field, let $n\in\N$ 
	and let $\{f_i:\Spec k\to \uGL_n(k)\}_{i\in I}$
	denote the $k$-sections of the $k$-scheme $\uGL_n(k)$.
	Then $\{f_i:\Spec k\to \uGL_n(k)\}_{i\in I}$ is $k$-universally
	schematically dominant.
\end{cor}

We say that a   morphism of 
$S$-algebraic spaces $f:Y\to X$ is \emph{$S$-fppf surjective} if it is surjective
when viewed as a morphism of sheaves on $(\Sch/S)_\fppf$.
This equivalent to saying that for every $S$-scheme
$V$ and $x\in X(V)$, there is an fppf covering $V'\to V$
and $y\in Y(V')$ such that $f(y)=x|_{V'}$.

\begin{lem}\label{LM:fppf-surj-is-univ-sch-dom}
	Let $f:Y\to X$ be an $S$-fppf surjective morphism of $S$-algebraic spaces. Then
	$f$ is  universally schematically dominant.
\end{lem}

\begin{proof}
	Let $g:X'\to X$ be a morphism of $S$-algebraic spaces.
	Let $\catS$ denote the category of sheaves on $(\Sch/S)_\fppf$.
	By viewing $X$, $Y$ and $X'$ as objects of $\catS$,
	one readily checks that $f_g:Y\times_XX'\to X'$ is also $S$-fppf surjective.
	It is therefore enough to show that  $f$ is schematically dominant.
	Let $U\to X$ be an \'etale morphism. We need
	to show that $\Gamma(U,\calO_X)\to \Gamma(Y\times_X U,\calO_Y)$ is injective.
	This is equivalent to saying that $\Mor_\catS(U,\bbA^1_S)\to\Mor_\catS(Y\times_XU, \bbA^1_S)$ is injective,
	and this holds because $Y\times_XU\to U$ is $S$-fppf surjective, i.e., surjective in $\catS$.
\end{proof}

The following results give several examples of $S$-fppf surjective morphisms.
By Lemma~\ref{LM:fppf-surj-is-univ-sch-dom}, these are also examples
of   universally schematically dominant morphisms.

Given a ring $A$, an element $a\in A$
and distinct  $i,j\in\{1,\dots,n\}$, we write $e_{ij}(a)$ for  
the matrix in $\nSL{A}{n}$ having $1$-s on the diagonal, $a$ in the $(i,j)$-entry,
        and $0$ in all other entries.

	\begin{prp}\label{PR:SLn-generation}
		For every $m\in\N$, there exists $r\in\N\cup\{0\}$ and
		a sequence of pairs of distinct integers $(i_1,j_1),\dots,(i_r,j_r)\in\{1,\dots,m\}^2$
		such that the morphism $f:\bbA^r_\Z\to \uSL_m(\Z)$ given by
		\[
		f( a_1,\dots ,a_r)=e_{i_1j_1}(a_1)\cdots e_{i_rj_r}(a_r)
		\]
		on sections is $\Z$-fppf surjective, and therefore universally schematically dominant.
    \end{prp}
    
    In fact, the proof shows that $f:\bbA^r_\Z\to \uSL_m(\Z)$ is even surjective when regarded
    as a morphism of sheaves on the large  Zariski  site of $\Z$.

    \begin{proof}
    	\Step{1} We first claim that there are $(i_1,j_1),\dots,(i_r,j_r)\in\{1,\dots,m\}^2$
    	such that for every semilocal ring $A$,
    	the induced map $f_A:A^r\to \nSL{A}{m}$ is surjective. This is the same
    	as showing that every matrix in $\nSL{A}{m}$ can be reduced to the identity
    	matrix by using $r$ \emph{predetermined} (and possibly vacuous) row and column operations
    	of the form $R_i+\alpha R_j\to R_i$ or $C_i+\alpha C_j\to C_i$. This is known, but we include
    	a proof for the sake of completeness.
    	   
        We prove this by induction on $m$. When $m=1$, we can take $r=0$.
        Suppose that   $m>1$,  and let $(i_\ell,j_\ell)_{\ell=1}^s$ be the sequence
        of pairs in $\{1,\dots,m-1\}^2$ supplied by the induction hypothesis.
		Let $A$ be a semilocal ring and let  $a=(a_{ij})\in\nSL{A}{m}$.
		Since $\det(a)=1$,
        we have $\sum_{i=1}^m a_{im}A=A$.
        As $A$ is semilocal, it has \emph{stable range} $1$ 
        (see \cite[\S4, Corollary~6.5]{Bass_1964_K_theory_and_stable_algs}), hence
        there are $\alpha_1,\dots,\alpha_{m-1}\in A$ such that
        $\sum_{i=1}^{m-1} (a_{im}+\alpha_ia_{mm})A=A$.
        We may therefore replace $a$ with $e_{1m}(\alpha_1)e_{2m}(\alpha_2)\cdots e_{(m-1)m}(\alpha_{m-1})a$,
        and assume that $\sum_{i=1}^{m-1} a_{im}A=A$.
        (This is allowed because the pairs $(1,m),\dots,(m-1,m)$ do not depend on $A$ or $a$.)
        Now, if $m-1>1$, then again there are $\alpha_1,\dots,\alpha_{m-2}\in A$
        with $\sum_{i=1}^{m-2} (a_{im}+\alpha_ia_{(m-1)m})A=A$, and we may
        replace $a$ with $e_{1m}(\alpha_1)e_{2m}(\alpha_2)\cdots e_{(m-2)m}(\alpha_{m-2})a$
        to assume $\sum_{i=1}^{m-2} a_{im}A=A$. Continuing in this manner, we reduce to the case
        where   
        $a_{1m}$ is invertible. Now, by replacing $a$ with $e_{m1}(a_{1m}^{-1}(1-a_{mm}))a$,
        we may assume that $a_{mm}=1$.
        Replacing $a$ once more by
        by $e_{1m}(-a_{1m})\cdots e_{(m-1)m}(-a_{(m-1)m})\cdot a\cdot 
        e_{m(m-1)}(-a_{m(m-1)})\cdots e_{m1}(-a_{m1})$
        reduces us to the case
        where $a=
        [\begin{smallmatrix} a' & 0 \\ 0 & 1 \end{smallmatrix}]$
        for some $a'\in\nSL{A}{m-1}$. 
        By the induction hypothesis, there are
        $\alpha_1,\dots,\alpha_s\in A$ such that
        $a=e_{i_1j_1}(\alpha_1)\cdots e_{i_sj_s}(\alpha_s)$ and we are done.

\smallskip		
	
		\Step{2} We now prove that $f:\bbA^r_\Z\to \uSL_m(\Z)$  is
		$\Z$-fppf surjective for
		the pairs $(i_1,j_1),\dots,(i_r,j_r)$ supplied by Step~1. 
		Let $V$ be a $\Z$-scheme and let $a\in \nSL{V}{m}$.
		By the previous paragraph, for every $x\in V$, there
		are $a_1^{(x)},\dots,a_r^{(x)}\in \calO_{V,x}$ such that
		$e_{i_1 j_1}(a_1^{(x)})\cdots e_{i_r j_r}(a_r^{(x)})=a_x$,
		where $a_x$ is the restriction of $a$ to $\Spec \calO_{X,x}$.
		By spreading out, we may assume that $a_1^{(x)},\dots,a_r^{(x)}$
		are defined on an open affine neighborhood
		$V_x$ of $x$ and that  $e_{i_1 j_1}(a_1^{(x)})\cdots e_{i_r j_r}(a_r^{(x)})=a|_{V_x}$.
		Put $V'=\bigsqcup_{x\in V} V_x$. Then $V'\to V$ is a Zariski covering of $V$
		such that $a|_{V'}$ is in the image of $\bbA^n_S(V')\to \nSL{V'}{m}$.
    \end{proof}

Let $H$ be an $S$-group acting from the left on an $S$-algebraic space $X$. We say that the action of $H$
on $X$  
is fppf transitive if the   map $H\times_S X\to X\times_S X$
given section-wise by $(h,x)\mapsto (hx,x)$ is $S$-fppf surjective.
This is equivalent to saying that for every $X$-scheme $V$
and $x,x'\in X(V)$, there exist  an fppf covering 
$V\to V'$ and $h\in H(V')$ such that $x'|_{V'}= h\cdot x|_{V'}$. In fact, it is enough to verify
this only when $V$ is affine.

\begin{lem}\label{LM:strongly-trans}
	Let $H$ be an $S$-group acting   from the left
	on an $S$-algebraic space $X$, and let $f:Y\to X$ be an $S$-morphism
	of algebraic spaces.
	If the action of $H$ on $X$ is fppf transitive, then
	the morphism $H\times_S Y\to X\times_S Y$ given section-wise
	by $(h,y)\mapsto (h\cdot f(y),y)$ is $S$-fppf surjective,
	and thus  universally schematically dominant.
\end{lem}

\begin{proof}
	Let $V$ be an $S$-scheme
	and let $(x,y)\in X(V)\times Y(V)$.
	Since the action of $H$ on $X$ is fppf transitive,
	we may replace $V$ with a suitable fppf covering to assume
	that there is $h\in H(V)$
	such that $h \cdot f(y)=x $.
	Then $(h\cdot f(y),y)=(x,y)$ in $X(V)\times Y(V)$.
\end{proof}

\section{$d$-Versal Torsors: Definition and Properties}
\label{sec:definitions}

Throughout this section, $S$ is a fixed base scheme. If not indicated
otherwise,   all schemes
and algebraic spaces are over $S$, all morphisms between them are $S$-morphisms,
and $G$ denotes an $S$-group.
Given $S$-algebraic spaces $X$, $Y$, $Z$ and morphisms $f:Z\to X$, $g:Z\to Y$,
we write $(f,g)$ to denote the morphism $Z\to X\times_S Y$ given by $z\mapsto (f(z),g(z))$
on sections.

\begin{dfn}\label{DF:versality-full}
Let $X$ be an $S$-algebraic space,
let $E\to X$ be a $G$-torsor, let $\catC$ be a class  of $S$-schemes,
and let $d\in\N\cup\{0 \}$.
\begin{enumerate}[label=(\roman*)]
	\item We say that $E\to X$ is \emph{weakly $d$-versal for $\catC$}
	if for every $G$-torsor $E'\to X'$ 
	with $X'\in \catC$,
	there are
	\begin{enumerate}[label=(\arabic*)]
		\item  an open $U\subseteq X'$ such that 
		$\codim(X'-U,X')>d$	  and
		\item an $S$-morphism $f:U\to X$
	\end{enumerate}	 
	such that $E'|_U\cong_U f^*E$ (the isomorphism is in $\Tors(G,U)$; see Section~\ref{sec:torsors}).
	
	\item We say that $E\to X$ is \emph{$d$-versal for $\catC$}
	if for every $G$-torsor $E'\to X'$ 
	with $X'\in \catC$,
	there are
	\begin{enumerate}[label=(\arabic*)]
	 	\item an open $U\subseteq X'$ such that $\codim(X'-U,X')>d$	 
		and 
		\item a collection of $S$-morphisms $\{f_i:U\to X\}_{i\in I}$
	\end{enumerate} 
	such that $E'|_U\cong_U f_i^*E$
	for all $i\in I$, and the family 
	$\{(f_i,\id_U):U\to X\times_S U\}_{i\in I}$	
	is
	$U$-universally schematically dominant.

	\item We say that $E\to X$ is \emph{strongly $d$-versal for $\catC$}
	if for every $G$-torsor $E'\to X'$ 
	with $X'\in \catC$,
	there are
	\begin{enumerate}[label=(\arabic*)]
		\item an open $U\subseteq X'$ such that $\codim(X'-U,X')>d$	  and
		\item an $S$-morphism $f:\bbA^n_S\times_S U\to X$ (for some $n\in\N$)
	\end{enumerate}
	such that $\bbA^n_S\times_S E'|_U\cong_{\bbA^n \times U} f^*E$ and the morphism
	$(f,p_2):\bbA^n_S\times_S U\to X\times_S U$, where $p_2:\bbA^n_S\times_S U\to U$
	is the second projection,	
	is 
	universally 
	schematically
	dominant.
\end{enumerate}

We say that $E\to X$ is  (weakly, strongly) $\infty$-versal for $\catC$ if it
is  (weakly, strongly) $d$-versal for every $d\in\N$.

We  say that $E\to X$ is (weakly, strongly) versal
for $\catC$ if we can always take $U=X'$ in (1).
\end{dfn}

Admittedly, there is some room for variation in the definition, e.g.,
we can replace ``($U$-)universally schematically dominant''
with ``schematically dominant'' in (ii) and (iii). The present formulation allows
for both strong and clean statements. Before we proceed, a few remarks are in order.

\begin{remark}
	(i) 
	Suppose that $S=\Spec k$ for an infinite field $k$.
	As noted in the introduction, a $G$-torsor is weakly $0$-versal
	for the class of irreducible $k$-varieties  if and only if it is weakly versal
	in the sense of Duncan and Reichstein \cite{Duncan_2015_versality_of_alg_grp_actions}.
	However, being (strongly) $0$-versal for this class is {\it a priori}
	stronger than being (strongly) versal in the sense of [op.\ cit.].
	
	(ii) Suppose that $S=\Spec k$ for a   field $k$.
	Conditions (ii) and (iii) of Definition~\ref{DF:versality-full}
	differ from the  simplified versions (ii$'$) and (iii$'$) of the introduction
	in two ways: First, in Definition~\ref{DF:versality-full},
	we require ($S$-)universal schematic dominance  and not just schematic dominance.
	Second, in Definition~\ref{DF:versality-full},
	we require  dominance for the family $\{(f_i,\id_U):U\to X\times_S U\}_{i\in I}$
	(resp.\ the morphism $(f,p_2):\bbA^n_S\times_S U\to X\times_S U$)
	instead of the family $\{f_i:U\to X\}_{i\in I}$
	(resp.\ the morphism $f:\bbA^n_S\times_S U\to X$).
	This is more restrictive than (ii$'$) (resp.\ (iii$'$)),
	because $X\times_k U\to X$ is schematically dominant whenever $U\neq\emptyset$
	(Corollary~\ref{CR:mor-of-var-is-sch-dom}). 
	The reason for making this change
	is that when $S$ is general,
	the morphism $X\to S$ is typically surjective 
	(see Proposition~\ref{PR:base-of-versal-torsor} below),
	and thus   the families considered in (ii$'$)
	and (iii$'$) can be dominant only when $X'\to S$ is dominant.
	In addition, by making this change, we get
	that (weak, strong) versality   persists under base-change of $S$ 
	(Proposition~\ref{PR:versality-under-base-change}).

	(iii) When the class $\catC$ consists of finite-dimensional
	schemes, being (weakly, strongly) $\infty$-versal for   $\catC$
	is the same as being  
	(weakly, strongly) versal for $\catC$.	
\end{remark}

\begin{remark}\label{RM:versal-objects}
	Similarly to Definition~\ref{DF:versality-full},
	one can define rank-$n$ vector bundles
	that are (weakly, strongly) $d$-versal  for a class
	of $S$-schemes $\catC$. Since rank-$n$ vector bundles over
	$S$-schemes are equivalent to $\uGL_n(S)$-torsors, a rank-$n$ bundle would
	be (weakly, strongly) $d$-versal for $\catC$ if and only if the same holds for
	its   corresponding $\uGL_n(S)$-torsor.
	
	More generally, if $\mathsf{F}$ is a fibered category over $S$-algebraic spaces,
	then we can talk about  objects in $\mathsf{F}$ that are (weakly, strongly) 
	$d$-versal for $\catC$. Propositions~\ref{PR:versality-basic-props}, \ref{PR:versality-under-base-change}
	and~\ref{PR:base-of-versal-torsor} below actually hold in this more general setting.
	Examples of fibered categories that will be considered
	later include the   category of degree-$n$ Azumaya algebras over $S$-algebraic spaces,
	which is equivalent to $\Tors(\uPGL_n(S))$, and, for a fixed finite group $\Gamma$,
	the category of $\Gamma$-Galois extensions of $S$-algebraic spaces, which is equivalent to $\Tors(\underline{\Gamma})$. (Here, $\underline{\Gamma}$ is the constant $S$-group
	associated to $\Gamma$.)
	In both examples,  the  morphisms from $A$ to $B$ are the specializations of $B$ to $A$.
	Another example that is not equivalent to a category of torsors under an $S$-group
	is the category of degree-$m$ Azumaya algebras over $S$-algebraic spaces
	whose Brauer class has period dividing $n$ ($m$ and $n$ are fixed).
\end{remark}

\begin{prp}\label{PR:versality-basic-props}
	Let $\catC$ be a class of   $S$-schemes, let $X$
	be an $S$-algebraic space and let $E\to X$ be a $G$-torsor.
	\begin{enumerate}[label=(\roman*)]
		\item If $E\to X$ is   $d$-versal for $\catC$ then it is weakly $d$-versal for $\catC$.
		\item If $E\to X$ is strongly $d$-versal for $\catC$, then it is weakly $d$-versal for $\catC$.
		If 
		$S$ is moreover a scheme over
		an affine noetherian scheme with infinite residue
		fields,
		then
		$E\to X$ is also $d$-versal for $\catC$.
	\end{enumerate}
\end{prp}

\begin{proof}
	(i) is immediate from the definition.
	
	(ii) Let $E'\to X'$ be a $G$-torsor with $X'\in \catC$,
	and let $U$
	and   $f:\bbA^n_S\times U\to X$ be as in Definition~\ref{DF:versality-full}(iii).
	Let $u:\bbA^n_U\to U$ be the structure morphism and let $z:U\to \bbA^n_U$ be the zero section.
	Then $u \circ z =\id_U$  and thus $z^*f^*E\cong z^*(\bbA^n_S\times_S E'|_U)=z^*u^*(E'|_U)=E'|_U$.
	This shows that   $E\to X$ is weakly $d$-versal for $\catC$.

	Keeping the previous notation, suppose now that there is a morphism of schemes $S\to S_0$, where 
	$S_0$ is affine, noetherian and has infinite residue fields.
	Let $\{g_i:S\to \bbA^n_S\}_{i\in I}$ be the family of sections
	of $\bbA^n_S\to S$. 
	We claim that the family  $\{f_i:=f\circ(g_i)_U: U\to   X \}_{i\in I}$ satisfies
	the requirements listed in Definition~\ref{DF:versality-full}(ii).
	Indeed, for every $i\in I$, we have $u\circ (g_i)_U=\id_U$
	and thus, as in the previous paragraph, $f_i^*E\cong_U E'|_U$.
	Next,  by Corollary~\ref{CR:sections-of-An},
	the family   $\{(g_i)_U: U\to   \bbA^n_U \}_{i\in I}$ is  
	$S$-universally schematically dominant.
	By assumption, $(f,p_2):\bbA^n_S\times_S U\to X\times_S U$
	is universally schematically dominant, so
	the family $\{(f_i,\id_U)=(f,p_2)\circ (g_i)_U\}_{i\in I}$
	is $S$-universally schematically dominant. 	
\end{proof}

\begin{prp}\label{PR:versality-under-base-change}
	Let $\catC$ be a class of   $S$-schemes, let $X$
	be an $S$-algebraic space and let $E\to X$ be a $G$-torsor.
	Let $S'$ be an $S$-scheme and let $\catC/S'$
	denote the class of $S'$-schemes with underlying $S$-scheme
	in $\catC$.
	If $E\to X$ is (weakly, strongly) $d$-versal for $\catC$,
	then the $G_{S'}$-torsor $E_{S'}\to X_{S'}$ is (weakly, strongly)
	$d$-versal for $\catC/S'$.
\end{prp}

\begin{proof}
	Let $E'\to X'$ be a $G_{S'}$-torsor with $X'\in \catC/S'$.
	Then $E'\to X'$ is also a $G$-torsor. 
	Let $s$ be the structure morphism $X'\to S'$.
	The proposition  follows from the following three observations (applied
	with $X'=U$ or $X'=\bbA^n_S\times_S U$ for $U$ as in Definition~\ref{DF:versality-full}),
	which can be readily verified on the level of sections:
	First, if $f:X'\to X$ is an $S$-morphism such that $f^*E\cong_{X'} E'$,
	then $g=(f,s):X'\to X_{S'}$ satisfies $g^*E_{S'}\cong_{X'} E'$ as $G_{S'}$-torsors.
	Second, for every $f:X'\to X$,
	the morphism $(f,\id_{X'}):X'\to X\times_S X'$   coincides
	with  $(f_{S'},\id_{X'}):X'\to X_{S'}\times_{S'} X'$ under the evident
	identification $X\times_S X'\cong X_{S'}\times_{S'} X'$.
	Third, for every $f:\bbA^n_S\times_S X'\to X$,
	the morphism $(f,p_2):\bbA^n_S\times_S X'\to X\times_S X'$
	coincides with $(f_{S'},p_2):\bbA^n_{S'}\times_{S'} X'\to {X_{S'}}\times_{S'} X'$
	under the evident identifications.
\end{proof}

\begin{prp}\label{PR:base-of-versal-torsor}
	Let $E\to X$ be a $G$-torsor, where $X$ is an $S$-scheme.
	\begin{enumerate}[label=(\roman*)]
		\item If $E\to X$ is weakly versal for $\{S\}$,
		then $X\to S$ is surjective (i.e.\ set-theoretically surjective).
		\item If $E\to X$ is versal for $\{S\}$ and $S$ is reduced,
		then $X$ is reduced.
		\item If $E\to X$ is strongly versal for $\{S\}$
		and $S$ is integral, then $X$ is integral.
		\item If $E\to X$ is strongly versal for the $S$-schemes
		$\{\Spec k(s)\where s\in S\}$, then   $X\to S$ has integral fibers.
	\end{enumerate}
\end{prp}

\begin{proof}
	(i) View $G\to S$ as a $G$-torsor by letting $G$ act on itself from the right.
	Since $E\to X$ is weakly versal for $\{S\}$, there is an $S$-morphism $S\to X$.
	The composition $S\to X\to S$ is $\id_S$, so $X\to S$ is surjective.
	
	(ii) Applying the versality to the $G$-torsor $G\to S$,
	we get a schematically dominant family of  morphisms  $\{f_i:S\to X\times_S S=X\}_{i\in I}$.
	The induced map $\calO_X^{\Zar}\to \prod_i (f_i)_*\calO^{\Zar}_S$  is therefore injective
	(notation as in Proposition~\ref{PR:sch-dom-equiv-with-SGA}).
	Since $S$ is reduced, this means that $\calO^{\Zar}_X$ has no nilpotent sections, hence $X$ is reduced.
	
	(iii) Applying strong versality to the $G$-torsor $G\to S$,
	we get a   morphism
	$f:\bbA^n_S\to X$ such that induced map $\calO^{\Zar}_X\to f_*\calO^{\Zar}_{\bbA^n_S}$
	is   injective. Since $S$ is integral, so is $ {\bbA^n_S}$,
	and thus   $f_*\calO_{\bbA^n_S}^{\Zar}$
	is a sheaf of integral domains. It follows that $\calO_X^\Zar$
	is also a sheaf of integral domains, hence $X$ is integral.
	
	(iv) Let $s\in S$,   put $X_s=X\times_S \Spec k(s)$, and define $E_s$ and $G_s$
	similarly.  By  Proposition~\ref{PR:versality-under-base-change}, the $G_s$-torsor $E_s\to X_s$
	is strongly versal for $\{\Spec k(s)\}$, so $X_s$ is integral by (iii).
\end{proof}

\begin{prp}\label{PR:reduction-to-overgroup}
	Let $G$ be an   $S$-group and let $H$ be a subgroup
	of $G$ that is flat and locally of finite presentation over
	$S$. Let $E\to X$ be a $G$-torsor that is weakly  
	$d$-versal for a class of  $S$-schemes $\catC$.
	Then $E\to E/H$ is an $H$-torsor that is 
	weakly   $d$-versal for $\catC$.  
\end{prp}

\begin{proof}
	Lemma~\ref{LM:free-action-quotient}
	and our assumptions on $H$ imply that $E\to E/H$ is an $H$-torsor.
	
	Let $E'\to X'$ be an $H$-torsor with $X'\in\catC$,  
	put $P = E'\times^H G$ and let $i:E'\to P$
	denote the morphism given by $i(e)=[e,1]$ on sections. 
	Then $P\to X'$ is a $G$-torsor and $i$ is $H$-equivariant.
	Since $E\to X$ is weakly $d$-versal for $\catC $,
	there exists 
	an open subscheme $U\subseteq X'$ with $\codim(X'-U,X')>d$
	and 	
	a $G$-equivariant morphism
	$P|_U\to E$. The composition $E'|_U\to P|_U\to E$
	is $H$-equivariant, and thus induces
	a morphism of $H$-torsors
	from $E'|_U\to U$
	to $E\to E/H$. We finish by Remark~\ref{RM:morphism-pullback-correspondence}.
\end{proof}

Our next proposition  shows that weakly versal $G$-torsors
which admit an action by another $S$-group $H$
(in the sense of Section~\ref{sec:torsors}) 
that is   fppf transitive on the base (see Section~\ref{sec:schematically-dom}) 
are sometimes
versal, or even strongly versal.

\begin{prp}\label{PR:symmetric-weak-versal-is-strong}
	Let $\catC$ be a class of   $S$-schemes,
	let $X$ be an $S$-algebraic space
	and let $ E\to X$ be a $G$-torsor that is weakly $d$-versal for
	$\catC$.
	Suppose that there is an $S$-group $H$
	acting on $E\to X$ (from the left) such
	that the action of
	$H$ on $X$ is fppf transitive.
	Then:
	\begin{enumerate}[label=(\roman*)]
		\item If the family of $S$-sections of $H$ is $S$-universally
		schematically dominant, then $E\to X$ is $d$-versal for $\catC$.
		\item If there is an $S$-fppf surjective morphism $\bbA^n_S\to H$
		for some $n\in\N\cup\{0\}$, then $E\to X$ is strongly $d$-versal for $\catC$.
	\end{enumerate}
\end{prp}

\begin{proof}
	(i) Let $\{h_i:S\to H\}_{i\in I}$ be a collection of   schematically
	dominant $S$-morphisms; we view the $h_i$
	as elements of the group $H(S)$. 
	
	Let $E'\to X'$ be a $G$-torsor
	with $X'\in \catC$. Then there are an open $U\subseteq X'$ 
	with $\codim(X'-U,X')>d$ 
	and a $G$-torsor morphism $(\hat{f},f)$
	from $E'|_U$ to $E$.
	For every $i\in I$, define $f_i:U\to X$  by $f_i(u)=h_i \cdot f(u)$
	and $\hat{f}_i:E'|_U\to E$ by $\hat{f}_i(e)=h_i\cdot \hat{f}(e)$  on sections.
	One readily checks that each $(\hat{f}_i,f_i)$ is a $G$-torsor
	morphism from $E'_U$ to $E$. Thus, $f_i^*E\cong_U E'|_U$ for all $i\in I$
	(Remark~\ref{RM:morphism-pullback-correspondence}).

	By Lemma~\ref{LM:strongly-trans}
	the map  $g:H\times_S U\to X\times_S U$
	given section-wise by $g(h,u)= (h\cdot f(u), u)$ is universally
	schematically dominant. Our assumption on $H$
	implies that $\{(h_i)_U:U\to H\times_SU\}_{i\in I}$
	is $U$-universally schematically dominant,
	so the family $\{g\circ (h_i)_U:U\to X\times_SU\}_{i\in I}$
	is $U$-universally schematically dominant.
	On the level of sections, we have
	$(g\circ (h_i)_U)(u)=g(h_i,u)=(h_i\cdot f(u),u)=(f_i(u),u)$,
	so  
	$\{(f_i,\id_U):U\to X\times_S U\}_{i\in I}$ is $S$-universally
	schematically dominant. 
	We conclude that $E\to X$ is $d$-versal for $\catC$.

	(ii) Let $h:\bbA^n_S\to H$ be a universally schematically
	dominant morphism, and let $E'\to X'$, $U$,
	$(\hat{f},f)$ and $g$ be as in the proof of (i).
	Define $w: \bbA^n_S\times_S U\to X$
	by $w(v,u)=h(v)\cdot f(u)$ 
	and $\hat{w}:\bbA^n_S\times E'|_U\to E$
	by $\hat{w}(v,e)=h(v)\cdot \hat{f}(e)$  on sections.
	One readily checks that $(\hat{w},w)$ is a $G$-torsor
	morphism from $\bbA^n_S\times_S E'|_U$ to $E$, hence $w^*E\cong_{\bbA^n\times U} 
	\bbA^n_S\times E'|_U$.
	Moreover,  
	$(w,p_2):	\bbA^n_S\times_S U\to X\times_SU$
	coincides with $g\circ h_U$, which is universally schematically
	dominant because $g$ and $h$ are.
\end{proof}

Putting everything together, we arrive at the following corollary, which reduces
the existence of strongly $d$-versal torsors for linear group schemes to the existence
of certain weakly $d$-versal   vector bundles.

\begin{cor}\label{CR:symmetric-weak-versal-is-strong}
	Let $S_0$ be a scheme,
	let $\catC$ be a class of   $S_0$-schemes
	and let $n\in\N$.
	Suppose that  there exists
	an $S_0$-scheme $X_0 $ and a 
	rank-$n$ vector bundle $V_0$ over $X_0$ that is \emph{weakly} $d$-versal for $\catC$.
	Let $E_0\to X_0$ be the $\uGL_n(S_0)$-torsor corresponding to $V_0$,
	and suppose 
	moreover that   there are $m_1,\dots,m_r\in \N$ 
	such that $H_0 :=\uSL_{m_1}(S_0)\times \dots\times \uSL_{m_r}(S_0)$ acts on $E_0 \to X_0 $
	(see Section~\ref{sec:torsors}) and the action of $H_0 $ on $E_0 $ 
	is fppf transitive\footnote{
		Caution: It is not enough to assume that $H_0$ acts transitively on $X_0$.
	}.
	Then, for every $S_0$-scheme $S$ and every   flat
	locally of finite presentation $S$-group $G$ that is isomorphic
	to a subgroup of $\uGL_n(S)$, there exists
	a  $G$-torsor over an $S$-algebraic space
	that is \emph{strongly} $d$-versal for $\catC/S$.
	Specifically, if $\rho: G\to \uGL_n(S)$ is a monomorphism of $S$-groups, then we can take
	this torsor to be $(E_0 \times_{S_0} S)\to (E_0 \times_{S_0} S)/G $, where $G$
	acts on $ E_0 \times_{S_0} S$ via $\rho$.
\end{cor}

\begin{proof}
	Let $\rho: G\to \uGL_n(S)$ be a monomorphism;
	we view $G$ as a subgroup of $\uGL_n( S)$ via $\rho$.
	Put $E= E_0\times_{S_0} S$, $X= X_0\times_{S_0} S$ and $H= H_0\times_{S_0} S$.
	By   Proposition~\ref{PR:versality-under-base-change}, 
	$E\to X$ is a $\uGL_n(S)$-torsor
	that is weakly $d$-versal for $\catC/S$.
	Thus, by Proposition~\ref{PR:reduction-to-overgroup},
	$E\to E/G $ is a $G $-torsor that is weakly $d$-versal
	for $\catC/S$. The action of $H$
	on $E\to X$ gives rise to an action of $H$ on $E\to E/G$,
	and $H$ acts fppf transitively
	on $E/G$ because it acts fppf transitively
	on $E$ and   $E\to E/G $ is $S$-fppf surjective.
	By Proposition~\ref{PR:SLn-generation}, there exists a universally
	schematically dominant morphism $\bbA^n_S\to H$, so
	Proposition~\ref{PR:symmetric-weak-versal-is-strong}
	tells us that $E\to E/G $ is strongly $d$-versal for $\catC/S$.
\end{proof}

\begin{remark}\label{RM:versal}
	In Corollary~\ref{CR:symmetric-weak-versal-is-strong},
	if we assume that $S_0$ is a scheme over an infinite field 
	and replace $ \uSL_{m_1}(S_0)\times \dots\times \uSL_{m_r}(S_0)$
	with $\uGL_{m_1}(S_0)\times \dots\times \uGL_{m_r}(S_0)$,
	then we can conclude that the $G$-torsor $(E_n\times_{S_0}S)\to (E_n\times_{S_0}S)/G$
	is   $d$-versal for $\catC/S$.
	To see this, replace Proposition~\ref{PR:SLn-generation} with Corollary~\ref{CR:sections-of-GLn}
	in the proof.
\end{remark}

In the next two sections, we will  construct  vector bundles
that   can 
be fed into Corollary~\ref{CR:symmetric-weak-versal-is-strong}.

\section{Generic Vector Bundles with a Fixed Number of Generators}
\label{sec:generic}

Let $m\geq n\geq 0$ be integers.
In this section, we recall two constructions of ``generic'' $m$-generated rank-$n$ locally free
modules
and describe  their corresponding $\uGL_n(\Z)$-torsors.
Here, being generic means that these vector bundles specialize to every
rank-$n$ vector bundle that is generated by $m$ global sections,
provided the base scheme is in some prescribed class of schemes, e.g.,
the affine   schemes.

\emph{Throughout this section, all schemes are over $\Z$.
The functor of points of a scheme means its functor of points when considered
as a $\Z$-scheme.}

Recall that if $X$ is a scheme, $\calM$
an $\calO_X$-module,
and $m_1,\dots,m_t\in \Gamma(X,\calM)$,
then $\calM$ is said to be generated by $m_1,\dots,m_t$
if the $\calO_X$-module map $\calO_X^t\to \calM$
given section-wise by $(a_1,\dots,a_t)\mapsto \sum_{i=1}^t a_i m_i$
is surjective. 
When $X=\Spec A$ for a ring $A$
and $\calM$ is quasi-coherent, $m_1,\dots,m_t$
generate $\calM$
if and only if they generate the $A$-module $M:=\Gamma(X,\calM)$ 
(but this is false in the non-affine case, e.g., see Example~\ref{EX:generators-for-Om} below).

\medskip

We begin with recalling the  generic $m$-generated rank-$n$ projective module
studied by M.\ Raynaud in \cite{Raynaud_1965_universal_proj_module}.
The idea is that every rank-$n$ projective module $P$ over a ring $R$
that is generated by $m$ elements is isomorphic to a summand of $R^{m}$,
and can therefore   be expressed as the image of a rank-$n$ idempotent
matrix $e\in\nMat{R}{m}$, or equivalently, the cokernel of $1-e$. 
Thus, the image of such  a universal idempotent matrix is
a rank-$n$ projective module which specializes to every rank-$n$ $m$-generated
projective module over a commutive ring;
an analogous statement holds for the corresponding vector bundle.
 
Formally, let $ X_{n,m}$ be the 
closed subscheme of $\mathbb{A}^{m\times m}_{\mathbb{Z}}$
(the affine space of   $m\times m$ matrices over $\Z$) cut by the matrix equations
$x^2=x$ and $\mathrm{char.pol.}(x)=t^{m-n}(t-1)^n$. 
The functor of points of $X_{n,m}$ maps a ring $R$ to the set of rank-$n$
idempotent matrices in $\nMat{R}{m}$.\footnote{
	The rank of an idempotent $e\in \nMat{R}{m}$
	is defined to be the rank of the projective module $\im(e)$.
	Equivalently, this is the function $\Spec R\to \N\cup\{0\}$
	mapping $\frakp\in\Spec R$ to the rank of the image of $e$ in   $ \nMat{k(\frakp)}{m}$.
}
Let $A_{n,m}$ denote the coordinate ring of $X_{n,m}$,
and let $x_{ij}\in A_{n,m}$ be the $(i,j)$-coordinate function.
Then $e_{n,m}:=(x_{ij})\in \nMat{A_{n,m}}{m}$ is an idempotent matrix
of rank $n$. Set $P_{n,m}$ to be $\coker(1-e_{n,m}:A_{n,m}^{m}\to A_{n,m}^{m})$.
Then $P_{n,m}$ is a projective module of rank $n$ over $A_{n,m}$
that is generated by $m$ elements. 
The rank-$n$ vector bundle over $X_{n,m}$
corresponding to $P_{n,m}$
is denoted $V_{n,m}$.

\begin{prp}\label{PR:generic-Vnd}
	Let $X$ be an affine scheme and let $V$ be a rank-$n$
	vector bundle over $X$ that  is generated by $m$
	global sections. Then 
	there is a morhpism $f:X\to X_{n,m}$
	such that $V\cong f^*V_{n,m}$.
\end{prp}

\begin{proof}
	Write $X=\Spec A$ and let $P=\Gamma(X,V)$.
	Then $P$ is a rank-$n$ projective $A$-module that is generated by
	$m$ elements. Thus, there is a rank-$n$ idempotent
	$e \in \nMat{A}{m}$ and an $A$-module isomorphism
	$P\cong \coker(1-e : A^{m}\to A^{m})$.
	The    ring homomorphism $\vphi: A_{n,m}\to A$ which specializes $e_{n,m}$ to $e$
	satisfies $P_{n,m}\otimes_{A_{n,m}}A =\coker(1-e_{n,m})\otimes_{A_{n,m}}A 
	\cong \coker(1-e)\cong P$ and the proposition follows.
\end{proof}

We now describe   the $\uGL_n(\Z) $-torsor  $E_{n,m}\to X_{n,m}$
corresponding to $V_{n,m}$.

\begin{prp}\label{PR:GLn-torsor-of-Vnd}
	Let $m\geq n\geq 0$ be integers and put $d=m-n$.
	With notation as above, the $\uGL_n(\Z)$-torsor corresponding
	to the vector bundle $V_{n,m}$ over $X_{n,m}$ is isomorphic to
	the quotient morphism
	\[\uGL_{m}(\Z)/
	\begin{bmatrix}
	1_{n\times n} & 0 \\
	0 & \uGL_d(\Z)
	\end{bmatrix}
	\to 
	\uGL_{m}(\Z)/
	\begin{bmatrix}
	\uGL_n(\Z) & 0 \\
	0 & \uGL_d(\Z)
	\end{bmatrix}.\]
	Here, 
	$\uGL_n(\Z)$ acts on $\uGL_{m}(\Z)/[\begin{smallmatrix}
	1  & 0 \\
	0 & \uGL_d(\Z)
	\end{smallmatrix}]$ from the right
	via $x\cdot g= x[\begin{smallmatrix} g & 0 \\ 0 & 1_{d\times d} \end{smallmatrix}]$.
\end{prp}

\begin{proof}
	\Step{1} 
	We first introduce an auxiliary morphism of affine schemes $p:T_{n,m}\to X_{n,m}$.
	The functor of points of $T_{n,m}$ (as a $\Z$-scheme) maps a ring $A$ to
	\begin{align*}
	T_{n,m}(A)=
	\{(a,x,b)\in \nMat{A}{n\times m}&\times \nMat{A}{m}\times \nMat{A}{m\times n}
	\suchthat \\
	& x\in X_{n,d}(A),ba=x, ab=1_{n\times n}\}.
	\end{align*}
	and $p:T_{n,m}\to X_{n,m}$ is given section-wise by 
	$p(a,x,b)=x$.
	The $\Z$-group $\uGL_{n}(\Z)$  acts on $T_{n,m}$ from the right   by  
   	$
   	(a,x,b)\cdot g= (g^{-1}a,x,bg) 
   	$
   	on sections.

\smallskip   
   	
	\Step{2} 
	Let $E\to X_{n,m}$ be the $\uGL_n(\Z)$-torsor corresponding to $V_{n,m}$.
	We claim that there is a $\uGL_n(\Z)$-equivariant $X_{n,m}$-isomorphism
	from $T_{n,m}$ to $E$. In particular, $p:T_{n,m}\to X_{n,m}$ is a $\uGL_n(\Z)$-torsor.
	
   	Indeed, by Lemma~\ref{LM:GLn-torsor-of-bundle}, the functor of points
   	of $E$ sends an $A_{n,m}$-ring $R$
   	to the set of tuples
	$(f,v_1,\dots,v_n)$, where $f:\Spec R\to X_{n,m}$
	is a morphism and $v_1,\dots,v_n$
	form a basis to $f^*V_{n,m}$.
	The datum of $f$ is equivalent to specifying
	a rank-$n$ idempotent $e\in \nMat{R}{m}$,
	and then $f^*V_{n,m}$
	is just     $eR^{m}$.
	Thus, $E(R)$ is in a natural bijection
	with the set of tuples $(e,v_1,\dots,v_n)$
	where $e\in \nMat{R}{m}$ is an idempotent
	and $v_1,\dots,v_n\in R^{m}$ form an $R$-basis to $eR^{m}$.
	
	Let $e\in\nMat{R}{m}$ be an idempotent or rank $n$, 
	let $v_1,\dots,v_n\in R^{m}$
	and let   $b$ be the $ m \times n$ matrix whose columns are $v_1,\dots,v_n$.
   	Then $v_1,\dots,v_n$   form an $R$-basis of $eR^{m}$ if and only if
   	the map $b:R^n\to R^{m}$ is injective with image $eR^{m}$.
   	Since $R^{m}=eR^{m}\oplus (1-e)R^{m}$, 
   	the latter is equivalent to the existence of an $R$-module homomorphism
   	$a:R^{m}\to R^n$ (viewed as an $ m \times n$ matrix)
   	which vanishes on $(1-e)R^{m}$
   	and satisfies $ab=1_{n\times n}$
   	and $ba=\id_{e R^{m}}\oplus 0_{(1-e)R^{m}}=e$.
   	Such $a$ is determined uniquely by $b$ and satisfies
   	$ae=a$. We have therefore
   	shown that each  $(e,v_1,\dots,v_n)\in E(R)$  determines 
   	a triple $(a,e,b)\in T_{n,m}(R)$. Conversely,
   	every $(a,e,b)\in T_{n,m}(R)$ determines
   	a tuple $(e,v_1,\dots,v_n)\in E(R)$, by taking
   	$v_1,\dots,v_n$ to be the columns of $b$.
	It is routine to check that this one-to-one correspondence is compatible
   	with base-change of $R$ and is $\nGL{R}{n}$-equivariant, so it determines
   	a $\uGL_n(\Z)$-equivariant isomorphism $E\cong T_{n,m}$,
   	which is also easily seen to be an isomorphism over $X_{n,m}$.

\medskip
   	
   	\Step{3} We finish by showing that $T_{n,m}\to X_{n,m}$
   	is isomorphic to the torsor defined in the proposition.
   	Observe that $\uGL_{m}(\Z)$ acts on $T_{n,m}$
   	and $X_{n,m}$   from the left
   	by
   	\[
   	h\cdot (a,x,b)  := (ah^{-1}, hxh^{-1}, hb) 
   	\qquad\text{and}\qquad
   	h\cdot x = hxh^{-1}, 
   	\] 
   	respectively, and the map $p:T_{n,m}\to X_{n,m}$ is $\uGL_{m}(\Z)$-equivariant.
   	We  claim   that $\uGL_{m}(\Z)$ acts fppf transitively on $T_{n,m}$
   	and $X_{n,m}$. Provided this holds, consider the $\Z$-point
   	\[
   	t_0=(a_0,x_0,b_0):= \Circs{\begin{bmatrix}
   	1_{n\times n} & 0_{n\times d}
   	\end{bmatrix},
   	\begin{bmatrix}
   	1_{n\times n} & 0_{n\times d} \\
   	0_{d\times n} & 0_{d\times d}
   	\end{bmatrix},
   	\begin{bmatrix}
   	1_{n\times n}  \\
   	0_{d\times n}  
   	\end{bmatrix}}\in T_{n,m}(\Z)
   	\]
   	and its image $x_0\in X_{n,m}(\Z)$.
   	One readily checks that the stablizer of $t_0$ in $\uGL_{m}(\Z)$
   	is $[\begin{smallmatrix}
	1  & 0 \\
	0 & \uGL_d(\Z)
	\end{smallmatrix}]$,
   	so we have an isomorphism $\uGL_{m}(\Z)/[\begin{smallmatrix}
	1  & 0 \\
	0 & \uGL_d(\Z)
	\end{smallmatrix}]\to T_{n,m}$
   	given section-wise by sending $[h]$ to $h t_0$.
   	Similarly, $[h]\mapsto h\cdot x_0: \uGL_{m}(\Z)/[\begin{smallmatrix}
	\uGL_n(\Z)  & 0 \\
	0 & \uGL_d(\Z)
	\end{smallmatrix}]\to X_{n,m}$
   	defines an isomorphism and the proposition follows.
   	
\smallskip   	
   	
   	\Step{4} It remains to show that $\uGL_{m}(\Z)$ acts fppf transitively on $T_{n,m}$,
   	and {\it a fortiori} on $X_{n,m}$. To see this, let $R$ be a ring a ring
   	and let $(a,x,b)\in T_{n,m}(R)$.
   	Then $b$ defines an isomorphism from $R^n$ to $xR^{m}$.
   	By replacing $R$ with a Zariski covering (i.e., a product of localizations
   	$R_{s_1}\times \dots \times R_{s_t}$ with $s_1,\dots,s_t\in R$ generating the unit
   	ideal), we may assume that     $(1-x)R^{m}$ is a free $R$-module.
   	Let $c:R^d\to (1-x) R^{m}$ be an isomorphism, viewed as an $ m \times d$ matrix.
   	Then $xc=0$ and $(1-x)c=c$.
   	Let $h=\begin{bmatrix}
   	b & c
   	\end{bmatrix}\in \nGL{R}{n}$.
   	We have 
   	\begin{align*}
   	&hb_0=b, \\
   	&x h = \begin{bmatrix}
   	xb & xc
   	\end{bmatrix}=\begin{bmatrix}
   	b & 0
   	\end{bmatrix} =hx_0,\\
   	&ah=\begin{bmatrix}
	ab & ac 
\end{bmatrix}= \begin{bmatrix}
	ab & abac 
\end{bmatrix} = \begin{bmatrix}
	ab & ax(1-x)c 
\end{bmatrix}=\begin{bmatrix}
	1_{n\times n} & 0_{n\times d}  
\end{bmatrix}=a_0.
	\end{align*}
	This means that   that $h\cdot t_0 = (a,x,b)$,
	completing the    proof.
\end{proof}

We proceed with showing that the dual of the tautological rank-$n$
vector bundle over the Grassmannian of $n$-planes inside an $m$-plane
also specializes to every rank-$n$ vector bundle $V$ that is generated by $m$ global sections.
However, in contrast with $V_{n,m}$, this is true even if the base scheme of $V$ is not affine, 
and there is a natural bijection between the set of specializations and the 
set of $m$-tuples
generating $V$. That said, $V_{n,m}$ has the advantage that its base scheme $X_{n,m}$
is affine, and therefore --- as we shall see in Section~\ref{sec:existence} ---
it will ultimately give rise to $d$-versal torsors whose base is  affine.

For the sake of simplicity,  
given a morphism of schemes $f:Y\to X$
and a vector bundle $V$ over $X$, we shall   identify
$V$ with $V^{\vee\vee}$ and $f^*(V^\vee)$ with $(f^* V)^\vee$
via the standard natural isomorphisms. We also identify
$\calO_X^\vee$ with $\calO_X$ and $f^*\calO_X$ with $\calO_Y$ in the usual way.

\medskip

Let  $m\geq n\geq 0$  be integers.
We let $\Gr(n,m)$ denote the Grassmannian $\mathbb{Z}$-scheme of $n$-planes inside
an $m$-plane; see 
\cite[Tag \href{https://stacks.math.columbia.edu/tag/089R}{089R}]{stacks_project},
for instance.
The functor of points of $\Gr(n,m)$ sends a  scheme
$X$ to the set of rank-$n$ subbundles
of $\calO_X^{m}$, i.e.,
the set of rank-$n$ locally free $\calO_X$-submodules $\calV$ of $\calO_X^{m}$
for which   $\calO^{m}_X/\calV$ is also   locally free.
The tautological rank-$n$ vector bundle on $\Gr(n,m)$ is denoted $V(n,m)$
and  
its functor of points (as a $\Z$-scheme) is given by
\[
V(n,m)(X) = \{(v,\calV)\where \calV\in \Gr(n,m)(X), \, v\in \Gamma(X,\calV)\}.
\]
The morphism   $(v,\calV)\mapsto v:V(n,m)\to \calO_{\Gr(n,m)}^{m}$ 
makes $V(n,m)$ into a subbundle of $\calO_{\Gr(n,m)}^{m}$.
(That   $\calO_{\Gr(n,m)}^{m}/V(n,m)$ is locally free
can be checked using the standard affine covering of $\Gr(n,m)$,
see \cite[Tag \href{https://stacks.math.columbia.edu/tag/089R}{089R}]{stacks_project}.)
Dualizing the embedding $V(n,m)\to \calO_{\Gr(n,m)}^{m}$
gives an $\calO_X$-epimorphism 
\[p_0: \calO_{\Gr(n,m)}^{m}\to V(n,m)^\vee. \]
Let $e_1,\dots,e_{m}$ denote the standard basis
of $\calO_{\Gr(n,m)}^{m}$ and let $v^{(n,m)}_i$ be the image of $e_i$
in $\Gamma(\Gr(n,m),V(n,m)^\vee)$.
Then the  $m$-tuple $(v_1^{(n,m)},\dots,v_{m}^{(n,m)})$
generates $V(n,m)^\vee$; we call it the canonical tuple of generators
of $V(n,m)^\vee$.

\begin{lem}\label{LM:tautological-bundle-pullback}
	Let $X$ be a scheme and let $V\in \Gr(n,m)(X)$,
	namely, $V$ is a rank-$n$ subbundle of $\calO_X^{m}$.
	Let $f:X\to \Gr(n,m)$ be the morphism corresponding to $V$.
	Then there is canonical isomorphism $a_V:f^*V(n,m)\to V$
	that is characterized by the commutativity of the following diagram.
	\[
	\xymatrix{
	V \ar@{^{(}->}[r] &
	\calO_X^{m} \ar@{=}[d] \\
	f^*V(n,m) \ar@{.>}[u]^{a_V}   \ar[r]_{f^*(p_0^\vee)}
	&
	f^*\calO^m_{\Gr(n,m)}
	}
	\]
	In particular, $V=\im f^*(p_0^\vee)$.
\end{lem}

\begin{proof}
	This is  known, but we 
	include a proof since we were unable to find a satisfactory source.
	
	Since   $V\to \calO_X^{m}$
	is a monomorphism, $a_V$ is uniquely determined if it exists,
	and it is necessarily $\calO_X$-linear.
	Let $T$ be a  $\Z$-scheme.
	By  the Yoneda Lemma, the map $f_T:X(T)\to \Gr(n,m)(T)$
	induced by $f$ is given by $f_T(x)=x^*V$.
	(Here, $x:T\to X$ is a morphism. Note that $x^*V\subseteq x^*\calO_X^{m}=\calO_T^{m}$.)
	We have $f^*V(n,m)=X\times_{f,\Gr(n,m)} V(n,m)$ as  schemes,
	so $f^*V(n,m)(T)=\{(x,(\calV,v))\in X(T)\times V(n,m)(T)\suchthat
	x^*V=\calV\}$. 
	On the other hand, for every $\Z$-scheme
	$T$, we have $V(T)=\{(x,v)\where x\in X(T),\, v\in\Gamma(T,x^*V)\}$
	(cf.\ the proof of Lemma~\ref{LM:GLn-torsor-of-bundle}).
	Thus, the map $(x,(\calV,v))\mapsto (x,v):f^*V(n,m)(T)\to V(T)$
	determines a scheme isomorphism $a_V: f^*V(n,m)\to V$.
	It makes the diagram in the lemma commute because $f^*(p_0^\vee)$
	is given section-wise by $(x,(\calV,v))\mapsto (x,v)$.
\end{proof}

We now show that $(V(n,m)^\vee,v_1^{(n,m)},\dots,v_{m}^{(n,m)})$
is universal among the rank-$n$ vector bundles equipped with an $m$-tuple
of generators.

\begin{thm}\label{TH:universality-of-tautological}
	Let $m\geq n\geq 0$, let $X$ be a scheme and let $V$ be a rank-$n$ vector bundle over $X$.
	There is a one-to-one correspondence between the following:
	\begin{enumerate}[label=(\alph*)]
		\item tuples $(v_1,\dots,v_{m})\in\Gamma(X,V)^{m}$
		which generate $V$;
		\item pairs $(f,\vphi)$ consisting of a morphism
		$f:X\to \Gr(n,m)$ and an $\calO_X$-isomorphism
		$\vphi:f^*V(n,m)^\vee \to V$.
	\end{enumerate}
	The correspondence maps $(f,\vphi)$ to
	$(\vphi(f^*v_1^{(n,m)}),\dots,\vphi(f^*v_{m}^{(n,m)}))$.
\end{thm}

\begin{proof}
	The tuples in (a) are in one-to-one correspondence with
	\begin{enumerate}
		\item[(a$'$)] $\calO_X$-epimorphisms $p:\calO_X^{m}\to V$
	\end{enumerate}
	via sending $p$ to $(p(e_1),\dots,p(e_{m}))$, where
	$e_1,\dots,e_{m}$ is the standard basis of $\calO_X^{m}$.
	One readily checks that, given $(f,\vphi)$ as in (b), the corresponding map $p:\calO_X^{m}\to V$
	in (a$'$)  is $\vphi\circ f^*p_0$. It is therefore enough to construct an inverse
	to the assignment $(f,\vphi)\mapsto p(f,\vphi):=\vphi\circ f^*p_0$
	
	Let $p:\calO_X^{m}\to V$ be an $\calO_X$-epimorphism.
	Then $p^\vee:V^\vee\to\calO_X^{m}$ is an $\calO_X$-monomorphism.
	Let $W$ be the image of $p^\vee$; it is a rank-$n$ subbundle
	of $\calO_X^m$, and so it determines a morphism $f=f(p):X\to \Gr(n,m)$.
	The morphism $p^\vee$ factors as
	\[
	V^\vee\xrightarrow{ j} W\embeds \calO_X^m,
	\]
	and by Lemma~\ref{LM:tautological-bundle-pullback}, we can complete this into a commutative
	diagram:
	\begin{equation}\label{EQ:correspondence-diagram}
	\xymatrix{
	V^\vee \ar[r]|{j} \ar@{.>}[rd]_{\vphi^\vee} \ar@/^1.3pc/[rr]^{p^\vee} &
	W \ar@{^{(}->}[r] & \calO_X^m   \\
	&
	f^*V(n,m) \ar[u]|{a_{W}} \ar[ru]_{f^*(p_0^\vee)}
	}
	\end{equation}
	We define $\vphi^\vee:V^\vee\to f^*V(n,m)$ to be $a_W^{-1}\circ j$,
	and its dual is the desired isomorphism $\vphi=\vphi(p):f^*V(n,m)^\vee\to V$.
	We have therefore constructed a pair $(f,\vphi)$ as in (b).
	
	It remains to check that the assignments $p\mapsto (f(p),\vphi(p))$
	and $(f,\vphi)\mapsto \vphi\circ f^*p_0$ are mutually inverse.
	Given an $\calO_X$-epimorphism $p:\calO_X^m\to V$, it follows
	readily from the diagram \eqref{EQ:correspondence-diagram} that
	$p^\vee = f(p)^*(p_0^\vee)\circ \vphi(p)^\vee$,
	and therefore $p=\vphi(p)\circ f(p)^*p_0$.
	Conversely, suppose that $p=\vphi\circ f^*p_0$
	for  $(f,\vphi)$ as in (b). Then $p^\vee=f^*(p_0^\vee)\circ \vphi^\vee$.
	Since $\vphi^\vee$ is an isomorphism, $f^*(p_0^\vee): f^*V(n,m)^\vee \to \calO_X^m$ has the same image
	as $p^\vee:V^\vee\to \calO_X^m$,
	which in turn has the same image as $f(p)^*(p_0^\vee)$ by the definition of $f(p)$; call this joint image $W$. 
	By Lemma~\ref{LM:tautological-bundle-pullback}, both $f$ and $f(p)$
	correspond to $W\in \Gr(n,m)(X)$, so $f=f(p)$. Moreover,
	we have a commutative diagram:
	\[
	\xymatrix{
	V^\vee \ar[r]^{j} \ar[rd]_{\vphi^\vee} &
	W \ar@{^{(}->}[r] & \calO_X^m   \\
	&
	f^*V(n,m) \ar[u]|{j\circ (\vphi^\vee)^{-1}} \ar[ru]_{f^*(p_0^\vee)}
	}
	\]
	Lemma~\ref{LM:tautological-bundle-pullback}   tells us that
	$j\circ (\vphi^\vee)^{-1}$ coincides with $a_W$, so $\vphi^\vee =\vphi(p)^\vee$,
	meaning that $\vphi=\vphi(p)$. This completes the proof.
\end{proof}

Let $E(n,m)\to \Gr(n,m)$ denote the $\uGL_n(\Z)$-torsor corresponding to
$V(n,m)^\vee$.
We give two descriptions of $E(n,m)$
that will be needed in the sequel.

\begin{prp}\label{PR:sections-of-univ-bundle-over-Gr}
	Let $m\geq n\geq 0$ be integers. The functor of points of $E(n,m)$
	(as a $\Z$-scheme) is isomorhpic to the functor
	sending a  scheme $X$ to 
	\[
	\{a\in\nMat{\Gamma(X,\calO_X)}{n\times  m }
	\suchthat \text{the map $a:\calO^{m}_X\to \calO_X^n$ is surjective}\}.
	\]
	Under this isomorphism, the  right  action of $\uGL_n(\Z)$ on $E(n,m)$
	is given section-wise  by $(a, g)\mapsto  g^{-1} a$, and the map $E(n,m)\to \Gr(n,m)$
	is given  by sending $a\in\nMat{\Gamma(X,\calO_X)}{n\times m}$ 
	to the $\calO_X$-submodule of $\calO_X^{m}$
	spanned by the rows of $a$. 
\end{prp}

\begin{proof}
	Let $E'(n,m)$ denote the $\uGL_n(\Z)$-torsor corresponding to $V(n,m)$.
	By   Lemma~\ref{LM:GLn-torsor-of-bundle},  for a $\mathbb{Z}$-scheme
	$X$, the set $E'(n,m)(X)$ is in natural bijection with the set of tuples
	$(f,v_1,\dots,v_n )$ consisting 
	of a morphism $f:X\to \Gr(n,m)$
	and a basis $v_1,\dots,v_n$ for $f^*V(n,m)$.
	The morphism $f$ specifies
	a rank-$n$ subbundle $\calV$ of $\calO_X^{m}$
	and there is a canonical isomorphism $f^*V(n,m)\cong\calV $ 
	by Lemma~\ref{LM:tautological-bundle-pullback}.
	Thus, $E'(n,m)(X)$ can be naturally identified
	with the set of  tuples $(\calV,v_1,\dots,v_n)$
	consisting of a rank-$n$ subbundle $\calV$ of $\calO_X^{m}$
	and a basis $v_1,\dots,v_n\in\Gamma(X,\calO_X^{m})$ of $\calV$.
	Let $a\in \nMat{\Gamma(X,\calO_X)}{n\times  m }$ be the matrix 
	whose rows are $v_1,\dots,v_n$. Then $a$ determines and is determined by
	$(\calV,v_1,\dots,v_n)$. Moreover, by checking at the stalks,
	one sees that   the condition that $v_1,\dots,v_n$
	is a basis of a subbundle of $\calO_X^m$ 
	is equivalent to the condition that $a:\calO^{m}_X\to \calO_X^n$
	is an epimorphism. Thus, we have defined
	a natural isomorphism
	between $E'(n,m)(X)$ and the set of $\calO_X$-epimorphisms $a:\calO_X^m\to\calO_X^n$.
	Lemma~\ref{LM:GLn-torsor-of-bundle} also tells us that, under this isomorphism, 
	the right
	$\uGL_n(\Z)$-action on $E'(n,m)$ is given by $(a,g)\mapsto g^{\trans} a$ on sections.
	Lemma~\ref{LM:torsor-of-dual} then completes the proof.
\end{proof}

\begin{prp}\label{PR:action-on-universal-torsor-over-Gr}
	Let $m\geq n\geq 0$ be integers and put $d=m-n$.
	The $\uGL_n(\Z)$-torsor $E(n,m)\to\Gr(n,m)$ corresponding to $V(n,m)^\vee$
	is isomorphic to
	the quotient morphism
	\[
	\uGL_{m}(\Z)/\begin{bmatrix}
	1_{n\times n} & \bbA^{n\times d}_\Z \\
	0_{d\times n} & \uGL_d(\Z)
	\end{bmatrix} 
	\to 
	\uGL_{m}(\Z)/\begin{bmatrix}
	\uGL_n(\Z) & \bbA^{n\times d}_\Z \\
	0_{d\times n} & \uGL_d(\Z)
	\end{bmatrix} ,
	\]
	where here, $\uGL_n(\Z)$ acts on $\uGL_{m}(\Z)/[\begin{smallmatrix}
	1_{n\times n} & \bbA^{n\times d}  \\
	0_{d\times n} & \uGL_d(\Z)
	\end{smallmatrix} ] $ via
	$x\cdot g = x[\begin{smallmatrix} g & 0 \\ 0  & 1_{d\times d} \end{smallmatrix}]$.
\end{prp}

\begin{proof}
	We use the description of $E(n,m)$ in Proposition~\ref{PR:sections-of-univ-bundle-over-Gr}
	to define left actions of $\uGL_{m}(\Z)$   on $E(n,m)$ 
	and $\Gr(n,m)$  given  by  
	\[h\cdot a=a h^{\trans}\qquad
	\text{and}\qquad
	h\cdot \calV = h\calV  
	\]
	on sections.
	One readily checks that $E(n,m)\to\Gr(n,m)$ is     $\uGL_{m}(\Z)$-equivariant.
	Put $t_0=\begin{bmatrix}
	1_{n\times n} & 0_{n\times d}
	\end{bmatrix}\in E(n,m)(\Z)$
	and let $x_0$ be the image of $t_0$ in $\Gr(n,m)(\Z)$.
	It is routine to check that the stabilizers of $t_0$
	and $x_0$ in $\uGL_{m}(\Z)$
	are the subgroups $[\begin{smallmatrix}
	1_{n\times n} & \bbA^{n\times d}  \\
	0_{d\times n} & \uGL_d(\Z)
	\end{smallmatrix} ]$ and $[\begin{smallmatrix}
	\uGL_n(\Z) & \bbA^{n\times d}  \\
	0_{d\times n} & \uGL_d(\Z)
	\end{smallmatrix} ]$, respectively.
	Provided that the action of $\uGL_{m}(\Z)$
	on $E(n,m)$ 
	is fppf transitive, we have   isomorphisms $[h]\mapsto h\cdot t_0:\uGL_{m}(\Z)/[\begin{smallmatrix}
	1_{n\times n} & \bbA^{n\times d}  \\
	0_{d\times n} & \uGL_d(\Z)
	\end{smallmatrix} ]\to E(n,m)$
	and $[h]\mapsto h\cdot x_0:\uGL_{m}(\Z)/[\begin{smallmatrix}
	\uGL_n(\Z) & \bbA^{n\times d}  \\
	0_{d\times n} & \uGL_d(\Z)
	\end{smallmatrix} ]\to \Gr(n,m)$,
	and the proposition follows. 
	
	It remains to check that $\uGL_{m}(\Z)$ acts fppf transitively on $E(n,m)$.
	Let $R$ be a ring and let $a\in E(n,m)(R)$.
	Then $a:R^{m}\to R^n$ is surjective, and thus $\ker a$ is a summand of $R^m$. By passing to a Zariski covering
	of $\Spec R$, we may assume that $\ker a$ is a free $R$-module. Fix an isomorphism
	$c:R^d\to \ker a\subseteq R^{m}$ and an $R$-homomorphism $b:R^{n}\to R^{m}$ such that $ab=1_{n\times n}$.
	Then $h=\begin{bmatrix}
	b & c
	\end{bmatrix} 	 :R^{m}\to R^{m}$ lives in $\nGL{R}{m}$
	and satisfies $ah =t_0$, so $h^\trans \cdot a =t_0$.
\end{proof}

We finish with noting the following   consequence of Theorem~\ref{TH:universality-of-tautological}.

\begin{cor}\label{CR:equiv-to-GLn-torsors}
	With notation as above, let $V$ be a rank-$n$ vector bundle over a scheme $X$
	and let $E$ be the $\uGL_n(\Z)$-torsor corresponding to $V$.
	Then
	there is a one-to-one correspondence between:
	\begin{enumerate} 
		\item[(a)] tuples $(v_1,\dots,v_{m})\in\Gamma(X,V)^{m}$
		which generate $V$ and
		\item[(b$'$)] morphisms of 
		$\uGL_n(\Z)$-torsors from $E\to X$ to $E(n,m)\to \Gr(n,m)$. 
	\end{enumerate}
	The correspondence is compatible with pullbacks along morphisms $f:X'\to X$.
\end{cor}

\begin{proof}
	This is a consequence of Theorem~\ref{TH:universality-of-tautological},
	the equivalence between $\uGL_n(\Z)$-torsors and rank-$n$ vector bundles,
	and Remark~\ref{RM:morphism-pullback-correspondence}.
\end{proof}

\begin{remark}\label{RM:equiv-to-GLn-torsors}
	Using Proposition~\ref{PR:sections-of-univ-bundle-over-Gr}, one can
	see that Corollary~\ref{CR:equiv-to-GLn-torsors}
	is in fact a special case of \cite[Proposition~4.1]{First_2022_generators}.
	(In [op.\ cit.], the base scheme is a field and not $\Spec \Z$, and $X$ is assumed
	to be affine, but the   proof   works in our more general setting.)
	In fact, we could use Proposition~\ref{PR:sections-of-univ-bundle-over-Gr} together
	with \cite[Proposition~4.1]{First_2022_generators} to prove
	Theorem~\ref{TH:universality-of-tautological}, but this would have obscured
	the elementary nature of the correspondence in the theorem.
\end{remark}

\section{The Number of Generators of a Globally Generated Vector Bundle}
\label{sec:generators}

In the last section, we saw
that for every $m\geq n\geq 0$,    there are vector bundles
which specialize to every rank-$n$ vector bundle that is generated
by $m$ global sections. On the other hand, a theorem of F\"orster \cite{Forster_1964_number_of_generators} says that
every rank-$n$ vector bundle over an \emph{affine} noetherian scheme of dimension $\leq d$
can be generated by $n+d$ elements. 
By coupling these two observations, we see that
the vector bundles $V_{n,n+d}$ and $V(n,n+d)^\vee$ considered in 
Section~\ref{sec:generic}  
specialize  to every rank-$n$ vector bundle over an affine noetherian  scheme of dimension $\leq d$,
that is, they are weakly versal for the class $\nAff_d$ of affine noetherian schemes of dimension $\leq d$.
The same holds for the corresponding $\uGL_n(\Z)$-torsors.
Thanks to Propositions~\ref{PR:versality-under-base-change} and~\ref{PR:reduction-to-overgroup}, it follows that every 
flat locally of finite presentation linear group scheme $G$ over a scheme $S$ admits a torsor
$E\to X$ (with $X$ an algebraic space over $S$) that is weakly versal
for $\nAff_d/S$.

In this section, we shall strengthen F\"orster's Theorem in   three ways
that will allow us to construct   torsors with stronger versality properties
in the next section. 
Our arguments stem from the   generalizations
of F\"orster's Theorem appearing in 
\cite{Swan_1967_num_of_generators_of_module},
\cite{First_2017_number_of_generators}
and \cite{First_2022_generators}.

\medskip

Throughout this section, if $R$ is a ring, $M$ is an $R$-module,
and $\frakp\in\Spec R$,
then we let $M(\frakp)=M\otimes_R k(\frakp)$
and denote  the image of $m\in M$ in $M(\frakp)$ by $m(\frakp)$.
If we set $X=\Spec R$ and let $\calM$ be the quasi-coherent $\calO_X$-module
corresponding to $M$, then $M(\frakp)$ is just $\calM(\frakp)$.
For a general scheme $X$ and a coherent $\calO_X$-module $\calM$,
Nakayama's lemma tells us that $m_1,\dots,m_t\in \Gamma(X,\calM)$
generate $\calM$ if and only if $m_1(x),\dots,m_t(x)$ generate $\calM(x)$
for every $x\in \clpnt{X}$.
We denote the restriction of $\calM$ to an open subset $U\subseteq X$ by $\calM_U$.

\begin{lem}\label{LM:insteaf-of-CRT}
	Let $R$ be a ring and let $M$ be an
	$R$-module. Let $S$ be a finite subset of $\Spec R$
	and let $m_1,\dots,m_t\in M$.
	Suppose that for every $\frakp\in S$, the elements
	$m_1(\frakp),\dots,m_t(\frakp)$ do not generate
	$M(\frakp)$ as a $k(\frakp)$-vector space.
	Then there exists $m_{t+1}\in M$
	such that for every $\frakp\in S$, 
	we have 
	\[\dim \Span_{k(\frakp)}\{m_1(\frakp),\dots,m_{t+1}(\frakp)\}=
	\dim \Span_{k(\frakp)}\{m_1(\frakp),\dots,m_{t}(\frakp)\}+1.\]
\end{lem}

\begin{proof}
	The proof is inspired by the proof of \cite[Lemma~1]{Swan_1967_num_of_generators_of_module}.
	We prove the lemma by induction on $|S|$.
	The case $S=\emptyset$
	is clear.
	Suppose  that $S\neq\emptyset$. Then there is some minimal element $\frakq\in S$.
	Put $T=S-\{\frakq\}$. By the induction hypothesis, there is
	$x\in M$ such that 
	\begin{align}\label{EQ:new-generator}
	\dim \Span_{k(\frakp)}\{m_1(\frakp),\dots,m_{t}(\frakp),x(\frakp)\}=
	\dim \Span_{k(\frakp)}\{m_1(\frakp),\dots,m_{t}(\frakp)\}+1	
	\end{align} 	
	for all $\frakp\in T$.
	If $x(\frakq)\notin \Span_{k(\frakq)}\{m_1(\frakq),\dots,m_t(\frakq)\}$,
	then we can take $m_{t+1}=x$. Assume henceforth that
	$x(\frakq)\in  \Span_{k(\frakq)}\{m_1(\frakq),\dots,m_t(\frakq)\}$.
	Since $\frakq$ is minimal in $S$ and prime,
	we have  $ \prod_{\frakp\in T} \frakp\nsubseteq \frakq$.
	Choose some $r\in \prod_{\frakp\in T} \frakp-\frakq$.
	Then replacing $x$ with $x+ry$ for some $y\in M$
	will not affect the validity of \eqref{EQ:new-generator} for all $\frakp \in T$.
	Since the image of $M$ in $M(\frakq)$ generates $M(\frakq)$ as a $k(\frakq)$-module,
	and since $m_1(\frakq),\dots,m_t(\frakq)$ do not generate
	$M(\frakq)$ as a $k(\frakp)$-vector space,
	there is some $y\in M$ such that 
	$y(\frakq)\notin \Span_{k(\frakq)}\{m_1(\frakq),\dots,m_{t}(\frakq)\}$.
	Put  $m_{t+1}:=x+ry$. Our assumptions on $x$ and $r$
	imply that 
	$m_{t+1}(\frakq)\notin \Span_{k(\frakq)}\{m_1(\frakq),\dots,m_{t}(\frakq)\}$. 
	We therefore conclude that 	
	$\dim \Span_{k(\frakp)}\{m_1(\frakp),\dots,m_{t+1}(\frakp)\}=
	\dim \Span_{k(\frakp)}\{m_1(\frakp),\dots,m_{t}(\frakp)\}+1$
	for all $\frakp\in T\cup\{\frakq\}=S$.
\end{proof}

The following is our first strengthening of F\"orster's Theorem.

\begin{thm}\label{TH:fr-vec-bundles}
	Let $X$ be the spectrum of a noetherian   ring, 
	and let $\calM$ be a coherent 
	$\calO_X$-module of rank $\leq n$ ($n\in \N\cup \{0\}$). 
	Then, for every $d\in \N\cup\{0\}$,
	there exists an open subscheme $U\subseteq X$
	such that $\codim(X-U,X)>d$  
	and $\calM_U$ is generated by $d+n$ global sections that can be extended
	to all of $X$.
\end{thm}

\begin{proof}
	This is a modification of the proof of \cite[Theorem~1.2]{First_2017_number_of_generators}.
	We give a complete proof in order to introduce
	a variation later on, and also since  it is simpler and clearer than 
	listing the necessary changes.

	We shall prove  that for every
	$j\in\N\cup \{0\}$, there exist $a_1,\dots,a_j\in \Gamma(X,\calM)$ 
	and a sequence of closed
	subsets $\emptyset=Y^{(j)}_{-1}\subseteq Y^{(j)}_0\subseteq Y^{(j)}_1\subseteq
	\dots\subseteq Y^{(j)}_n=X$  such that the following
	conditions hold:\footnote{
		The closed set  $Y^{(j)}_i$
		corresponds to the union of the locally closed sets
		$\calF^{(j)}_0,\dots,\calF^{(j)}_i$ in \cite[Theorem~1.2]{First_2017_number_of_generators}.
		Note, however,
		that 
		here
		we consider all the points of $X$,
		whereas in [op.\ cit.] we consider
		only the closed points.
		The reason 
		for this is that, for a closed subset $Z\subseteq X$, one cannot bound
		from below
		$\codim(Z,X)$ by means of $\codim(\clpnt{Z},\clpnt{X})$.
	}
	\begin{enumerate}[label=(\arabic*)]
		\item For every $i\in\{0,\dots,n\}$ and $x\in Y^{(j)}_i-Y^{(j)}_{i-1}$,
		there are $b_{i+1},\dots,b_n\in \calM(x)$
		such that $a_1(x),\dots,a_j(x),b_{i+1},\dots,b_n$
		span $\calM(x)$ as a $k(x)$-vector space. 
		\item 
		$\codim(Y^{(j)}_i,X)\geq j-i$   for every
		$i\in\{0,\dots,n-1\}$.	
	\end{enumerate}
	Provided these data have been constructed, 
	we can take $U=X-Y^{(n+d)}_{n-1}$.
	Indeed, $U$ is open and condition (2) 
	says that $\codim(X-U,X)>d$.
	Moreover, by (1), for every $x\in U$,
	the elements $a_1(x),\dots,a_{n+d}(x)$
	generate  
	$\calM(x)=\calM_U(x)$, so $\calM_U$ is generated 
	by the restriction of $a_1,\dots,a_{n+d}$ to $U$.
	
	We  prove the existence of the $a_j$
	and the
	$Y^{(j)}_i$ by induction on $j$.
	For the case $j=0$, take $Y^{(0)}_0=\dots=Y^{(0)}_{n-1}=X$.
	
	Suppose   that $j>0$ and $a_1,\dots,a_j$
	and $Y^{(j)}_0,\dots,Y^{(j)}_n$
	were constructed. 
	Fix $i\in\{0,\dots,n-1\}$.
	For   every irreducible component $X'$ of $X$
	and every irreducible component
	$Y'$ of $Y^{(j)}_i\cap X'$ that is not contained in $Y^{(j)}_{i-1}$,
	choose a point $x\in Y'-Y^{(j)}_{i-1}$, and let $S_i$ denote the finite set
	of points chosen in this manner.
	Note that $S_0,\dots,S_{n-1}$
	are pairwise disjoint. 
	
	By condition (1) and Lemma~\ref{LM:insteaf-of-CRT},
	there exists $a_{j+1}\in \Gamma(X,\calM)$
	with the property
	that, for every $i\in \{0,\dots,n-1\}$
	and every $x\in S_i$, there exist $b_{i+2},\dots,b_n\in\calM(x)$
	such that $a_1(x),\dots,a_{j+1}(x),b_{i+2},\dots,b_n$
	span $\calM(x)$ as a $k(x)$-vector space.
	Moreover, by \cite[Lemma~2.3]{First_2017_number_of_generators},
	$x$ admits  an open neighborhood $U_x$
	such that this condition holds for every point in $U_x$.

	Let $i\in \{0,\dots,n-1\}$ and let
	$U_i=(\bigcup_{x\in S_i} U_x)- Y^{(j)}_{i-1}$.
	Then $U_i$ is an open neighborhood of $S_i$ not meeting $Y^{(j)}_{i-1}$.
	Define 
	\[
	Y^{(j+1)}_i=Y^{(j)}_i-U_i.
	\]
	Since $U_i\cap Y^{(j)}_{i-1}=\emptyset$, we have 
	$Y^{(j+1)}_i= Y^{(j)}_i-U_i\supseteq Y^{(j)}_{i-1}\supseteq Y^{(j+1)}_{i-1}$,
	so
	$Y^{(j+1)}_{0}\subseteq \dots\subseteq 
	Y^{(j+1)}_{n-1}$. We set $Y^{(j+1)}_{-1}=\emptyset$
	and $Y^{(j+1)}_n=X$. 
	
	We finish by
	checking that (1) and (2) hold
	for the filtration $Y^{(j+1)}_{-1}\subseteq Y^{(j+1)}_{0}\subseteq \dots\subseteq 
	Y^{(j+1)}_{n}$ and $a_1,\dots,a_{j+1}$.
	For every $i\in \{0,\dots,n\}$, 
	we have $Y^{(j+1)}_i-Y^{(j+1)}_{i-1}\subseteq  (Y^{(j)}_i-Y^{(j)}_{i-1}) \cup U_{i-1}$
	(with the convention $U_{-1}=\emptyset$),
	so (1) holds by the induction hypothesis and the construction of $U_{i-1}$.
	To prove (2), let $X'$ be an irreducible component of $X$
	and let $Y'$ be an irreducible component of $Y^{(j+1)}_i\cap X'$.
	If $Y'\subseteq Y^{(j)}_{i-1}$, then $Y'$ is contained in some irreducible
	component $Y''$ of $Y^{(j)}_{i-1}\cap X'$
	and we get $\codim(Y',X')\geq \codim(Y'',X')\geq j-(i-1)=(j+1)-i$.
	Otherwise, there is an irreducible component $Y''$ of $Y^{(j)}_i\cap X'$
	containing $Y'$. Since $U_i$ meets $Y''-Y^{(j)}_{i-1}$ and does not meet $Y'-Y^{(j)}_{i-1}$,
	we must have $Y'\subsetneq Y''$, so $\codim(Y',X')\geq \codim(Y'',X')+1\geq
	j-i+1=(j+1)-i$. This completes the proof.
\end{proof}

We will also need the following variant of Theorem~\ref{TH:fr-vec-bundles}.
The case $Z=\emptyset$ is a theorem of Swan \cite[Theorem~1]{Swan_1967_num_of_generators_of_module}.

\begin{thm}\label{TH:fr-variant}
	Let $X$ be an affine   scheme such that
	the subspace $\clpnt{X}$ is noetherian and has finite dimension 
	$\leq d$,
	and let $\calM$ be a coherent $\calO_X$-module of rank $\leq n$
	($n\in \N\cup \{0\}$).
	Let $u:Z\to X$ be a closed immersion, let $\calM_Z=u^*\calM$
	and let $b_1,\dots,b_{n+d}\in \Gamma(Z,\calM_Z)$ be sections
	which generate $\calM_Z$.
	Then there exist $a_1,\dots,a_{n+d}\in \Gamma(X,\calM)$
	which generate $X$ and satisfy $u^*a_i=b_i$ for all $i\in\{1,\dots,d+n\}$.	
\end{thm}

\begin{proof}
	We may assume that $Z$ is a closed subscheme of $X$.
	Give $X_1:=\clpnt{X}-Z$ the topology induced from $X$.
	We claim that for every $j\in\{0,\dots,n+d\}$,
	there is a filtration of $X_1$ by closed
	subsets $\emptyset = Y_{-1}^{(j)}\subseteq Y_{0}^{(j)}\subseteq\dots\subseteq
	Y_n^{(j)}=X_1$ and $a_1,\dots,a_j\in \Gamma(X,\calM)$ such that:
	\begin{enumerate}[label=(\arabic*)]
		\item For every $i\in\{0,\dots,n\}$ and $x\in Y^{(j)}_i-Y^{(j)}_{i-1}$, 
		there are $a'_{i+1},\dots,a'_n\in \calM(x)$
		such that  $a_1(x),\dots,a_j(x),a'_{i+1},\dots,a'_n$ generate
		$\calM(x)$ as a $k(x)$-vector space.
		\item 
		$\codim(Y^{(j)}_i,X)\geq j-i$		
		for every $i\in\{0,\dots,n-1\}$.
		\item  
		$u^*a_\ell= b_\ell$ for all $\ell\in\{1,\dots,j\}$.
	\end{enumerate}
	Once the above data have been constructed, $a_1,\dots,a_{n+d}$
	satisfy all the requirements. 
	Indeed,  condition (2) implies that $Y^{(d+n)}_{n-1}=\emptyset$,
	so by 
	condition (1), the elements   $a_1(x),\dots,a_{n+d}(x)$
	generate $\calM(x)$ for all $x\in X_1=\clpnt{X}-Z$. On the other hand, for
	all $x\in \clpnt{X}\cap Z$, the elements $a_1(x),\dots,a_{n+d}(x)$
	are   $b_1(x),\dots,b_{n+d}(x)$ by (3), so they generate
	$\calM(x)$ by assumption.
	
	That $a_1,\dots,a_j$ and $Y^{(j)}_{-1},\dots,Y^{(j)}_n$
	exist is shown exactly as in the proof of Theorem~\ref{TH:fr-vec-bundles},
	but with the following difference: Instead of using Lemma~\ref{LM:insteaf-of-CRT}
	to choose $a_{j+1}$, we use the Chinese Remainder Theorem 
	to choose $a_{j+1}\in\Gamma(X,\calO_X)$
	such that $u^* a_{j+1}=b_{j+1}$
	and such that $a_1(x),\dots,a_{j+1}(x)$ generate an $(i+1)$-dimensional subspace
	of $\calM(x)$ for every $x\in S_i$ and $i\in\{0,\dots,n-1\}$.
\end{proof}

\begin{remark}
	Similarly to \cite[Theorem~1.2]{First_2017_number_of_generators},
	Theorems~\ref{TH:fr-vec-bundles} and~\ref{TH:fr-variant} remain true if  we replace
	the $\calO_X$-module $\calM$ with a coherent $\calO_X$-algebra (or even an $\calO_X$-\emph{multialgebra}),
	and interpret ``generation'' as generation of $\calO_X$-algebras.
	The proof is the same, except one needs to replace Lemma~\ref{LM:insteaf-of-CRT}
	with a slightly more involved  generalization to algebras. We omit the details
	as they will not be needed here.
\end{remark}

Our third variant of F\"orster's Theorem
applies to vector bundles over \emph{general}
schemes \emph{over an infinite field}.
The idea originates
from \cite[Proposition~5.2]{First_2022_generators}.
We begin with a consequence of   Kleiman's Transversality Theorem
\cite[Theorem~2]{Kleiman_1974_transversality_theorem}.

\begin{lem}\label{LM:generic-intersection}
	Let $k$ be a  field, let $X$, $X_1$, $X_2$ be   $k$-schemes of finite type,
	and let $f_i:X_i\to X$ ($i=1,2$) be a $k$-morphism.
	Let
	$H$ be a connected algebraic group over $k$ acting fppf transitively on $X$. 
	View $X_1$ as an $X$-scheme via $f_1$
	and, given $h\in H(k)$, write $hX_2$
	for $X_2$ viewed as an $X$-scheme via $h\circ f_2$ (here, $h$ also denotes
	the automorphism of $X$ it induces).
	Then there exists a  nonemtpy open subscheme $U$ of $H$
	such that for all  $h\in U(k)$, we have
	\[
	\dim X_1 \times_X hX_2
	\leq \dim X_1+\dim X_2-\dim X. 
	\]
\end{lem}

\begin{proof}
	If $X$, $X_1$ and $X_2$ are integral and $k$ is algebraically
	closed, then this is a consequence of Kleiman's Transversality Theorem.
	We reduce the lemma to the theorem by an argument
	similar to the proof of  \cite[Claim on p.~7289]{First_2022_generators}.
	
	Consider the closed subscheme $I$ of $H\times X_1\times X_2$
	cut by the equation $ f_1(x_1)= h\cdot f_2(x_2)$, and let $p:I\to H$
	be the restriction of the first projection.
	For $h\in H(k)$, the scheme-theoretic fiber of $p$ over
	$h$, denoted $I_h$, is precisely $X_1\times_X hX_2$.
	Let $Z=\{h\in H\suchthat \dim I_h>\dim X_1+\dim X_2-\dim X\}$
	and put $U=U_{X_1,X_2}=H-\quo{Z}$.
	Then $U$ satisfies
	all the requirements if it is nonempty.
	
	Let $\quo{k}$ be an algebraic closure of $k$,
	and let $q$ denote the morphism $H_{\quo{k}}:=H\times_k \quo{k}\to H$.
	We claim that $q$ is closed and open.
	Indeed, $q$ is closed because it is integral \cite[Tag \href{https://stacks.math.columbia.edu/tag/01WM}{01WM}]{stacks_project}. To see that $q$ is open, observe
	that every open   $U'\subseteq H_{\quo{k}}$
	arises as $U''\times_{\ell}\quo{k}$
	for some finite $k$-field $\ell\subseteq \quo{k}$
	and open $U''\subseteq H_{\ell}$, and the image of $U''$
	in $H$ is open because $H_\ell\to H$ is flat and locally of finite presentation.	
	 
	Next, observe that $X$ and $X_{\quo{k}}:=X\times_k \quo{k}$ have the same
	dimension, and likewise  for $X_1$ and $X_2$.
	Put $Z'=\{h\in H_{\quo{k}}\suchthat \dim (I_{\quo{k}})_h>\dim X_1+\dim X_2-\dim X\}$.
	By \cite[Tag  \href{https://stacks.math.columbia.edu/tag/02FY}{02FY}]{stacks_project},
	we have $Z'=q^{-1}(Z)$. Since $q$ is continuous and closed, this means
	that $q(\quo{Z'})=\quo{q(Z')}=\quo{Z}$. Suppose now that there is $h'\in H_{\quo{k}}-\quo{Z'}$.
	If it were the case that $q(h')\in \quo{Z}$,
	then   we would have $q(h')\in q(H_{\quo{k}}-\quo{Z'})\cap \quo{Z}$.
	Since $q$ is open, this would mean that $q(H_{\quo{k}}-\quo{Z'})\cap Z\neq\emptyset$,
	which contradicts with $q^{-1}(Z)=Z'$. Thus, $q(h')\notin \quo{Z}$.
	This means that in order to prove that $H-\quo{Z}\neq\emptyset$,  
	it is enough to prove   that $H_{\quo{k}}-\quo{Z'}\neq\emptyset$.
	To that end, we may replace $X,X_1,X_2$ and $H$ with their base change
	along $\Spec \quo{k}\to \Spec k$ and assume that $k=\quo{k}$.
	(Note that $H$ remains connected.)

	Now that $k=\quo{k}$, the scheme $H_{\mathrm{red}}$ is a subgroup
	of $H$ \cite[Corollary~1.39]{Milne_2017_algebraic_groups}, and
	the action of $H$ on $X$ restricts
	to an action of	
	$H_{\mathrm{red}}$ on $X_{\mathrm{red}}$ \cite[Proposition~2.5.1(1)]{Brion_2017_alg_groups}.
	We may therefore replace $X,X_1,X_2$ and $H$ with their reductions.
	Since $H$ is connected, it is irreducible,
	and since 
	and $H(k)$ acts transitively on $X(k)$,
	it follows that   $X$ is also irreducible. 
	
	Let $\{X_{1i}\}_{i=1}^s$ and $\{X_{2j}\}_{j=1}^t$ be the irreducible
	components of $X_1$ and $X_2$, respectively.
	Then $U_{X_1,X_2}\supseteq \bigcap_{i,j} U_{X_{1i},X_{2j}}$. Since
	$X$ is irreducible,    it is enough
	to show that $U_{X_1,X_2}\neq \emptyset$ when $X_1$ and $X_2$ are irreducible.
	We are now at the setting of Kleiman's Transversality Theorem (i.e.,
	$H$, $X$, $X_1$, $X_2$ are integral and $k=\quo{k}$),
	which gives the desired conclusion   $U_{X_1,X_2}\neq 0$.
\end{proof}

Recall from Section~\ref{sec:generic} that $E(n,m)$ ($m\geq n\geq 0$) denotes
the $\uGL_n(\Z)$-torsor associated to the dual of the tautological rank-$n$
bundle $V(n,m)$ over $\Gr(n,m)$.
Given a field $k$, we write $\Gr(n,m)_k$ for the base-change
of $\Gr(n,m)$ along $\Spec k\to \Spec\Z$, and likewise
for $V(n,m)$ and $E(n,m)$.
By Proposition~\ref{PR:sections-of-univ-bundle-over-Gr},
for a ($\Z$-)scheme $X$, we may view $E(n,m)(X)$
as the collection of matrices $a=(a_{ij})_{i,j}\in \nMat{\Gamma(X,\calO_X)}{n\times m}$
for which $a:\calO_X^{m}\to \calO_X^n$ is surjective.

\begin{lem}\label{LM:Grassmannian-properties}
	Let 
	$r\geq m\geq n\geq 0$ be integers.	
	With notation as above, let 
	$Z =Z(n,m,r)$ denote the closed subscheme of $E(n,r)$
	cut by the equations 
	\[\det((a_{ij})_{i\in\{1,\dots,n\},j\in I})=0\]
	as $I$ ranges over the   subsets of $\{1,\dots,m\}$ with $|I|=n$.
	Then:
	\begin{enumerate}[label=(\roman*)]
		\item $\dim Z_k\leq  nr-(m-n+1)$.\footnote{
			In fact, equality holds, but we do not need that.		
		}
		\item $Z /\uGL_n(\Z)$ is a scheme and $Z \to Z /\uGL_n(\Z)$ is a $\uGL_n(\Z)$-torsor.
		\item There is a $\uGL_n(\Z)$-equivariant morphism
		$(E(n,r)-Z)\to E(n,m)$ given by $(a_{ij})_{i\in\{1,\dots,n\},j\in\{1,\dots,r\}}\mapsto
		(a_{ij})_{i\in\{1,\dots,n\},j\in\{1,\dots,m\}}$ on sections.
	\end{enumerate}
\end{lem}

\begin{proof}
	(i) 
	Let $W=W(n,m,r)$ be the closed
	subscheme of $\bbA^{n\times r}_\Z$  
	cut by the same equations as those defining $Z$.
	Then $Z$ is an open subscheme of $W$.
	It is therefore enough to prove that $\dim W_k\leq  nr-(m-n+1)$.
	Since $W\cong W(n,m,m)\times  \mathbb{A}^{n\times (r-m)}_{\Z}$,
	we further reduce into proving that
	$\dim W(n,m,m)_k\leq nm-(m-n+1)$,
	and this well-known, e.g.,   see  \cite[Lemma~6.2(b), Example~6.4]{First_2022_generators}.

	(ii) 
	This can be checked directly using the standard affine covering of $\Gr(n,r)$,
	see  \cite[Tag \href{https://stacks.math.columbia.edu/tag/089T}{089T}]{stacks_project}, but we can
	give a quicker (albeit less elementary) proof using results    from
	Section~\ref{sec:torsors}.
	Since $Z$ is a closed subscheme 
	of $E(n,r)$, the group $\uGL_n(\Z)$
	acts freely on $Z$.
	Thus,
	by Lemma~\ref{LM:free-action-quotient}, $Z/\uGL_n(\Z)$ is a 
	$\Z$-algebraic space   and $Z\to Z/\uGL_n(\Z)$ 
	is a $\uGL_n(\Z)$-torsor.
	The induced morphism $Z/\uGL_n(\Z)\to E(n,r)/\uGL_n(\Z)=\Gr(n,r)$
	is a closed immersion because its pullback along the fppf covering $E(n,r)\to \Gr(n,r)$
	is the closed immersion $Z\to E(n,r)$
	\cite[Tag \href{https://stacks.math.columbia.edu/tag/0420}{0420}]{stacks_project}
	(see also \cite[Tag \href{https://stacks.math.columbia.edu/tag/02YU}{02YU}]{stacks_project}).
	In particular, the morphism of algebraic spaces 
	$Z/\uGL_n(\Z)\to  \Gr(n,r)$ is representable, so $Z/\uGL_n(\Z)$
	is a scheme.
	
	(iii) Let $X$ be a  scheme and let $a=(a_{ij})$
	be an $X$-section of $E(n,r)-Z$. We need
	to show that $\tilde{a}:=(a_{ij})_{i\in\{1,\dots n\},j\in\{1,\dots,m\}}:\calO_X^{m}\to \calO_X^n$
	is surjective. It is enough to check this after specializing to all the residue fields of $X$.
	We may therefore assume that $X=\Spec k$ for a field $k$
	and view $a$ as an $n\times r$ matrix over $k$.
	Now, since the   unique point of $X$ is not mapped to $Z$ under the morphism $X\to E(n,r)$,
	there is some $I\subseteq\{1,\dots,m\}$ with $\det((a_{ij})_{i,j\in I})\neq 0$,
	so $\tilde{a}:k^{m}\to k^n$ is surjective. 
\end{proof}

\begin{thm}\label{TH:generators-bound-any-scheme}
	Let $k$ be an infinite field, let $d\in\N\cup\{0\}$, let $X$ be a finite type $k$-scheme,
	and let $V$ be a rank-$n$ vector bundle over $X$.
	If $V$ is generated by   global sections, then 
	there is an open subscheme $U$ of $X$
	such that $\codim(X-U,X)>d$  
	and $V|_U$ is generated by $n+d$ global sections.
	In particular, if $\dim X=d$, then $V$ is generated by $n+d$ global sections.
\end{thm}

\begin{proof}
	Since $X$ is quasi-compact and $V$ is coherent,
	$V$ is generated by finitely many global sections.
	Say it is generated by $r$ global sections with $r\geq n$.
	If $r\leq n+d$, we can take $U=X$ and finish, so assume $r>n+d$.

	Let $p:E\to X$ denote the $\uGL_n(k)$-torsor corresponding to $V$.
	By Corollary~\ref{CR:equiv-to-GLn-torsors}, the $r$ global sections generating
	$V$ give rise to a
	$\uGL_n(k)$-torsor morphism
	$(\hat{f},f)$ from 
	$E\to X$ to $E(n,r)_k\to \Gr(n,r)_k$.
	Let $Z=Z(n,n+d,r)$ be as in Lemma~\ref{LM:Grassmannian-properties},
	let $Y=Z/\uGL_n(\Z)$,
	and let $(\hat{j},j)$ denote the evident $\uGL_n(k)$-torsor morphism
	from $Z_k\to Y_k$
	to $E(n,r)_k\to \Gr(n,r)_k$.
	Then $\hat{j}$ is a closed embedding, and therefore so is $j$
	\cite[Tag \href{https://stacks.math.columbia.edu/tag/02L6}{02L6}]{stacks_project}.
	Proposition~\ref{PR:action-on-universal-torsor-over-Gr} implies
	that $\uGL_{r}(k)$ acts on $E(n,r)_k\to \Gr(n,r)_k$
	and the action is fppf transitive on both source and target.
	
	Let $X'$ be an 	irreducible component of $X$
	and let $f'$ be the restriction of $f:X\to \Gr(n,r)_k$
	to $X'$.
	By Lemma~\ref{LM:generic-intersection} (applied with $H=\uGL_{r}(k)$,
	$X_1=Y_k$ and $X_2=X'$) there is a nonempty open subscheme $W'$ of $\uGL_{r}(k)$
	such that for all $h\in W'(k)$
	we have  
	\begin{align*} 
	\dim Y_k\times_{\Gr(n,r)_k} hX'&\leq \dim Y_k+\dim X' -\dim\Gr(n,r)_k\\
	&\leq [nr-(d+1)-n^2]+\dim X'-n(r-n)=\dim X'-d-1,
	\end{align*}
	where in the second inequality we used Lemma~\ref{LM:Grassmannian-properties}(i)
	and the fact that $\dim Z_k=\dim Y_k+\dim\uGL_n(k)$.
	Let $W$ denote the (finite) intersection of all the $W'$ as $X'$ ranges over the irreducible
	components of $X$. Then $W$ is a nonempty open subscheme of $\uGL_{r}(k)$.
	As $k$ is infinite,
	there exists $h\in W(k)$.

	Put $X_1=Y_k\times_{\Gr(n,r)_k} hX$. Since $j:Y\to \Gr(n,r)$
	is a closed embedding, we may and shall view $X_1$ as a closed subscheme of $X$.
	By the previous paragraph, we have $\dim(X_1\cap X')\leq \dim X'-d-1$ for every
	irreducible component $X'$ of $X$, so $\codim(X_1,X)>d$. Put $U=X-X_1$
	and
	let $\hat{g}=h\circ \hat{f}:E\to E(n,r)_k$. Then $\hat{g}$ maps $E|_U$
	to $E(n,r)_k-Z_k$. By Lemma~\ref{LM:Grassmannian-properties}(iii),
	we now have a well-defined $\uGL_n(k)$-equivariant morphism
	$E|_U\to E(n,r)_k-Z_k\to E(n,n+d)_k$.
	Applying Corollary~\ref{CR:equiv-to-GLn-torsors} again,
	we conclude   that $V|_U$ is generated by $n+d$ global sections.
\end{proof}

\begin{example}\label{EX:generators-for-Om}
	Let $k$ be an infinite field and let $n\in\N$.
	As usual, we let $\calO(m)=\calO_{\bbP^n_k}(m)$ ($m\in\Z$) denote the line bundle on $\bbP^n_k$
	whose   sections on an open     $U\subseteq \bbP^n_k$ 
	are the degree-$m$ homogeneous regular functions
	on the cone of $U$ in $\bbA^{n+1}_k$.
	Let  $x_0,\dots,x_n$ denote the coordinate functions of $\bbA^{n+1}_k$.
	It is well-known that $\calO(1)$ is generated by the global sections
	$x_0,\dots,x_n$, so for every $m\geq 0$,
	the line bundle $\calO(m)\cong \calO(1)^{\otimes m}$
	is generated by the ${m+n}\choose n$ global sections $\{x_{i_1}\cdots x_{i_m}\where 0\leq i_1\leq \dots\leq i_m\leq n\}$.
	However, Theorem~\ref{TH:generators-bound-any-scheme} tells us that $\calO(m)$ can be generated
	by $\dim \bbP^n_{k}+\rank \calO(m)=n+1$ global sections,
	and
	indeed,   $x_0^m,\dots,x_n^m$ generate
	$\calO(m)$ as an $\calO_{\bbP^n_k}$-module. This can be shown, for instance,
	by checking it at the stalks, or by looking at the graded module associated
	to $\calO(m)$. 
	(Note, however, that   $x_0^m,\dots,x_n^m$ do \emph{not}
	generate $\Gamma(\bbP^n_k,\calO(m))$ as a $\Gamma(\bbP^n_k,\calO_{\bbP^n_k})$-module
	if $m>1$.)
\end{example}

\section{$d$-Versal Torsors: Existence}
\label{sec:existence}

We now put together the results of the last three sections to show
that strongly $d$-versal torsors exist for all linear groups. We give three constructions
with varying additional properties. In particular, we show that:
\begin{enumerate}
	\item Every linear reductive group scheme over an affine
	noetherian scheme $S$ admits a torsor over a smooth affine base that is strongly versal
	for the class of affine $S$-schemes of dimension $\leq d$.
	\item Every linear reductive group scheme over a
	noetherian scheme $S$ admits a torsor over a scheme that is strongly $d$-versal
	for all affine $S$-schemes.
	\item Every affine algebraic group over an infinite field $k$
	admits a torsor over a smooth quasi-projective $k$-scheme that is strongly $d$-versal
	for all quasi-projective $k$-schemes.
\end{enumerate}
Consult Theorem~\ref{TH:Thomason} and \cite{Gille_2022_when_is_reductive_grp_sch_linear}
regarding the question of which reductive group schemes are linear.
The torsors in (2) also have an additional property
which allows one to extend   specializations of these torsors
from a closed subscheme to the ambient scheme, see Theorem~\ref{TH:extension}.

Recall that if $\catC$ is a class of $S$-schemes and $S_1$ is an $S$-scheme, then  
$\catC/S_1$   denotes the class of $S_1$-schemes with underlying $S$-scheme in $\catC$.

\begin{thm}\label{TH:highly-versal-I}
	Let $n,d\in\N\cup\{0\}$, let $S$ be a scheme,
	and let $G$ be a subgroup of $\uGL_n(S)$ that is flat and locally of finite presentation
	over $S$.
	Let $\catC_d$ denote the class of affine   schemes $Y$ 
	such that $\clpnt{Y}$ is noetherian and $\dim \clpnt{Y}\leq d$.
	If $d>0$, then the $G$-torsor  given by the quotient map
	\[
	E:=\uGL_{n+d}(S)/\begin{bmatrix}
	1_{n\times n} & 0_{n\times d} \\
	0_{d\times n} & \uGL_{ d}(S)
	\end{bmatrix} \to \uGL_{n+d}(S)/\begin{bmatrix}
	G & 0_{d\times n} \\
	0_{n\times d} & \uGL_{ d}(S)
	\end{bmatrix}=:X\]
	is strongly versal
	for   $\catC_d/S$. Here, $G$ acts on the space
	$\uGL_{n+d}(S)/[\begin{smallmatrix}
	1  & 0  \\
	0  & \uGL_{ d}(S)
	\end{smallmatrix}]$ via $(x,g)\mapsto x[\begin{smallmatrix}
	g & 0  \\
	0  & 1_{d\times d}
	\end{smallmatrix}]$. 
	If $d=0$ and $S$ is a scheme over an infinite field,
	then $E\to X$ is versal for $\catC_0/S$.
	
	Furthermore, the $S$-algebraic space $X$ is smooth.
	If $P$ is one of the properties: being an $S$-scheme,
	being an $S$-scheme locally of finite type,
	being an $S$-scheme of finite type, then
	$X$ has $P$ if  $\uGL_n(S)/G$ has $P$
	(cf.\ Theorems~\ref{TH:quo-of-subgroups},
	\ref{TH:quotient-of-groups-over-a-field}(i)).
	If $S$ is locally noetherian and $G$ is an extension of a reductive $S$-group 
	by a finite locally free $S$-group, then $X\to S$ is   affine   of finite type.
\end{thm}

\begin{proof}
	We use notation from Section~\ref{sec:generic}.
	
	Suppose that $d>0$. By Corollary~\ref{CR:symmetric-weak-versal-is-strong},
	in order to prove the first part of the theorem,
	it is enough to show that the vector bundle $V_{n,n+d}$
	is weakly versal for $\catC_d$,
	and that its corresponding  
	$\uGL_n(\Z)$-torsor (see Proposition~\ref{PR:GLn-torsor-of-Vnd}),
	\[E_{n,n+d}:=\uGL_{n+d}(\Z)/[\begin{smallmatrix}
	1  & 0  \\
	0  & \uGL_{d}(\Z)
	\end{smallmatrix}] \to \uGL_{n+d}(\Z)/[\begin{smallmatrix}
	\uGL_n(\Z) & 0 \\
	0 & \uGL_{d}(\Z)
	\end{smallmatrix}]= X_{n,n+d},\]  admits an action by $\uSL_{n+d}(\Z)$
	that is fppf transitive on $E_{n,n+d}$. 
	
	That $V_{n,d}$ is weakly versal for $\catC_d$ follows
	from Proposition~\ref{PR:generic-Vnd}
	and Theorem~\ref{TH:fr-variant} (with $Z=\emptyset$).
	We get the required $\uSL_{n+d}(\Z)$-action on $E_{n,n+d}\to X_{n,n+d}$
	by restricting the evident action of $\uGL_{n+d}(\Z)$ on $E_{n,n+d}\to X_{n,n+d}$.
	The action of $\uSL_{n+d}(\Z)$ on $E_{n,n+d}$ is fppf transitive
	because the action of $\uGL_{n+d}(\Z)$ is, and for every ring $R$, we have
	$\nSL{R}{n+d}[\begin{smallmatrix}
	1_{n\times n}  & 0  \\
	0  & \nGL{R}{d}
	\end{smallmatrix}]=\nGL{R}{n+d}$ (because $d>0$).
	This proves the first part of the theorem in the case $d>0$.
	
	When $d=0$ and $S$ is a scheme over an infinite field, we prove the first part
	by 
	using Remark~\ref{RM:versal} instead of Corollary~\ref{CR:symmetric-weak-versal-is-strong},
	noting that   $\uGL_{n+d}(\Z)$ acts fppf transitively on $E_{n,n+d} $ even when $d=0$.
	
	We proceed with second part of the theorem, namely, the
	assertions about the structure of $X$. 
	By Lemma~\ref{LM:free-action-quotient}, the morphism
	$\uGL_{n+d}(S)\to X$ is a $G\times \uGL_d(S)$-torsor, hence flat, locally
	of finite presentation and surjective
	(in the sense of \cite[Tag \href{https://stacks.math.columbia.edu/tag/03MC}{03MC}]{stacks_project}).
	Since the composition $\uGL_{n+d}(S)\to X\to S$ is smooth, it follows that
	$X\to S$ is smooth \cite[Tag \href{https://stacks.math.columbia.edu/tag/0AHE}{0AHE}]{stacks_project}.

	Since $E_{n,n+d}$ and $X_{n,n+d}$ are $\Z$-schemes of finite type
	(see   the proof of Proposition~\ref{PR:GLn-torsor-of-Vnd}), 
	the spaces 
	$E = E_{n,n+d}\times_{\Z} S$ and $X_{n,n+d}\times_{\Z} S$ are also   $S$-schemes of finite type.
	Applying Proposition~\ref{PR:quo-of-GLn-torsor-is-a-scheme} to $E\to X_{n,n+d}\times_\Z S$
	now tells us that
	$X=E/G$ has property $P$ whenever $\uGL_n(S)/G$ has it.
	Finally, if $S$ is locally noetherian and $G$ is 
	an extension of a reductive $S$-group by a finite locally free $S$-group,
	then the same holds for $G\times \uGL_d(S)$,
	and $X\to S$ is affine of finite type by Theorem~\ref{TH:quo-of-subgroups}.
\end{proof}

\begin{thm}\label{TH:highly-versal-II}
	Let $n,d,S,G$ be as in Theorem~\ref{TH:highly-versal-I}.
	Let $\nAff $ denote the class of noetherian affine schemes.
	If $d>0$, then the $G$-torsor  given by the quotient map 
	\[E:=\uGL_{n+d}(S)/\begin{bmatrix}
	1_{n\times n} & \bbA^{n\times d}_S \\
	0_{d\times n} & \uGL_{ d}(S)
	\end{bmatrix} \to \uGL_{n+d}(S)/\begin{bmatrix}
	G & \bbA^{n\times d}_S \\
	0_{d\times n} & \uGL_{ d}(S)
	\end{bmatrix}=:X\]
	is strongly $d$-versal
	for   $\nAff /S$
	and strongly versal for the class $\catC_d$
	from Theorem~\ref{TH:highly-versal-I}. Here, $G$ acts on
	$\uGL_{n+d}(S)/[\begin{smallmatrix}
	1_{n\times n} & \bbA^{n\times d}_S \\
	0_{d\times n} & \uGL_{ d}(S)
	\end{smallmatrix}]$ via $(x,g)\mapsto x[\begin{smallmatrix}
	g & 0  \\
	0  & 1_{d\times d}
	\end{smallmatrix}]$. 
	If $d=0$ and $S$ is a scheme over an infinite field, then $E\to X$
	is $0$-versal for $\nAff$ and versal for $\catC_0$.
	
	Furthermore, the $S$-algebraic space $X$ is smooth.
	If $P$ is one of the properties: being an $S$-scheme,
	being an $S$-scheme locally of finite type,
	being an $S$-scheme of finite type, then
	$X$ has $P$ if  $\uGL_n(S)/G$ has $P$.
	In particular, if $S$ is locally noetherian and
	$G$ is an extension of a reductive
	$S$-group by a finite locally free $S$-group, then $X$ is an $S$-scheme of finite type. 
\end{thm}

\begin{proof}
	This is similar to  the proof of Theorem~\ref{TH:highly-versal-I}
	with the following differences: We use the vector bundle
	$V(n,n+d)^\vee $ over $\Gr(n,n+d)$ considered   in Section~\ref{sec:generic}
	instead of $V_{n,n+d}$,  
	Proposition~\ref{PR:action-on-universal-torsor-over-Gr} instead of	
	Proposition~\ref{PR:GLn-torsor-of-Vnd} (to describe the $\uGL_n(\Z)$-torsor
	corresponding to $V(n,n+d)^\vee$),
	and Theorems~\ref{TH:universality-of-tautological} and~\ref{TH:fr-vec-bundles} instead of 
	Proposition~\ref{PR:generic-Vnd} and Theorem~\ref{TH:fr-variant}
	(to show that $V(n,n+d)^\vee$ is weakly $d$-versal for $\nAff$).
\end{proof}

The torsor $E\to X$ of Theorem~\ref{TH:highly-versal-II} enjoys an additional remarkable property
which we call the \emph{extension property}.

\begin{thm}[Extension Property]
	\label{TH:extension}
	Let $n$, $d$, $S$, $G$ and $E\to X$ be as in Theorem~\ref{TH:highly-versal-II}.
	Let $X'\in \catC_d/S$, with $\catC_d$ as in Theorem~\ref{TH:highly-versal-I},
	and let $u:X''\to X'$ be a closed embedding.
	Let $E'\to X'$ be a $G$-torsor and let $E''=u^*E'$.
	Suppose there is a morphism $f:X''\to X$ such that
	$f^*E\cong_{X''} E''$. Then $f$ factors as $g\circ u$
	for a morphism $g:X'\to X$ such that $g^*E\cong_{X'} E'$.	
\end{thm}

\begin{proof}
	Put $F'=E'\times^G \uGL_n(S)$,
	$F''=E''\times^G \uGL_n(S)$ and
	recall that $E$ is the $S$-scheme $E(n,n+d)_S$ (see 
	the proof of Theorem~\ref{TH:highly-versal-II}).
	By Remark~\ref{RM:morphism-pullback-correspondence}, the isomorphism $f^*E\cong_{X''} E''$
	gives rise to a $G $-torsor
	morphism $(\hat{f},f)$ from $E''$ to $E=E(n,n+d)_S$. 
	Since $E(n,n+d)_S$ is also a $\uGL_n(S)$-torsor,
	we have a $\uGL_n(S)$-equivariant morphism
	$\hat{h}'':F''\to E$ given section-wise
	by $\hat{h}''([x'',a])=\hat{f}(x'')\cdot a$.	
	By Corollary~\ref{CR:equiv-to-GLn-torsors}, $\hat{h}''$
	determines a tuple $(v''_1,\dots,v''_{n+d})\in\Gamma(X'',V'')^{n+d}$
	which generates $V''$. By Theorem~\ref{TH:fr-variant},
	there exist $v'_1,\dots,v'_{n+d}\in \Gamma(X',V')$
	which generate $V'$, and satisfy $u^* v'_i=v''_i$ for all
	$i\in\{1,\dots,n+d\}$. Applying Corollary~\ref{CR:equiv-to-GLn-torsors} again,
	we see that $(v'_1,\dots,v'_{n+d})$  
	determines a $\uGL_n(S)$-equivariant morphism
	$\hat{h}':F'\to E$ and that
	$\hat{h}' \circ \hat{u}  =\hat{h}''$, where $\hat{u}:F''\to F'$ is the pullback of $u$ along $F'\to X'$.
	Let  $\hat{g}$ denote the composition of $\hat{h}':F'\to E$
	and the   $G$-equivariant map $e'\mapsto [e',1]:E'\to F'$.
	Then   $\hat{g}:E'\to E$   descends to a morphism   $g:X'\to X$ 
	such that $g^*E\cong E'$, and
	$g\circ u =f$  because
	$\hat{h}'\circ \hat{u}=\hat{h}''$.
\end{proof}

Our last construction of strongly versal $d$-torsors applies only for algebraic groups
over infinite fields, but the strong  $d$-versality applies to  non-affine schemes as well.
Recall from Section~\ref{sec:generic} that $V(n,m)$ ($m\geq n\geq 0$)
denotes the tautological vector bundle over $\Gr(n,m)$.
In the special case $n=1$, we have $\Gr(1,m)=\bbP^{m-1}_\Z$
and $V(1,m)=\calO_{\bbP^{m-1}_\Z}(-1)$. 

\begin{thm}\label{TH:highly-versal-III}
	Let $k$ be an infinite field,   let $d,n\in\N$, and let $G$ be a   subgroup
	of $\uGL_n(k)$.
	Let $Y=\Gr(n,n+d)_k\times_k \bbP^d_k$
	and let $V$ denote the vector bundle $p_1^*V(n,n+d)_k^\vee\otimes_{\calO_Y} p_2^*\calO(-1)$
	on $Y$, where $p_1:Y\to \Gr(n,n+d)_k$ and $p_2:Y\to \bbP^d_k$
	are the first and second projections, respectively.
	Let $E\to Y$ be the $\uGL_n(k)$-torsor corresponding to $V$.
	Then the $G$-torsor
	\[
	E\to E/G
	\]
	is strongly $d$-versal for the class 
	of quasi-projective $k$-schemes.	
	The $k$-algebraic space $E/G$ is smooth,
	and if $G$ is of finite type over $k$,
	it is also a   quasi-projective $k$-scheme.
\end{thm}

We remark that the quasi-projective $k$-schemes are precisely the
finite type $k$-schemes admitting an ample line bundle \cite[Tag \href{https://stacks.math.columbia.edu/tag/0B42}{0B42}]{stacks_project}.

\begin{proof} 
	Let $\qProj/k$ denote the class of quasi-projective $k$-schemes.	
	By Corollary~\ref{CR:symmetric-weak-versal-is-strong},
	in order to show that $E\to E/G$ is strongly $d$-versal
	for $\qProj/k$, it is enough to show that the rank-$n$ vector bundle
	$V$ is weakly versal for $\qProj/k$ and that $\uSL_{n+d}(k)\times\uSL_{1+d}(k)$
	acts on $E\to Y$ in such a way that the action is fppf transitive on $E$.
	
	To prove the first claim, let $V'$ be a rank-$n$ vector bundle on $X'\in  \qProj/k $.
	Then $X'$ admits an ample line bundle $L$.
	There exists some $\ell\in\N$ 
	such that both   $W:=V'\otimes L^{\otimes \ell}$ and $L^{\otimes \ell}$
	are generated by  global sections
	\cite[Tag \href{https://stacks.math.columbia.edu/tag/01Q3}{01Q3}]{stacks_project}.
	By Theorem~\ref{TH:generators-bound-any-scheme},
	there is an open subscheme $U_1\subseteq X'$
	with $\codim(X'-U_1,X')>d$
	such that $W|_{U_1}$ is generated by $n+d$ global sections,
	and
	by Theorem~\ref{TH:universality-of-tautological},
	this determines a $k$-morphism $f_1:U_1\to \Gr(n,n+d)_k$
	such that $W|_{U_1}\cong f_1^*V(n,n+d)_k^\vee$.	
	Similarly, there is an open subscheme $U_2$ of $X'$
	and a $k$-morphism $f_2:U_2\to \Gr(1,1+d)=\bbP^d_k$ 
	such that $\codim(X'-U_2,X')>d$
	and $f_2^*\calO_{\bbP^d_k}(1)=f^*_2V(1,1+d)_k^\vee \cong L^{\otimes\ell}|_{U_2}$.
	
	Put $U=U_1\cap U_2$ and define $f:U\to \Gr(n,n+d)_k\times_k \bbP^d_k$
	by $f(x)=(f_1(x),f_2(x))$ on sections.
	Then $(p_1\circ f)^*V(n,n+d)_k^\vee \cong W|_U$
	while $(p_2\circ f)^*\calO_{\bbP^d_k}(-1)\cong (L^{\otimes \ell})^\vee|_U$.
	This means that
	\begin{align*}
		f^*V & \cong f^*p_1^*V(n,n+d)_k^\vee\otimes_{\calO_{U}} f^*p_2^*\calO(-1)
		\cong W|_U\otimes_{\calO_{U}}  (L^{\otimes \ell})^\vee|_U\cong V'|_U.
	\end{align*}
	Since $\codim(X'-U,X')>d$,
	this proves that $V$ is weakly $d$-versal for $\qProj/k$.
	
	To show	that $E\to Y$ has the desired  $\uSL_{n+d}(k)\times\uSL_{1+d}(k)$-action,
	we unfold the definition of $E$.
	Let $t_1:E_1\to X_1:=\Gr(n,n+d)_k$ be the $\uGL_{n+d}(k)$-torsor
	corresponding to $V(n,n+d)^\vee_k$, and let $t_2:E_2\to X_2:=\Gr(1,1+d)_k=\bbP^d_k$
	be the $\uGL_1(k)$-torsor corresponding to $V(1,1+d)_k=\calO(-1)$.
	Proposition~\ref{PR:action-on-universal-torsor-over-Gr}
	and Lemma~\ref{LM:torsor-of-dual} imply that $\uGL_{n+d}(k)$ acts on $t_1:E_1\to X_1$
	and $\uGL_{1+d}(k)$ acts on $t_2:E_2\to X_2$. Since $d>0$,
	these actions restrict to fppf transitive actions of $\uSL_{n+d}(k)$ and $\uSL_{1+d}(k)$
	on $E_1$ and $E_2$, respectively.
	The $\uGL_n(k)$-torsor corresponding to
	$p_1^*V(n,n+d)_k^\vee$ is $p_1^*E_1=E_1\times_k X_2\to X_1\times_k X_2=Y$
	and the $\uGL_1(k)$-torsor corresponding   
	to $p_2^*\calO (-1)$ is $X_1\times_k E_2\to X_1\times_k X_2=Y$.
	Thus, $E_1\times_k E_2=(p_1^*E_1)\times_Y(p_2^*E_2)$
	is a $\uGL_n(k)\times_k \uGL_1(k)$-torsor corresponding
	to the pair of vector bundles $(p_1^*V(n,n+d)^\vee_k, p_2^*\calO(-1))$.
	The induced morphism $E_1\times_k E_2\to X_1\times_k X_2=Y$ is just $t_1\times t_2$
	so   $\uSL_{n+d}(k)\times_k \uSL_{1+d}(k)$ acts on this torsor,
	and   the action on $E_1\times_k E_2$ is fppf transitive.
	
	Consider the $k$-group   morphism
	$\vphi: G:=\uGL_n(k)\times_k  \uGL_1(k)\to \uGL_n(k)$
	given by $\vphi(g_1,g_2)=g_1\otimes g_2$ on sections. 
	If $X$ is a $k$-scheme, $(V_1,V_2)$
	is   a pair  of vector bundles on $X$ such that
	$\rank V_1=n$ and $\rank V_2=1$, and $P\to X$ is the
	$\uGL_n(k)\times_k\uGL_1(k)$-torsor
	corresponding to the pair $(V_1,V_2)$, 
	then the $\uGL_n(k)$-torsor corresponding to $V_1\otimes_{\calO_X} V_2$
	is  $P\times^{\vphi,G}\uGL_n(k)$  up to $X$-isomorphism.
	This can be seen, for instance, by choosing a Zariski covering $\{U_i\to X\}_{i\in I}$
	trivializing $V_1$ and $V_2$, and looking at the \v{C}ech $1$-cocycles
	defining the torsors corresponding to $V_1$, $V_2$ and $V_1\otimes_{\calO_X} V_2$.
	Applying this to the  $\uGL_n(k)\times_k \uGL_1(k)$-torsor
	$E_1\times_k E_2\to X_1\times_k X_2=Y$, we see that
	$E$ from the theorem is $(E_1\times_k E_2)\times^{\vphi,G}\uGL_n(k)$. The $\uGL_n(k)$-torsor $E\to Y$ inherits
	from  $E_1\times_k E_2\to Y$ an $\uSL_{n+d}(k)\times_k \uSL_{1+d}(k)$-action
	such that 
	the   canonical map $E_1\times_k E_2\to E$ is    $\uSL_{n+d}(k)\times_k \uSL_{1+d}(k)$-equivariant.
	Since $\vphi:G\to\uGL_n(k)$ is $k$-fppf surjective, so is 
	$E_1\times_k E_2\to E$, and it follows
	that $\uSL_{n+d}(k)\times_k \uSL_{1+d}(k)$ acts fppf transitively
	on $E$.

	The smoothness of $E/G$ over $k$ is shown as in Theorem~\ref{TH:highly-versal-I}. To finish, suppose that $G$ is of finite type.
	By Theorem~\ref{TH:quotient-of-groups-over-a-field}(i),
	$\uGL_n(k)/G$ is a $k$-scheme,
	and the same applies to $E/G$ by Proposition~\ref{PR:quo-of-GLn-torsor-is-a-scheme}.
	Note that $E_1$ and $E_2$, and thus $E/G$, admits a $k$-point $x$
	(see Proposition~\ref{PR:sections-of-univ-bundle-over-Gr}),
	and $x$  
	is closed because $E/G$ is a $k$-scheme.
	Since $H:=\uSL_{n+d}(k)\times \uSL_{1+d}(k)$ acts fppf transitively
	on $E/G$, the $k$-group $H$ has \emph{closed} subgroup $H_1$ --- the stabilizer of $x$ --- 
	such that $E/G\cong H/H_1$,
	and the latter is a quasi-projective $k$-scheme by 
	Theorem~\ref{TH:quotient-of-groups-over-a-field}(i).	
\end{proof}

In light of the previous theorems, we ask:

\begin{que}
	Let $S$ be a scheme, let $G$ be a   linear $S$-group and let $d\in\N$.
	Is there a $G$-torsor  over an $S$-algebraic space 
	that is strongly $d$-versal for the class of $S$-schemes admitting
	an ample line bundle?
\end{que}

\begin{que}
	Let $G$ be an affine algebraic group over an infinite field $k$.
	Is there a $G$-torsor  over a  $k$-algebraic space 
	that is strongly $d$-versal for the class of $k$-schemes of finite
	type?
\end{que}

\section{Examples}
\label{sec:examples}

In this section, we apply the results of Section~\ref{sec:existence} to particular
group schemes in order to deduce the existence of strongly versal Galois extensions, Azumaya algebras
and tori. We refer the reader to \cite{Saltman_1999_lectures_on_div_alg} or \cite{Ford_2017_separable_algebras}
for    definitions concerning Galois extensions and Azumaya algebras,
and to \cite[III.\S8]{Knus_1991_quadratic_hermitian_forms} for information about Azumaya algebras with involution.

As in Theorem~\ref{TH:highly-versal-I}, we write
$\catC_d$ for the class of affine schemes $X$ for which the subspace
$\clpnt{X}$ is noetherian
and has dimension $\leq d$. It includes all noetherian affine schemes of dimension $\leq d$.
We call a ring homomorphism $A\to B$ \emph{universally
injective} if $a'\mapsto 1_B\otimes a': A'\to B\otimes_A A'$ is injective for every $A$-ring $A'$.
By virtue of Proposition~\ref{PR:sch-dom-equiv-with-SGA}(ii),
this is  equivalent to $\Spec B\to \Spec A$ being universally schematically dominant.
It is also equivalent  to $A$ being a \emph{pure} $A$-submodule of $B$.

\begin{example}\label{EX:versal-Galois}
	Let $\Gamma$ be a finite group, and let $G$ be the constant
	$\Z$-group corresponding to $\Gamma$.
	We may embed $G$ as a closed subgroup of $\uGL_n(\Z)$ for some $n\in\N$,
	e.g., via the regular representation. Then Theorem~\ref{TH:highly-versal-I}
	says that for every $d\in\N$, the $G$-torsor $E:=\uGL_{n+d}(\Z)/[\begin{smallmatrix}
	1 & 0 \\ 0 & \uGL_d(\Z)\end{smallmatrix}]\to \uGL_{n+d}(\Z)/[\begin{smallmatrix}
	G & 0 \\ 0 & \uGL_d(\Z)\end{smallmatrix}]:=X$ is strongly versal for the class $\catC_d$, and $X$ is affine and 
	smooth  over $\Z$. Since the category of $\Gamma$-Galois extensions
	of rings (morphisms are specializations)   is (anti-)equivalent to the category
	of $G$-torsors over affine schemes, we can rephrase this result in terms of 
	$\Gamma$-Galois extension as follows: Write $X=\Spec R$
	and let $S/R$ be the $\Gamma$-Galois extension corresponding to
	$E\to X$. Then $R$ is a smooth $\Z$-algebra and $S/R$
	is a $\Gamma$-Galois extension that is strongly versal for $\catC_d$.
	That is, for every $\Gamma$-Galois extension
	$S'/R'$ such that the space $\Max R'$ is notherian of dimension $\leq d$
	(e.g., if $R'$ is noetherian of Krull dimension $\leq d$), there are\footnote{
		By unfolding the proof of Theorem~\ref{TH:highly-versal-I},	one sees
		that  $m$ depends only on $\Gamma$ and not on $R'\to S'$.
	}
	$m\in\N$
	and a  ring homomorphism $\vphi:R \to R'[t_1,\dots,t_m]$
	such that the specialization of $S/R$ along $\vphi$
	is isomorphic to $S'[t_1,\dots,t_m]/ R'[t_1,\dots,t_m] $ (as  $R'[t_1,\dots,t_m]$-rings
	and $\Gamma$-modules), and the induced map $r\otimes r'\mapsto \vphi(r)r':R\otimes_{\Z} R'\to R'[t_1,\dots,t_m]$ is universally injective. 	
	In particular, by specializing
	all the $t_i$ to $0$, we see that $S/R$ specializes to any $\Gamma$-Galois
	extension $S'/R'$ such that $\Max R$ is noetherian of dimension $\leq d$.
\end{example}

\begin{example}\label{EX:versal-Azumaya}
	Let $G=\uPGL_n(\Z)$, the automorphism $\Z$-scheme of the $\Z$-algebra
	$\nMat{\Z}{n}$.
	Then $G$ is linear and reductive.
	It is well-known that there is an equivalence of
	categories over $\Z$-algebraic spaces
	between $\Tors(G)$ (see Section~\ref{sec:torsors})
	and the opposite category of   degree-$n$ Azumaya algebras over $\Z$-algebraic spaces.
	Thus, 
	applying Theorem~\ref{TH:highly-versal-I} to
	$G$ implies
	that, for every $d\in\N$, there is a smooth $\Z$-algebra $R$
	and a degree-$n$ Azumaya algebra $A$ over $R$   that is strongly
	$d$-versal for $\catC_d$.
	That is, for every ring $R'$ such that $\Max R'$ is noetherian of dimension $\leq d$
	and every  degree-$n$ Azumaya $R'$-algebra $A'$,
	there are $m\in\N$ and  a ring homomorphism $\vphi:R\to R'[t_1,\dots,t_m]$
	such that $A\otimes_R R'[t_1,\dots,t_m]\cong A'[t_1,\dots,t_m]$ as
	$R'[t_1,\dots,t_m]$-algebras, and moreover, $r\otimes r'\mapsto \vphi(r)r':R\otimes_\Z R'\to R'[t_1,\dots,t_m]$ is universally
	injective. In particular, $A/R$ specializes to any degree-$n$ Azumaya  algebra
	over any  $R'$ as above.
\end{example}
	
	For a ring $R$, we write $\umu_n(R)$ for the $R$-group of $n$-th roots of unity.
	
\begin{example} 
	The conclusions of Example~\ref{EX:versal-Azumaya} continue  to hold
	if we work over $\Z[\frac{1}{2}]$ instead of $\Z$ and
	replace ``degree-$n$ Azumaya algebra'' with ``degree-$n$ Azumaya algebra with orthogonal (resp.\ symplectic)
	involution'' ($n$ is even in the symplectic case).
	To see this, apply Theorem~\ref{TH:highly-versal-I}  with $G=\mathbf{PO}_n(\Z[\frac{1}{2}])=\uO_n(\Z[\frac{1}{2}])/\umu_2(\Z[\frac{1}{2}])$ (resp.\ 
	$\mathbf{PSp}_n(\Z[\frac{1}{2}])=\uSp_n(\Z[\frac{1}{2}])/\umu_2(\Z[\frac{1}{2}])$ for $n$ even)
	and note that  
	$\Tors(G)$ is
	anti-equivalent to the category of degree-$n$
	Azumaya algebras with orthogonal (resp.\ symplectic) involution over $\Z[\frac{1}{2}]$-algebraic spaces;
	see \cite[III.\S8.5]{Knus_1991_quadratic_hermitian_forms}.
\end{example}

Recall that the period (also called exponent)
of an Azumaya algebra over a ring $R$
is the order of its Brauer class in the Brauer group of $R$.
We say that an Azumaya algebra is \emph{$n$-periodic} if its period divides $n$.

\begin{prp}\label{PR:versal-period-limited}
	Let $m,n,d\in\N$. Then there is a smooth $\Z$-ring $R$
	and an $n$-periodic degree-$m$  Azumaya $R$-algebra  
	that is strongly versal (among the $n$-periodic degree-$m$ Azumaya algebras)
	for $\catC_d$.
	That is, for any ring  $R'$ such that 
	$\Max R'$ is noetherian of dimension $\leq d$ 
	and every $n$-periodic degree-$m$ Azumaya $R'$-algebra $A'$, 
	there are $\ell\in\N$ and a ring homomorphism $\vphi:R\to R'[t_1,\dots,t_\ell]$
	such that $A\otimes_R R'[t_1,\dots,t_\ell]\cong A'[t_1,\dots,t_\ell]$ as $R'[t_1,\dots,t_\ell]$-algebras
	and   $r\otimes r'\mapsto \vphi(r)r':R\otimes_\Z R'\to R'[t_1,\dots,t_\ell]$ is universally injective.
\end{prp}

\begin{proof}
	Let $G=\uGL_m(\Z)/\umu_n(\Z)$,
	and let $\vphi$ denote the evident $\Z$-group morphism $G\to \uGL_m(\Z)/\nGm{\Z}=\uPGL_m(\Z)$.
	If $E'\to X'$ is a $G$-torsor, then $E'\times^{\vphi,G}\uPGL_m(\Z)$
	is a $\uPGL_m(\Z)$-torsor which corresponds to a degree-$m$ Azumaya algebra $A'$ over $X'$.
	It is well-known that the period of $A'$ divides $n$,
	and conversely,    every $n$-periodic degree-$m$ Azumaya algebra over $X'$ 
	is obtained in this manner (i.e., from a $G$-torsor over $X'$).
	
	The $\Z$-group $G$ is   reductive and linear (Theorem~\ref{TH:Thomason}), and so Theorem~\ref{TH:highly-versal-I}
	implies that there exists a $G$-torsor $E $ over a smooth affine
	$\Z$-scheme $X$ that is $d$-versal for $\catC_d$.
	Write $X=\Spec R$, let $E_1:=E\times^{\vphi,G}\uPGL_n(\Z)$, and let
	$A$ be the Azumaya $R$-algebra corresponding to $E_1$.
	Let $R'$ and $A'$ be as in the proposition, and let $E'_1\to X':=\Spec R'$
	denote the $\uPGL_n(\Z)$-torsor corresponding to $A'$.
	Then there is a $G$-torsor $E'\to X'$ such that $E'_1\cong_{X'} E'\times^{\vphi,G}\uPGL_n(\Z)$. 
	Since $E\to X$ is strongly versal for $\catC_d$, there is $\ell\in\N$
	and a morphism $f:\bbA^\ell_\Z\times X'\to X$
	such that $f^*E\cong \bbA^\ell_\Z\times E'$,
	and $(f,p_2):\bbA^\ell_\Z\times X'\to X\times X'$ is universally
	schematically dominant. The former  means that $f^*E_1\cong \bbA^\ell_\Z \times E'_1$.
	Writing $\vphi:R\to R'[t_1,\dots,t_\ell]$ for the adjoint of $f:\bbA^\ell_\Z \times X'\to X$, the proposition follows.
\end{proof}

\begin{example}
	Let $k$ be an infinite field of characteristic $0$.
	Given an abstract group $\Gamma$, let $\underline{\Gamma}$
	denote the constant $k$-group corresponding to $\Gamma$.
	The $k$-group  $\underline{\Gamma}$ is flat and locally
	of finite presentation, but it is not of finite type if $\Gamma$
	is infinite.
	
	Let $n\in\N$ and let $G=\underline{\nGL{\Z}{n}}$. There is a $k$-group monomorphism  
	from $G=\bigsqcup_{\gamma \in\nGL{\Z}{n}}\Spec k$ to $\uGL_n(k)$ 
	which restricts to the $\gamma$-section $\Spec k\to \uGL_n(k)$
	on the $\gamma$-copy of $\Spec k$ for every
	$\gamma\in \nGL{\Z}{n}$. Thus, $G$ is linear, and so Theorem~\ref{TH:highly-versal-II}
	tells us that, for every $d\in\N$, there is a $G$-torsor $E_d\to X_d$ that is strongly $d$-versal
	for affine noetherian $k$-schemes. 	
	However, if $n>1$, then the $k$-algebraic space 
	$X_d=\uGL_{n+d}(k)/ [\begin{smallmatrix}
	G & \bbA^{n\times d}_k \\
	0_{n\times d} & \uGL_{n+d}(k)
	\end{smallmatrix}]$, while smooth, is not    a $k$-scheme, and also not separated
	(cf.\ \cite[Tag \href{https://stacks.math.columbia.edu/tag/02Z7}{02Z7}]{stacks_project}).
	
	It is well-known
	that the automorphism sheaf  of the $k$-group $\nGm{k}^n$ is isomorphic
	to $G=\underline{\nGL{\Z}{n}}$. Thus, a standard   argument  shows
	that $\Tors(G)$ is equivalent as a category over $k$-algebraic spaces
	to the category of rank-$n$ tori over $k$-algebraic spaces (morphisms are specializations).
	Let $T_d$ be the rank-$n$ torus over the $k$-algebraic space  $X_d$
	corresponding to $E_d\to X_d$.
	Then   $T_d$  is strongly $d$-versal for affine noetherian $k$-schemes. In particular,
	$T_d$
	specializes to any rank-$n$ torus $T$ over  
	an affine noetherian $k$-scheme $X$ away from some codimesion-$(d+1)$ closed subscheme of $X$.
\end{example}

\section{$\infty$-Versal Torsors}
\label{sec:infty-versal}

Throughout, $\Aff$ denotes the class of affine schemes.
Let $S$ be a scheme and let $G$ be an $S$-group. This section concerns with $G$-torsors
that are    versal for $\Aff/S$. 
In particular, 
they  specialize to   
every $G$-torsor over  an affine scheme over $S$.
We will show that  there are nontrivial $S$-groups     admitting
such torsors (over a finite type $S$-scheme),
and characterize all these groups   when $S$ is the spectrum of a field of characteristic $0$.

We write $\uU_n(S)$ for the $S$-group of unipotent upper triangular $n\times n$
matrices over $\calO_S$.

\begin{thm}\label{TH:infinity-versal}
	Let $S$ be a scheme and let $G$ be a  subgroup of $\uU_n(S)$
	that is flat and locally of finite presentation over $S$.
	Then the $G$-torsor $\uU_n(S)\to \uU_n(S)/G$ is strongly versal for $\Aff/S$.
	
	Furthermore, the $S$-algebraic space $\uU_n(S)/G$ is smooth.
	If $S$ is the spectrum of of field and $G$ is of finite type
	over $S$, then $\uU_n(S)/G\to S$
	is quasi-projective. If $G$ is finite and locally free, then $\uU_n(S)/G\to S$
	is   affine and of finite type.
\end{thm}

\begin{proof}
	Observe first that for every
	for every $X\in\Aff/S$, we have $\HH^1_{\fppf}(X,\uU_n(S))=\{*\}$.
	Indeed, $\uU_n(S)$ admits a filtration 
	$\uU_n(S)=U_0\rhd U_1\rhd \cdots\rhd U_r=0$
	such that $U_{i-1}/U_i\cong \nGa{S}$ for all $i\in\{1,\dots,r\}$.
	Since $X$ is affine, $\HH^1_{\fppf}(X,\nGa{S})=\HH^1_{\fppf}(X,\calO_X)=0$
	\cite[III, Lem.~2.15, Prop.~3.7]{Milne_1980_etale_cohomology}.
	Now, an inductive argument using the long cohomology
	exact sequence associated to $1\to U_i\to U_{i-1}\to U_{i-1}/U_i \to 1$
	shows that $\HH^1_{\fppf}(X,U_i)=\{*\}$ for all $i$, and 
	$\HH^1_{\fppf}(X,\uU_n(S))=\{*\}$ in particular.
	
	Put $E_0:=\uU_n(S)\to \uU_n(S)/G=:X_0$ and recall that this is indeed
	a $G$-torsor by Lemma~\ref{LM:free-action-quotient}.
	We first show that $E_0\to X_0$ is weakly versal for $\Aff/S$.
	Let $X\in\Aff/S$
	and let $E\to X$ be a $G$-torsor.
	Then $E\to X$ corresponds to  a class $\alpha\in \HH^1_\fppf(X,G)$.
	It is well-known 
	that the embedding $G\to \uU_n(S)$ gives rise to an exact sequence of pointed sets:
	\[
	(\uU_n(S)/G)(X)\xrightarrow{\gamma} \HH^1_\fppf(X,G)\to \HH^1_\fppf(X,\uU_n(S))
	\]
	in which $\gamma$ maps a morphism $f:X\to \uU_n(S)/G$ to the cohomology
	class of $f^*E_0$.
	By the previous paragraph, the image of $\alpha$ vanishes in $\HH^1_\fppf(X,\uU_n(S))$,
	so it lifts to some $f:X\to \uU_n(S)/G$, and $f^*E_0\cong E$.
	
	To see that $E_0\to X_0$
	is strongly versal for $\Aff/S$,
	observe that the left action of $\uU_n(S)$ on $E_0$ and $X_0$
	gives rise to an action of  $\uU_n(S)$ on the $G$-torsor $E_0\to X_0$.
	This action is clearly fppf   transitive on   $E_0$.
	Since $\uU_n(S)\cong \bbA^{n(n-1)/2}_S$
	as $S$-schemes, 
	Proposition~\ref{PR:symmetric-weak-versal-is-strong}(ii)
	tells us that 
	$E_0\to X_0$ is strongly versal for $\Aff/S$.
	
	The smoothness of $\uU_n(S)/G\to S$
	is shown as in Theorem~\ref{TH:highly-versal-I}.
	If $S$ is the spectrum of a field and $G\to S$ is of finite type, 
	then $\uU_n(S)/G$
	is a quasi-projective  $S$-scheme   by Theorem~\ref{TH:quotient-of-groups-over-a-field}(i).
	If $G$ is finite, then $\uU_n(S)/G\to S$
	is   affine   of finite type by Theorem~\ref{TH:quo-of-subgroups}.
\end{proof}

\begin{remark}\label{RM:infty-versal}
	Another sufficient condition for the existence of a $G$-torsor
	that is 
	strongly versal for $\Aff/S$ is:
	\begin{enumerate}
		\item[($*$)] Every $G$-torsor $E\to X$ with $X\in \Aff/S$
		is trivial.
	\end{enumerate}	  
	Indeed, in this case, every   $G$-torsor as in ($*$) is a
	specialization of the trivial $G$-torsor $G\to S$, making $G\to S$
	strongly versal for $\Aff/S$. 
	Such $S$-groups  may be deemed as ``uninteresting'' for 
	the problem of finding groups with torsors that are weakly
	versal for $\Aff/S$.
		
	The proof of Theorem~\ref{TH:infinity-versal} can now be broken
	into the observations that $\uU_n(S)$
	satisfies ($*$), and that every flat locally of finite presentation 
	subgroup of an $S$-group satisfying
	($*$) admits a torsor that is weakly versal for $\Aff/S$.
	
	When $S=\Spec k$ for a field $k$ of characteristic $0$, it is well-known
	\cite[Proposition~14.22, Remark~14.24]{Milne_2017_algebraic_groups}
	that every finite type subgroup $G$ of $\uU_n(k)$ admits  a filtration
	$G=G_0\rhd G_1\rhd \cdots\rhd G_r=0$ with $G_i/G_{i-1}\cong \nGa{k}$,
	and thus already satisfies ($*$). Consequently, in this case, Theorem~\ref{TH:infinity-versal}
	does not give ``interesting'' examples of $k$-groups admitting torsors
	that are strongly
	versal for $\Aff/S$.
	
	However,  when $S=\Spec k$ for a field $k$ of   characteristic $p>0$, the $k$-group
	$\uU_n(k)$ admits subgroups which do not satisfy ($*$), and so 
	we get ``interesting''
	examples of  torsors that are strongly versal for $\Aff/S$.
\end{remark}

\begin{example}\label{EX:Artin-Schrier}
	Let $k$ be a field of characteristic $p>0$
	and let $C_p$ denote the   cyclic  group of order $p$,
	viewed as a constant $k$-group.
	We identify $C_p$ with the subgroup of $\uU_2(k)\cong \nGa{k}$
	cut by the equation $t^p=t$ (with $t$ being the coordiate of $\nGa{k}\cong \bbA^1_k$).
	Then, by Theorem~\ref{TH:infinity-versal}, 
	$X=\nGa{k}/C_p$ is an affine $k$-scheme of finite type,
	and $\nGa{k}\to \nGa{k}/C_p$ is strongly versal for $\Aff/k$.
	In fact,   $\nGa{k}\to\nGa{k}/C_p$
	is isomorphic to the $k$-group homomorphism
	$\nGa{k}\to \nGa{k}$ given by $x\mapsto x^p-x$ on sections.
	As in Example~\ref{EX:versal-Galois}(i), we can rephrase this in the language
	of $C_p$-Galois extensions:
	The $C_p$-Galois extension corresponding
	to $\nGa{k}\to \nGa{k}/C_p$ is the generic Artin--Schrier extension
	%$k[t]\subseteq  k[t,x\where x^p-x=t]$ 
	$k[t,x\where x^p-x=t]/k[t]$
	(here, $C_p=\Trings{\sigma\where\sigma^p=1}$ 
	acts on the $k[t]$-algebra $k[t,x\where x^p-x=t]$
	by $\sigma(x)=x+1$). Thus, for every  $C_p$-Galois  extension \emph{of $k$-rings} 
	$ S/R$, there is $n\in\N$ (in fact, one can take $n=1$)
	and  a morphism
	$\vphi:k[t]\to R[t_1,\dots,t_n]$ such that $S[t_1,\dots,t_n]/ R[t_1,\dots,t_n]  $ is isomorphic
	to the specailiazation of $k[t,x\where x^p-x=t]/k[t]  $ along $\vphi$. Moreover,
	the induced map $R[t]=k[t]\otimes_k R\to R[t_1,\dots,t_n]$ is universally
	injective. Taking $k=\F_p$ and specializing $t_1,\dots,t_n$ to $0$, we recover
	the well-known fact that
	every $C_p$-Galois extension $S/R$ of characteristic-$p$
	rings is an Artin-Schier extension, namely, there is $a\in R$
	such that $S\cong R[x\where x^p-x=a]$
	as $R$-algebras and $C_p$-modules.	
\end{example}

The last example generalizes to other Galois extensions as follows.

\begin{cor}\label{CR:generic-Galois}
	Let $k$ be a field of characteristic $p>0$
	and let $\Gamma$ be a finite $p$-group.
	Then there exists a $\Gamma$-Galois extension of $k$-rings
	$S/R$ that is strongly versal for (spectra of) $k$-rings.
	In particular, it   specializes to every other $\Gamma$-Galois extension of $k$-rings.
	The ring $R$ can be taken to be a smooth $k$-algebra.
\end{cor}

\begin{proof}
	Let $G$ denote the constant $k$-group corresponding to $\Gamma$.
	Since $\Gamma$ is a $p$-group, it can be embedded in the abstract
	group $U_n(\F_p)$ for some $n\in\N$. Thus, $G$   is isomorphic
	to a subgroup of the $k$-group $\uU_n(k)$.
	By Theorem~\ref{TH:infinity-versal} and the equivalence
	between $\Gamma$-Galois extensions and $G$-torsors,
	the $\Gamma$-Galois extension $R\to S$ corresponding
	to the $G$-torsor $\uU_n(k)\to \uU_n(k)/G$
	satisfies the requirements.
\end{proof}

\begin{remark}
	With notation as in Corollary~\ref{CR:generic-Galois},
	the existence of $\Gamma$-Galois extensions $S/R $ that are weakly
	versal for $\Aff/k$
	goes back to Saltman  \cite[Theorem~2.3]{Saltam_1978_noncrossed_products},
	who moreover moreover gave explicit examples in which $R$ is a polynomial ring
	over $k$ (i.e., they are \emph{generic}).
	Woodcock and Fleishmann \cite{Fleischmann_2011_non_linear_actions_of_p_groups},
	\cite{Fleischmann_2018_free_actions_of_p_groups} gave other explicit constructions of 
	weakly versal $\Gamma$-Galois extensions $S/R$ such that both $R$ and $S$ are polynomial rings
	over $k$ and which also enjoy     additional cogeneration
	properties. They further investigate the minimal possible Krull dimension of such $R$.
	Our construction is very different from  
	these sources, and it gives rise to strongly versal $\Gamma$-Galois extensions.
\end{remark}

Let $k$ be a field, and let   $\ftAff/k$
denote the class of finite-type affine $k$-schemes.
We finish this section with   a partial converse to Theorem~\ref{TH:infinity-versal},
showing that every affine algebraic group $G$ over a characteristic-$0$ field
$k$ which admits a torsor over a 
finite type   $k$-scheme
that is weakly versal for $\ftAff/k$ is   isomorphic to
a subgroup of $\uU_n(k)$, or equivalently, unipotent.
Consequently,    all such algebraic groups $G$ are ``uninteresting''
in the sense that every $G$-torsor over an affine base is trivial. 
This contrasts with the case  
$\Char k>0$, where there always exist  ``interesting'' such groups.

\begin{thm}\label{TH:converse}
	Let $k$ be a field of characteristic $0$  and let $G$ be an affine
	algebraic group over $k$.
	Suppose that there exists a $G$-torsor
	$E\to X$ that is weakly versal for $\ftAff/k$
	and such that $X$ is a scheme of finite type over $k$.
	Then $G$ is     unipotent. 
\end{thm}

We first prove the following lemma.

\begin{lem}\label{LM:vashing-of-sing-coh}
	Let $X$ be a   $\C$-scheme of finite type.
	Give $X(\C)$ the analytic topology. Then there is $n\in\N$
	such that the singular cohomology $\HH^i(X(\C),\Z)$ vanishes for all $i\geq n$.
\end{lem}

\begin{proof}
	When $X$ is quasi-affine, this follows from a theorem
	of Lojasiewicz \cite[Theorem~1]{Lojasiewicz_1964_triangulations}, who showed that $X(\C)$ has the homotopy
	type of a CW-complex of dimension $\leq 2\dim X$.
	The general case follows from the quasi-affine case by taking
	a finite affine covering $\calU=\{U_i\}_{i=1}^r$ of $X$ and using the 
	\v{C}ech spectral sequence $E_2^{p,q}=\HH^p(\calU,\underline{\HH}^q(X(\C),\Z)))\Longrightarrow
	\HH^i(X(\C),\Z)$
	\cite[Tag~\href{https://stacks.math.columbia.edu/tag/03OW}{03OW}]{stacks_project}.
\end{proof}

\begin{proof}[Proof of Theorem~\ref{TH:converse}]
	We will use the fact that $\Char k=0$
	to apply analytic tools  such as the 
	$G$-equivariant cohomology
	theory considered in \cite[\S11]{First_2022_generators}.
	This requires $k$ to be a subfield of $\C$, and reducing
	to this situation will necessitate  
	some maneuvering.
	
	For the sake of contradiction, suppose that $G$ is not unipotent.
	
\smallskip	
	
	\Step{1} 
	We may assume that $\Q\subseteq k$.
	Since $G$, $E $ and $X $
	are   of finite presentation over $\Spec k$,
	there is   a  finitely generated $\Q$-subfield $k_0$ of $k$,
	a $k_0$-group $G_0$, and a $G_0$-torsor $E_{ 0}\to X_{ 0}$
	such that $ G_0 \times_{k_0} k=G$
	and $E \to X $
	is the base-change of $E_{ 0}\to X_{ 0}$ along $\Spec k\to \Spec k_0$.
	(In what follows the subscript $0$ indicates that the object at hand
	is a $k_0$-scheme or a $k_0$-morphism.)

	Next, let $e\in\N$; the value of $e$ will be determined in Step~3.
	It is well-known, e.g.\ see \cite[Remark~1.4]{Totaro_1999_chow_ring_of_classifying_space}, 
	that there exists a finite-dimensional $k_0$-vector space $V_0$
	(which we also view as a $k_0$-scheme $\bbA^{\dim V_0}_{k_0}$),
	a linear representation
	$G_0\to \uGL(V_0)$ and an open subcheme $U_0\subseteq V_0$
	stable under $G_0$ such that $U_0/G_0$ is a quasi-projective $k_0$-scheme,
	$U_0\to U_0/G_0$ is a $G_0$-torsor, and $\codim(V_0-U_0,V_0)>e$.
	Jouanolou \cite{Jouanolou_1972_suite_exacte_de_Mayer_Vietoris}
	showed that there also exists   
	a vector bundle $P_0$ over $U_0/G_0$ and a $P_0$-torsor $q_0:W_0\to U_0/G_0$
	such that $W_0$ is an affine $k_0$-scheme.	
	%affine $k_0$-scheme $W_0$
	%and a morphism $q:W_0\to U_0/G_0$ that is a torsor under a vector bundle over $U_0/G_0$.
	Let $F_0\to W_0$ be the pullback of the $G_0$-torsor $U_0\to U_0/G_0$ along $q_0$,
	and put $F=F_0\times_{k_0} k$ and $W=W_0\times_{k_0} k$.
	Then $F\to W$ is   
	a $G$-torsor over an \emph{affine} $k$-scheme of finite type.
	Since $E \to X$ is weakly versal for $\ftAff/k$, 
	there exists a $G$-equivariant morphism
	$\hat{g}:F\to E $.
	As in the last paragraph, $k$ admits   a finitely generated
	$k_0$-subfield $k_1$ such that $\hat{g}$ is extended
	from a $G_0\times_{k_0}k_1$-equivariant morphism  
	$\hat{g}_1:F_1:=F_0\times_{k_0}k_1\to E_{ 0}\times_{k_0} k_1:=E_{1}$.
	We   put $W_1=W_0\times_{k_0}k_1$, $U_1=U_0\times_{k_0}k_1$, $G_1=G_0\times_{k_0}k_1$
	and so on,
	and let $\hat{q}_1$ and $q_1$ denote the $k_1$-morphisms
	obtained from $\hat{q}_0$ and $q_0$ by base-change along $k_0\to k_1$. 
	The morphism $\hat{g}_1:F_1\to E_1$
	gives rise to  a $k_1$-morphism $g_1:W_1=F_1/G_1\to E_1/G_1= X_1$.
	%$F_1\to U_1$
	%obtained by base-change from $F_0\to U_0$.
	
%	To conclude this step, we constructed finitely generated
%	$\Q$-subfields $k_0\subseteq k_1 $ of  $k$
%	and $G_1$-equivariant morphisms
%	\[U_1\xleftarrow{\hat{q}} F_1\xrightarrow{\hat{g}_1} E_1\]
%	Abusing the notation, 
%	we henceforth denote $G_1=G_0\times_{k_0} k_1$ simply as $G$.

	To conclude this step, we constructed finitely generated
	$\Q$-subfields $k_0\subseteq k_1 $ of  $k$
	and a commutative diagram of finite type $k_1$-schemes
	\[\xymatrix{
	V_1 & 
	U_1 \ar@{_{(}->}[l]_\alpha \ar[d] &
	F_1 \ar[l]_{\hat{q}_1} \ar[r]^{\hat{g}_1} \ar[d] &
	E_1 \ar[d] \\
	& 
	U_1/G_1 &
	W_1 \ar[l]_{q_1} \ar[r]^{g_1} &
	X_1}
	\]
	in which the columns are $G_1$-torsors and the squares represent morphisms
	of $G_1$-torsors. 
	%Moreover, the morphisms $q_1$ and $\hat{q}_1$
	%also carry a structure of a $P_1$-torsor, where $P_1$ is a vector bundle
	%over $U_1/G_1$, and the vertical arrows
	%of the left square specify a morphism
	%of $P_1$-torsors. 
	Moreover, $q_1$ carries the structure of a $P_1$-torsor
	for some vector bundle $P_1$  over $U_1/G_1$.	
	The map denoted $\alpha$ is the base change of the $G_0$-equivariant
	open immersion $U_0\to V_0$ to $k_1$. Thus, $V_1$ is a represntation of $G_1$,
	and $U_1$ is a $G_1$-stable open subscheme of $V_1$ satisfying $\codim(V_1-U_1,V_1)>e$.
	
	\emph{Abusing the notation, 
	we henceforth denote $G_1=G_0\times_{k_0} k_1$  as $G$.}
	
\smallskip	
	
	\Step{2}
	Choose an embedding of $k_0$ in $\C$; this embedding will
	be used to choose the number $e$ from Step~1 later on.
	Once $e$ has been chosen, the field $k_1$
	from Step~1 can be determined and we extend the embedding $k_0\to \C$
	into an embedding $k_1\to \C$.
	
	We are now in position to apply the $G $-equivariant
	cohomology theory 
	considered in  \cite[\S11]{First_2022_generators}.
	We recall that this theory assigns 
	to any  
	finite-type $k_1$-scheme $Y$ carrying a $G $-action
	cohomology groups $\{H^i_G(Y)\}_{i\in\Z}$,  contravariantly in $Y$.
	It has the following properties:
	\begin{enumerate}[label=(\roman*)]
		\item If $T\to Z$ is a $G$-torsor, then
	$\HH^i_G(T)\cong \HH^i(Z(\C),\Z)$ canonically, with the right
	hand side denoting the singular cohomology of the analytic space
	$Z(\C)$.
		\item If $V$ is a representation of $G$
		and $U$ is a $G$-invariant open subvariety of $V$ such that $\codim(V-U,V)>d$, 
		$U/G$ is a $k$-variety
		and $U\to U/G$ is a $G$-torsor, then
		the map $\HH^i_G(Y)\to\HH^i_G(Y\times_{k_1} U)\cong \HH^i(((Y\times U)/G)(\C),\Z)$
		is an isomorphism for $i\leq 2d$.
	\end{enumerate}
	Letting $G$ act trivially on $\Spec k_1$,
	we arrive at the cohomology groups
	$\HH^*_G(\Spec {k_1})$,
	which are denoted $\HH^*(BG)$.
	
	Since every $k_1$-variety $Y$ carrying a $G$-action 
	admits a unique $G$-equivariant morphism to $\Spec k_1$,
	we have a canonical morphism $\HH^*(BG)\to \HH^*_G(Y)$.
	The morphisms $\hat{q}_1$ and $\hat{g}_1$ from Step~1 now give rise
	to a commutative diagram
	\[
	\xymatrixcolsep{4pc}\xymatrix{
	H^i((U_1/G)(\C),\Z) \ar[r]^{H^i(q_1)} \ar@{=}[d] 
	& H^i(W_1(\C))	\ar@{=}[d] 
	& H^i(X(\C)) \ar[l]_{H^i(g_1)} \ar@{=}[d]  \\
	H^i_G(U_1) \ar[r]^{H_G^i(\hat{q}_1)} &
	H^i_G(F_1) &
	H^i_G(E_1) \ar[l]_{H_G^i(\hat{g}_1)}
	\\
	& 
	H^i(BG) \ar[lu]^{\alpha_i} \ar[u]|{\beta_i} \ar[ru]_{\gamma_i}
	&	
	}
	\]	
	
	Applying (ii) to with $Y=\Spec k_1$ and $U=U_1$ from Step~1,
	we find that the map 
	$
	\alpha_i: \HH^i(BG)\to\HH^i_G(U_1)$
	is an isomorphism for   all $i\leq 2e$.
	Since the morphism  $q_1:W_1\to U_1/G$ is a 
	$P_1$-torsor,
	the map $W_1(\C)\to (U_1/G)(\C)$ is a homotopy equivalence,
	and thus $\HH^i(\hat{q}_1):\HH^i_G(U_1)\to\HH^i_G(F_1)$ is an isomorphism for all $i\in\Z$.
	This means that
	\[
	\beta_i:\HH^i(BG)\to\HH^i_G(F_1)
	\]
	is an isomorphism for all $i\leq 2e$.

%	On the other hand,
%	by applying Lemma~\ref{LM:vashing-of-sing-coh} to $X_0\times_{k_0}{\C}$,
%	we find that there exists $n\in\N$ (depending only on $k_0$)
%	such that 
%	\[
%	\HH^i_G(E_{1})=\HH^i(X_{1}(\C))=\HH^i(X_0(\C))=0\qquad\forall i> n.
%	\]
	
\smallskip	
	
	\Step{3}
	We finally explain how to choose $e$ and achieve the desired
	contraction. By applying Lemma~\ref{LM:vashing-of-sing-coh} to $X_0\times_{k_0}{\C}$,
	we see that there exists $n\in\N$ (depending only on $X_0$)
	such that 
	\[
	\HH^i_G(E_{1})=\HH^i(X_{1}(\C))=\HH^i(X_0(\C))=0\qquad\forall i> n.
	\]
	
	Since $G $ is not unipotent,
	by \cite[Lemma~12.1]{First_2022_generators},
	there exists $i>n$ such that $\HH^i(BG)\neq 0$.
	We note that while the condition $\HH^i(BG)\neq 0$ 
	assumes  that $k_1$ was chosen, one can chose $i$
	that works for every field $k_1$ lying between $k_0$ and $\C$,
	see \cite[Remark~12.2]{First_2022_generators}.
	Take $e=\ceil{i/2}$. 
	
	Consider the diagram from Step~2.
	The map $\beta_i:\HH^i(BG)\to \HH^i_G(F_1)$ is an isomorphism,
	because   $i\leq 2e$, and it factors via $\HH^i_G(E_1)$,
	which is $0$ because $i>n$. This forces $\HH^i(BG)=0$,
	which is absurd since $\HH^i(BG) \neq 0$ by our choice of $i$.
	Having reached a contradiction, we conclude that the original group $G$ must be unipotent. 
%	
%	By Step~3, the map
%	\[
%	\HH^j(BG)\to \HH^j_G(F_1)
%	\]
%	is an isomorphism. However, since there
%	is a $G$-equivariant
%	morphism $F_1\to E_{1}$ (Step~1),
%	we see that $ \HH^j(BG)\to \HH^j_G(F_1)$ factors
%	via $\HH^j_G(E_{1})=0$. 
%	This contradicts our choice of $j$, so 
%	we conclude that the original group $G$ must be unipotent. 
\end{proof}

\begin{cor}
Let $k$ be a field of characteristic $0$ and let $G$ be a nontrivial
finite group. There is   no $G$-Galois extension $S/R$,
with $R$ a $k$-ring of finite type, such that $S/R$ specializes
to any $G$-Galois extension of    finite type $k$-rings.
\end{cor}

\begin{proof}
	View $G$ as a constant algebraic group over $\Spec k$.
	The corollary follows from Theorem~\ref{TH:converse} and the equivalence between
	$G$-torsors and $G$-Galois extensions, because $G$ is not unipotent.
\end{proof}

\begin{que}
	Let $k$ be any field
	and let $G$ be an affine algebraic group over $k$ admitting
	a   torsor
	$E \to X $ that is weakly versal for $\ftAff/k$,
	and such that $X$ is of finite type over $k$.
	Is $G$ unipotent? 
\end{que}

\section{An Application to Symbol Length}
\label{sec:symbol}

We finish this paper with an application of strongly
versal torsors to symbol length of Azumaya algebras over semilocal rings.

Hoobler \cite{Hoobler_2006_Merkurjev_Suslin_over_semilocal}
extended the Merkurjev--Suslin Theorem from fields to semilocal rings.
As a corollary \cite[Corollary~3]{Hoobler_2006_Merkurjev_Suslin_over_semilocal}, 
it follows that if $n\in\N$ and $R$ is a semilocal 
$\Z[\frac{1}{n},e^{2\pi i/n}]$-ring,
then
every    Azumaya $R$-algebra $A$ of period dividing $n$ 
is Brauer equivalent to the tensor product of \emph{symbol algebras}, i.e.,
\[A\sim_{\Br} (a_1,b_1)_{n,R}\otimes_R \cdots\otimes_R (a_\ell,b_\ell)_{n,R}.\]
Here, as usual, 
$(a ,b )_{n,R}=R\Trings{x,y\where x^n=a ,\, y^n=b,\ xy=\rho_n yx}$
where $a,b\in \units{R}$ and $\rho_n$
is the image of $e^{2\pi i/n}$ in $R$. 
The maximum possible $\ell$ required as $A$ ranges over the $n$-periodic degree-$m$
Azumaya $R$-algebras  is denoted $ \ell_R(n,m)$ and   called
the \emph{$(n,m)$-symbol length} of $R$.

When $R$ is a field $k$, it is well-known that $ \ell_k(n,m)$ is finite. Classically,
this is shown using the theory of versal (or generic) division algebras, or equivalently,
versal  torsors over appropriate algebraic groups. 
By using  the strongly versal  torsors
constructed in Section~\ref{sec:existence}, 
%%%%%%%%%%%% Change begins 
we can extend  this finiteness
result to semilocal rings in two ways.
%%%%%%%%%%%% Change ends

\begin{thm}\label{TH:symbol-length-I}
	Let $n,m\in\N$ and let $R$ be a semilocal $\Z[\frac{1}{n},e^{2\pi i/n}]$-ring.
	Then there is $L\in\N$ such that for every semilocal 
	$R$-ring $R'$ with infinite residue fields, we have
	$
	\ell_{R'}(n,m)\leq L
	$
\end{thm}

\begin{thm}\label{TH:symbol-length-II}
%%%%%%%%%%%% Change begins 
	Let $n,m,s\in\N$.
	Then there is $L\in\N$ such that for every semilocal 
	$\Z[\frac{1}{n},e^{2\pi i/n}]$-ring $R'$ 
	having at most $s$ maximal ideals, 
%%%%%%%%%%%% Change ends
	we have 
	$
	\ell_{R'}(n,m)\leq L
	$.
\end{thm}

%%%%%%%%%%%% Change begins 
Using strongly versal torsors
is key to the proof of   Theorem~\ref{TH:symbol-length-I}.
The proof of Theorem~\ref{TH:symbol-length-II} only makes use
of weakly versal torsors, but for group schemes   over $\Spec\Z[\frac{1}{n},e^{2\pi i/n}]$.
%%%%%%%%%%%% Change ends

\begin{proof}[Proof of Theorem~\ref{TH:symbol-length-I}]
	Let $\catC_1$ denote the class of affine schemes $X$ for which the subspace
	$\clpnt{X}$ 
	is noetherian of dimension $\leq 1$. Note that $\catC_1$ includes
	the spectra of all semilocal rings. 
	
	By Proposition~\ref{PR:versal-period-limited},
	there is a smooth $\Z$-algebra $R_0$ and an $n$-periodic degree-$m$ Azumaya  algebra $A_0$
	that   
	is strongly versal for $\catC_1$ (within the category of all $n$-periodic degree-$m$ Azumaya  algebras). 
	Let $R_1=R_0\otimes_\Z R$ and $A_1=A_0\otimes_\Z R$. 
	By Proposition~\ref{PR:versality-under-base-change} (see also Remark~\ref{RM:versal-objects}), the $R_1$-algebra $A_1$
	is an $n$-periodic degree-$m$ Azumaya $R_1$-algebra that 
	is strongly versal 
	for   $\catC_1/\Spec R$.
	
	Let $\frakm$ be a maximal ideal of $R$.
	By Proposition~\ref{PR:base-of-versal-torsor}(iv), $R_1/\frakm R_1$ is an integral domain,
	so $\frakm R_1\in\Spec R_1$. Thus, $M=R_1-\bigcup_{\frakm\in\Max R} \frakm R_1$ is a multiplicative
	subset of $R_1$.
	Let 
	$\tilde{R}_1=M^{-1}R$ and $\tilde{A}_1=M^{-1}A_1$.
	Since $R$ has only finitely many maximal ideals, the ring   $\tilde{R}_1$
	has finitely many maximal ideals, namely    $\{\frakm \tilde{R}_1\where \frakm\in \Max R\}$.
	Now, by Hoobler's Theorem \cite[Corollary~3]{Hoobler_2006_Merkurjev_Suslin_over_semilocal},
	there is $L\in \N$
	and $a_1,b_1,\dots,a_L,b_L\in \tilde{R}_1$
	such that 
	\[
	\tilde{A}_1\sim_{\Br} (a_1,b_1)_{n,\tilde{R}_1}\otimes_{\tilde{R}_1}\dots \otimes_{\tilde{R}_1} (a_L,b_L)_{n,\tilde{R}_1}.
	\]
	By spreading out,
	we see that there is some $u\in M$ such that $a_1,b_1,\dots,a_L,b_L$ live
	in $(R_1)_u:=\{1,u,u^2,\dots\}^{-1}R_1$ and
	\[
	(A_1)_u\sim_{\Br}  (a_1,b_1)_{n,(R_1)_u}\otimes_{(R_1)_u}\dots \otimes_{(R_1)_u} (a_L,b_L)_{n,(R_1)_u}.
	\]

	Now let $R'$ be a semilocal $R$-ring as in the theorem and
	let $A'$ be an $n$-periodic degree-$m$ Azumaya $R'$-algebra. By the strong versality of $A_1$,
	there is $s\in\N$ and an $R$-ring morphism $\vphi:R_1\to R'[t_1,\dots,t_s]$
	such that $A_1\otimes_{R_1} R'[t_1,\dots,t_s]\cong A'[t_1,\dots,t_s]$
	as $R'[t_1,\dots,t_s]$-algebras, and 
	the induced map 
	$R_1\otimes_R R'\to R'[t_1,\dots,t_s]$
	is universally injective. 
	The former means that 
	\[
	(A_1)_u\otimes_{R_1}R'[t_1,\dots,t_s]\cong A'\otimes_{R'} R'[t_1,\dots,t_s]_{\vphi(u)},
	\]
	and therefore $A'\otimes_{R'} R'[t_1,\dots,t_s]_{\vphi(u)}$ is Brauer-equivalent
	to the product of $L$ symbol algebras over $R'[t_1,\dots,t_s]_{\vphi(u)}$.
	As a result, if there exists an $R'$-ring morphism $R'[t_1,\dots,t_s]_{\vphi(u)}\to R'$,
	then $A'$ would be the product of $L$ symbol algebras over $R'$. We will show that such
	a morphism exists.
	
	We first claim that $(R_1)_u$ is faithfully flat over $R$.
	To see this, observe  that $R_1$ is smooth over $R$, and therefore so is $(R_1)_u$.
	This means that $\Spec (R_1)_u\to \Spec R$ is flat and of finite presentation, hence
	open. Since $u\notin \frakm R_1$ for every $\frakm\in \Max R$, the image of $\Spec (R_1)_u\to \Spec R$
	contains $\Max R$. As this map is also open, it follows that $\Spec (R_1)_u\to \Spec R$ is surjective,
	hence faithfully flat.
	
	Next, we claim that $\vphi(u)\notin \frakm'[t_1,\dots,t_s]$ for every $\frakm'\in \Spec R'$.
	Indeed,  since $R\to (R_1)_u$ is faithfully flat, so is $R'\to   (R_1)_u\otimes_R R'$,
	and thus  $  ((R_1)_u\otimes_R R')\frakm'\neq (R_1)_u\otimes_R R'$.
	Since $R_1\otimes_R R'\to R'[t_1,\dots,t_s]$ is universally injective,
	the same holds for $(R_1)_u\otimes_R R'\to R'[t_1,\dots,t_s]_{\vphi(u)}$.
	Therefore, the latter map remains injective after tensoring with $R'/\frakm'$ over $R'$.
	Since $  ((R_1)_u\otimes_R R')\frakm'\neq (R_1)_u\otimes_R R'$, this means
	that $R'[t_1,\dots,t_s]_{\vphi(u)}\cdot \frakm'\neq R'[t_1,\dots,t_s]_{\vphi(u)}$,
	so $\vphi(u)\notin\frakm'[t_1,\dots,t_s]$.

	To finish, put $f=\vphi(u)\in R'[t_1,\dots,t_s]$.
	Then the image of $f$ in $(R'/\frakm')[t_1,\dots,t_s]$ is nonzero
	for every $\frakm'\in\Max R'$.
	Our assumption that $R'/\frakm'$ is infinite for all $\frakm'\in\Max R'$
	and the Chinese Remainder Theorem imply that there exist
	$r_1,\dots,r_s\in R'$ such that $f(r_1,\dots,r_s)\notin \frakm'$
	for every $\frakm'\in\Max R'$, or rather, $f(r_1,\dots,r_s)\in \units{R'}$.
	The $R'$-algebra morphism $R'[t_1,\dots,t_s]\to R'$ which specializes
	each $t_i$ to $r_i$ therefore gives rise to the desired $R'$-algebra morphism
	$R'[t_1,\dots,t_s]_{\vphi(u)}\to R'$.
\end{proof}

\begin{proof}[Proof of Theorem~\ref{TH:symbol-length-II}] 
%%%%%%%%%%%% Change begins
	Put $R=\Z[\frac{1}{n},e^{2\pi i /n}]$ and then define
	$R_1$ and $A_1$   as in the proof of Theorem~\ref{TH:symbol-length-I}.
	Then $A_1$
	is an $n$-periodic degree-$m$ Azumaya $R_1$-algebra that 
	is weakly versal 
	for all semilocal $R$-rings.
	
	Let $P$ be a set of prime ideals in $R$ with $|P|\leq s$.
	As in the proof of Theorem~\ref{TH:symbol-length-I}, by localizing at
	the multiplicative set $R-\bigcup_{\frakp\in P}\frakp$,
	we see that
	there is some $u=u(P)\in R-\bigcup_{\frakp\in P}\frakp$ and $a_1,b_1,\dots,a_{\ell(P)},b_{\ell(P)}\in (R_1)_u$,
	such that
	\begin{equation}\label{EQ:local-prod-of-symbols}
	(A_1)_{u(P)}\sim_{\Br}  (a_1,b_1)_{n,(R_1)_{u(P)}}\otimes_{(R_1)_{u(P)}}\dots \otimes_{(R_1)_{u(P)}} (a_{\ell(P)},b_{\ell(P)})_{n,(R_1)_{u(P)}}.
	\end{equation}
	Put $U(P)=\{\frakp\in\Spec R_1\suchthat u(P)\notin \frakp\}$.
	By construction, the collection
	$\{U(P)^s\where P\subseteq \Spec R,\,|P|\leq s\}$
	is an open covering of $(\Spec R_1)^s$ endowed with the product topology.
	Since $\Spec R_1$ is quasi-compact, Tychonoff's theorem tells us that there
	are $P_1,\dots,P_t\subseteq \Spec R_1$ such that $(\Spec R_1)^s=\bigcup_{i=1}^t U(P_i)^s$.
	Otherwise said, for every $Q\subseteq \Spec R$ with $|Q|\leq s$, there is some $i\in\{1,\dots,t\}$
	such that $Q\subseteq U(P_i)$, or equivalently, $u_i:=u(P_i)\notin \bigcup_{\frakp\in Q}\frakp$.
	
	Now let $A'$ be an $n$-periodic degree-$m$ Azumaya over
	a semilocal $R$-ring $R'$ having at most $s$ maximal ideals.
	By the construction of $A_1$, there is a ring homomorphism
	$\vphi:R_1\to R'$ such that $A'\cong A_1\otimes_{R_1} R'$.
	Let $Q=\{\vphi^{-1}(\frakm')\where\frakm'\in\Max R'\}$.
	Then $Q$ is a subset of $\Spec R_1$ of cardinality at most $s$.
	Thus, there is $i\in\{1,\dots,t\}$
	such that $u_i\notin\bigcup_{\frakp\in Q}\frakm$. This means
	that $\vphi(u_i)\in R'-\bigcup_{\frakm'\in\Max R'}\frakm'=\units{R'}$, so $\vphi$ extends to a ring
	homomorphism $\psi:(R_1)_{u_i}\to R'$.
	Taking $P=P_i$ in \eqref{EQ:local-prod-of-symbols} and tensoring
	both sides with $R'$ now shows that the symbol length of $A'$
	is at most $\ell(P_i)$. As this holds for every choice of $A'$,
	we conclude that $\ell_{R'}(n,m)\leq  \max\{\ell(P_1),\dots,\ell(P_t)\}$. Take the right
	hand side to be the required number $L$.
\end{proof}

The following corollary
%%%%%%%%%%%% Change begins
of Theorem~\ref{TH:symbol-length-II}
%%%%%%%%%%%% Change ends
is known to experts.
It is also well-known if the characteristic of the field
$F$ is fixed in advance.

\begin{cor}
	For every $n,m\in\N$,
	there is $L\in \N$ such that $\ell_F(n,m)\leq L$
	for every field $F$ containing a primitive $n$-th root of unity.
\end{cor}

%%%%%%%%%%%% Change begins
%\begin{proof}
%	If $F$ is finite, then $\ell_F(n,m)=0$ by Wedderburn's Theorem.
%	The case where $F$ is infinite is covered by Theorem~\ref{TH:symbol-length-II}.
%\end{proof}

\begin{que}
	Does Theorem~\ref{TH:symbol-length-I} 
	continue to hold if we drop the requirement that the residue fields
	of $R'$ are infinite?
	Does Theorem~\ref{TH:symbol-length-II} continue to hold
	if we do not bound from above the number of maximal ideals?
\end{que}
%%%%%%%%%%%% Change ends

\bibliographystyle{plain}
\bibliography{MyBib_18_05}

\end{document}

%% file: amsart_header_18_05.tex
\usepackage{amsfonts}
\usepackage{amssymb}
\usepackage{amsmath}
\usepackage{hyperref}
\usepackage{mathrsfs}
\usepackage{centernot}
\usepackage{mathdots}
\usepackage{stmaryrd}
\usepackage[all]{xy}

\usepackage{mathtools}

\usepackage{color}

\newtheorem{thm}{Theorem}[section]
\newtheorem{lem}[thm]{Lemma}
\newtheorem{prp}[thm]{Proposition}
\newtheorem{cor}[thm]{Corollary}
\newtheorem{dfn}[thm]{Definition}

\newtheorem{que}[thm]{Question}

\theoremstyle{definition}

\newtheorem{example}[thm]{Example}
\newtheorem{remark}[thm]{Remark}

\theoremstyle{plain}

\newcommand{\Step}[1]{\smallskip\noindent {\it Step #1.}} % For steps in proof

\newcommand{\rem}[1]{}

\newcommand{\C}{\mathbb{C}}

\newcommand{\F}{\mathbb{F}}

\newcommand{\N}{\mathbb{N}}
\newcommand{\Q}{\mathbb{Q}}

\newcommand{\Z}{\mathbb{Z}}

\newcommand{\HH}{{\mathrm{H}}}

\newcommand{\frakm}{{\mathfrak{m}}}

\newcommand{\frakp}{{\mathfrak{p}}}
\newcommand{\frakq}{{\mathfrak{q}}}

\newcommand{\calF}{{\mathcal{F}}}

\newcommand{\calM}{{\mathcal{M}}}

\newcommand{\calO}{{\mathcal{O}}}

\newcommand{\calU}{{\mathcal{U}}}
\newcommand{\calV}{{\mathcal{V}}}

\newcommand{\bbA}{{\mathbb{A}}}

\newcommand{\bbP}{{\mathbb{P}}}

\newcommand{\catC}{{\mathscr{C}}}

\newcommand{\catS}{{\mathscr{S}}}

\newcommand{\vphi}{\varphi}

\newcommand{\embeds}{\hookrightarrow}

\newcommand{\suchthat}{\,:\,}
\newcommand{\where}{\,|\,}

\newcommand{\quo}[1]{\overline{#1}}

\newcommand{\Circs}[1]{\left( #1 \right)}

\newcommand{\Trings}[1]{\left< #1 \right>}

\newcommand{\ceil}[1]{\lceil {#1} \rceil}

\DeclareMathOperator{\Br}{Br} %
\DeclareMathOperator{\Char}{char} %
\DeclareMathOperator{\codim}{codim} %
\DeclareMathOperator{\coker}{coker} %
\DeclareMathOperator{\GL}{GL} %
\DeclareMathOperator{\Hom}{Hom} %
\DeclareMathOperator{\id}{id} %
\DeclareMathOperator{\im}{im} %

\DeclareMathOperator{\Max}{Max}
\DeclareMathOperator{\Mor}{Mor} %
\DeclareMathOperator{\rank}{rank}
\DeclareMathOperator{\Span}{span} %
\DeclareMathOperator{\Spec}{Spec} %
\newcommand{\nGL}[2]{\mathrm{GL}_{#2}({#1})}

\newcommand{\nSL}[2]{\mathrm{SL}_{#2}({#1})}

\newcommand{\nMat}[2]{\mathrm{M}_{#2}(#1)}

\newcommand{\trans}{{\mathrm{t}}}

\newcommand{\uGL}{{\mathbf{GL}}}
\newcommand{\uPGL}{{\mathbf{PGL}}}
\newcommand{\uSL}{{\mathbf{SL}}}

\newcommand{\uO}{{\mathbf{O}}}

\newcommand{\uSp}{{\mathbf{Sp}}}

\newcommand{\uU}{{\mathbf{U}}}

\newcommand{\nGm}[1]{{\mathbf{G}}_{\mathbf{m},{#1}}}
\newcommand{\nGa}[1]{{\mathbf{G}}_{\mathbf{a},{#1}}}

\newcommand{\umu}{{\boldsymbol{\mu}}}

\newcommand{\Zar}{\mathrm{Zar}}
\newcommand{\et}{\mathrm{\acute{e}t}}
\newcommand{\fppf}{\mathrm{fppf}}

\newcommand{\units}[1]{{#1^\times}}